\documentclass[11pt,reqno]{amsart}
\usepackage{amssymb,mathrsfs,graphicx,subfigure, enumerate}
\usepackage{amsmath,amsfonts,amssymb,amscd,amsthm,bbm}
\usepackage{extpfeil}
\usepackage{graphicx,colortbl}
\usepackage{float}
\usepackage{epsfig}
\usepackage{caption}
\usepackage{bm}
\graphicspath{{Figure/}}
\allowdisplaybreaks
\usepackage{comment}
\usepackage[margin=1in]{geometry}

\makeatletter
\@namedef{subjclassname@2020}{\textup{2020} Mathematics Subject Classification}
\makeatother

\title[Stability of inflow problem of the Navier-Stokes equations]{Convergence to Superposition of Boundary Layer, Rarefaction and Shock for the Inflow Problem of the 1D Navier--Stokes Equations}

\author[Han]{Sungho Han}
\address[Sungho Han]{\newline Department of Mathematical Sciences \newline Korea Advanced Institute of Science and Technology, Daejeon  34141, Republic of Korea}
\email{sungho\_han@kaist.ac.kr}

\author[Kang]{Moon-Jin Kang}
\address[Moon-Jin Kang]{\newline Department of Mathematical Sciences \newline Korea Advanced Institute of Science and Technology, Daejeon  34141, Republic of Korea}
\email{moonjinkang@kaist.ac.kr}

\author[Kim]{Jeongho Kim}
\address[Jeongho Kim]{\newline Department of Applied Mathematics, \newline Kyung Hee University, 1732 Deogyeong-daero, Giheung-gu, Yongin-si, Gyeonggi-do 17104, Republic of Korea}
\email{jeonghokim@khu.ac.kr}

\author[Kim]{Nayeon Kim}
\address[Nayeon Kim]{\newline Department of Mathematical Sciences \newline Korea Advanced Institute of Science and Technology, Daejeon  34141, Republic of Korea}
\email{ny2347@kaist.ac.kr}

\author[Oh]{HyeonSeop Oh}
\address[HyeonSeop Oh]{\newline Department of Mathematical Sciences \newline Korea Advanced Institute of Science and Technology, Daejeon  34141, Republic of Korea}
\email{ohs2509@kaist.ac.kr}

\begin{document}
	\newtheorem{theorem}{Theorem}[section]
	\newtheorem{lemma}{Lemma}[section]
	\newtheorem{corollary}{Corollary}[section]
	\newtheorem{proposition}{Proposition}[section]
	\newtheorem{remark}{Remark}[section]
	\newtheorem{definition}{Definition}[section]
	
	\renewcommand{\theequation}{\thesection.\arabic{equation}}
	\renewcommand{\thetheorem}{\thesection.\arabic{theorem}}
	\renewcommand{\thelemma}{\thesection.\arabic{lemma}}
	\newcommand{\bbr}{\mathbb R}
	\newcommand{\bbz}{\mathbb Z}
	\newcommand{\bbn}{\mathbb N}
	\newcommand{\bbs}{\mathbb S}
	\newcommand{\bbp}{\mathbb P}
	\newcommand{\bbt}{\mathbb T}
	\newcommand{\<}{\langle}
	\renewcommand{\>}{\rangle}
	\newcommand{\e}{\varepsilon}
	\newcommand{\pa}{\partial}
	\newcommand{\tU}{\widetilde{U}}
	\newcommand{\tu}{\tilde{u}}
	\newcommand{\tv}{\tilde{v}}
	\newcommand{\tw}{\tilde{w}}

	\newcommand{\rd}{\partial}
	\newcommand{\na}{\nabla}
	
	\newcommand{\infi}{\infty}
	\newcommand{\R}{\mathbb{R}}	
	\newcommand{\sX}{\dot{\bold{X}}}
	\newcommand{\bB}{\bold{B}}
	\newcommand{\bG}{\bold{G}}
	\newcommand{\bS}{\bold{S}}
	\newcommand{\bD}{\bold{D}}
	\newcommand{\bX}{\bold{X}}
	\newcommand{\bY}{\bold{Y}}
	\newcommand{\cB}{\mathcal{B}}
	\newcommand{\cG}{\mathcal{G}}
	\newcommand{\cS}{\mathcal{S}}
	\newcommand{\cD}{\mathcal{D}}
	
	
	\newcommand{\norm}[1]{\left\lVert#1\right\rVert}
	\newcommand{\beq}{\begin{equation}}
\newcommand{\eeq}{\end{equation}}

\newcommand{\eps}{\varepsilon }
	

	
	\subjclass[2020]{35Q35, 76N06} 
	
	\keywords{$a$-contraction with shift; asymptotic behavior; inflow problem; barotropic Navier-Stokes equations}
	
	\thanks{\textbf{Acknowledgment.} S. Han, M.-J. Kang, N. Kim and H. Oh were partially supported by the National Research Foundation of Korea  (RS-2024-00361663  and NRF-2019R1A5A1028324). J. Kim was supported by Samsung Science and Technology Foundation under Project Number SSTF-BA2401-01. }

	\begin{abstract} 
		We establish the asymptotic stability of solutions to the inflow problem for the one-dimensional barotropic Navier--Stokes equations in half space. When the boundary value is located at the subsonic regime, all the possible thirteen asymptotic patterns are classified in \cite{M01}. We consider the most complicated pattern, the superposition of the boundary layer solution, the 1-rarefaction wave, and the viscous 2-shock waves. In this superposition, the boundary layer is degenerate and large. We prove that, if the strengths of the rarefaction wave and shock wave are small, and if the initial data is a small perturbation of the superposition, then the solution asymptotically converges to the superposition up to a dynamical shift for the shock. 
		As a corollary, our result implies the asymptotic stability for the simpler case where the superposition consists of the degenerate boundary layer solution and the viscous 2-shock. Therefore, we complete the study of the asymptotic stability of the inflow problem for the 1D barotropic Navier--Stokes equations for subsonic boundary values. 
	\end{abstract}
	
	\maketitle
	
	\tableofcontents
	
	\section{Introduction}
	We consider the inflow problem of the one-dimensional barotropic compressible Navier-Stokes (NS) equations, where the domain of flow is the half-line $\R_+\cup\{0\}=[0,\infty)$:
	\begin{align}
	\begin{aligned}\label{eq:NS_inflow_eulerian}
	&\rho_t +(\rho u)_x=0,\quad t>0,\quad x>0,\\
	&(\rho u)_t +(\rho u^2+P(\rho))_x=\mu u_{xx},\\
	&(\rho,u)(t,0) = (\rho_-,u_-),\quad \mbox{with}\quad \rho_-,u_->0,
	\end{aligned}
	\end{align}
	subject to the initial data
	\[(\rho,u)(0,x)=(\rho_0,u_0)(x),\quad \inf_{x>0}\rho_0(x)>0\]
	and the far-field condition
	\[\lim_{x\to\infty}(\rho,u)(t,x)=(\rho_+,u_+).\]
	Here, $\rho(t,x)$ and $u(t,x)$ are the density and velocity of the fluid in Eulerian coordinates, $P(\rho)=\rho^\gamma$ is the isentropic pressure law with an adiabatic constant $\gamma>1$, and $\mu$ is the viscosity of the fluid. For simplicity, we normalize $\mu=1$. In Lagrangian mass coordinates, the inflow problem \eqref{eq:NS_inflow_eulerian} can be reformulated in terms of the specific volume $v=1/\rho$ and the velocity $u$:
	\begin{align}
	\begin{aligned}\label{eq:NS_inflow_x}
	&v_t -u_x = 0,\quad t>0,\quad x>\sigma_-t,\\
	&u_t +p(v)_x=\left(\frac{u_x}{v}\right)_{x},\\
	&(v,u)(t,\sigma_-t)=(v_-,u_-),\quad \mbox{with}\quad v_-,u_->0,
	\end{aligned}
	\end{align}
	and the far-field condition becomes
	\begin{equation}\label{eq:fftx}
    \lim_{x\to\infty}(v,u)(t,x)=(v_+,u_+),
    \end{equation}
	where $\sigma_-:=-\frac{u_-}{v_-}<0$, $p(v)=v^{-\gamma}$, and $v_\pm = 1/\rho_\pm$. As the boundary $x=\sigma_-t$ moves along time, it is convenient to introduce a moving frame $\xi = x-\sigma_-t$ to fix the boundary at $\xi=0$. Then, the system \eqref{eq:NS_inflow_x} can be written in terms of $(t,\xi)$ as
	\begin{align}
	\begin{aligned}\label{eq:NS_inflow}
	&v_t -\sigma_- v_\xi -u_\xi = 0,\quad t>0,\quad \xi>0,\\
	&u_t-\sigma_-u_\xi +p(v)_\xi=\left(\frac{u_\xi}{v}\right)_{\xi},\\
	&(v,u)(t,0)=(v_-,u_-),
	\end{aligned}
	\end{align}
	subject to the initial data and the far-field condition
	\begin{equation}\label{boundary_inflow}
	(v,u)(0,\xi)=(v_0,u_0)(\xi),\quad \lim_{\xi\to\infty}(v,u)(t,\xi) = (v_+,u_+).
	\end{equation}
	We refer to \cite{MN01} for the detailed derivation of the inflow problem \eqref{eq:NS_inflow}--\eqref{boundary_inflow} in terms of the Lagrangian coordinates.
	
	In this paper, we are interested in the asymptotic behavior of the solutions $U=(v,u)$ to the inflow problem \eqref{eq:NS_inflow}--\eqref{boundary_inflow}. For the whole space problem, that is, when the spatial domain is $\R$, the asymptotic patterns of the solution are represented as the rarefaction wave, shock wave, or their superposition. Unlike the case of the whole space problem, the inflow problem has another type of stationary solution due to the boundary, which is called the ``boundary layer (BL) solution''. Therefore, the asymptotic patterns of the inflow problem are much more complicated, and not only the rarefaction and viscous shock waves appear, but also the BL solution appears. The possible asymptotic patterns depend on the location of the boundary value $U_-=(v_-,u_-)$ on the $(v,u)$-space. Let $c(v):=\sqrt{\gamma}v^{-(\gamma-1)/2}$ be the speed of sound. Then, depending on the relationship between $u$ and $c(v)$, the $(v,u)$-space is separated into the following subsonic, transonic, and supersonic regions respectively:
	\begin{align*}
	&\Omega_{\textup{sub}}:=\left\{(v,u)~|~u<c(v),\quad v>0,\quad u>0\right\},\\
	&\Gamma_{\textup{trans}}:=\left\{(v,u)~|~u=c(v),\quad v>0,\quad u>0\right\},\\
	&\Omega_{\textup{super}}:=\left\{(v,u)~|~u>c(v),\quad v>0,\quad u>0\right\}.
	\end{align*}
	We will focus on the case when $U_-\in \Omega_{\textup{sub}}$.
	In this case, there are 13 possible asymptotic patterns composed of the combinations of the BL solution, rarefaction wave, and viscous shock wave \cite{M01} (Figure \ref{fig:classification}, left). Among them, we are interested in case (13), where the asymptotic pattern is given as a superposition of the BL solution, 1-rarefaction, and viscous 2-shock, which are introduced as follows.
	
	For a given $U_-\in \Omega_{\textup{sub}}$, consider a line segment
	\[BL(U_-):=\left\{U\in \Omega_{\textup{sub}}\cup\Gamma_{\textup{trans}}~\bigg|~\frac{u}{v}=\frac{u_-}{v_-}=-\sigma_->0\right\},\]
	and let $U_*\in BL(U_-)\cap \Gamma_{\textup{trans}}$ (See Figure \ref{fig:classification}, right). Then, it is proved in \cite{M01} that there exists a unique stationary (degenerate) BL solution $U^{BL}(\xi)=(v^{BL}(\xi),u^{BL}(\xi))$ to \eqref{eq:NS_inflow} that satisfies
	\begin{align}
	\begin{aligned}\label{eq:BL}
	&-\sigma_-v^{BL}_\xi-u^{BL}_\xi=0,\\
	&-\sigma_-u^{BL}_\xi + p(v^{BL})_\xi = \left(\frac{u^{BL}_\xi}{v^{BL}}\right)_\xi,\\
	&U^{BL}(0)=U_-,\quad \lim_{\xi\to\infty}U^{BL}(\xi)=U_*.
	\end{aligned}
	\end{align}
	
	Next, we consider the right-end state $U_+$ that can be connected from $U_*$ by a superposition of 1-rarefaction and viscous 2-shock, with the intermediate state $U^*$. In other words, we consider $U^*\in R_1(U_*)$ where $R_1(U_*)$ is an integral curve of the first characteristic field associated with $\lambda_1(v)=-\sqrt{-p'(v)}$ for the hyperbolic part of \eqref{eq:NS_inflow_x}:
	\[R_1(U_*):=\left\{(v,u)~\bigg|~u=u_*-\int_{v_*}^v\lambda_1(s)\,ds,\quad v>v_*\right\}.\]
	Then, for any $U^*\in R_1(U_*)$, the following Riemann problem of the Euler equations on the whole space
	\begin{align}
	\begin{aligned}\label{eq:Euler}
	&v_t - u_x = 0,\quad t>0,\quad x\in \R,\\
	&u_t + p(v)_x=0,\\
	&(v,u)(0,x)=\begin{cases}
	(v_*,u_*),\quad \mbox{if}\quad x<0\\
	(v^*,u^*),\quad \mbox{if}\quad x>0
	\end{cases}
	\end{aligned}
	\end{align}
	admit the solution $U^r=(v^r,u^r)$ given as the 1-rarefaction wave, where
	\begin{equation}\label{eq:rarefaction_whole}
	v^r(x/t):=\begin{cases}
	v_*,\quad &x<\lambda_1(v_*)t,\\
	\lambda_1^{-1}(x/t),\quad &\lambda_1(v_*)t\le x\le \lambda_1(v^*)t,\\
	v^*,\quad &x>\lambda_1(v^*)t,
	\end{cases}\quad z_1(v^r(x/t),u^r(x/t))=z_1(v_*,u_*).
	\end{equation}
	Here, $z_1(v,u):=u+\int^v\lambda_1(s)\,ds$ is a 1-Riemann invariant to the Euler equations \eqref{eq:Euler}.
	
	Finally, we choose $U_+\in S_2(U^*)$ where $S_2(U^*)$ is a Hugoniot locus corresponding to the second characteristic field:
	\[S_2(U^*):=\left\{(v,u)~\Big|~ u=u^*-\sigma(v-v^*),\quad v>v^*\right\},\quad \sigma = \sqrt{-\frac{p(v)-p(v^*)}{v-v^*}}.\]
	Then, there exists a unique  viscous 2-shock wave $U^S=(v^S,u^S)(\zeta)$ with $\zeta = x-\sigma t=\xi-(\sigma-\sigma_-)t$, which is a traveling wave solution to \eqref{eq:NS_inflow}. A profile of the viscous 2-shock is given by the solution to the following system of ODEs:
	\begin{align}
	\begin{aligned}\label{eq:shock}
	&-\sigma (v^S)' - (u^S)'=0,\quad '=\frac{d}{d\zeta},\\
	&-\sigma (u^S)' + p(v^S)' = \left(\frac{(u^S)'}{v^S}\right)',\\
	&\lim_{\zeta\to-\infty}(v^S,u^S)(\zeta) = (v^*,u^*),\quad \lim_{\zeta\to+\infty}(v^S,u^S)(\zeta) = (v_+,u_+).
	\end{aligned}
	\end{align}
	We refer to Figure \ref{fig:classification} for the locations of the states $U_\pm$, $U_*$ and $U^*$.

	\begin{figure}[h!]
		\includegraphics[width=0.45\textwidth]{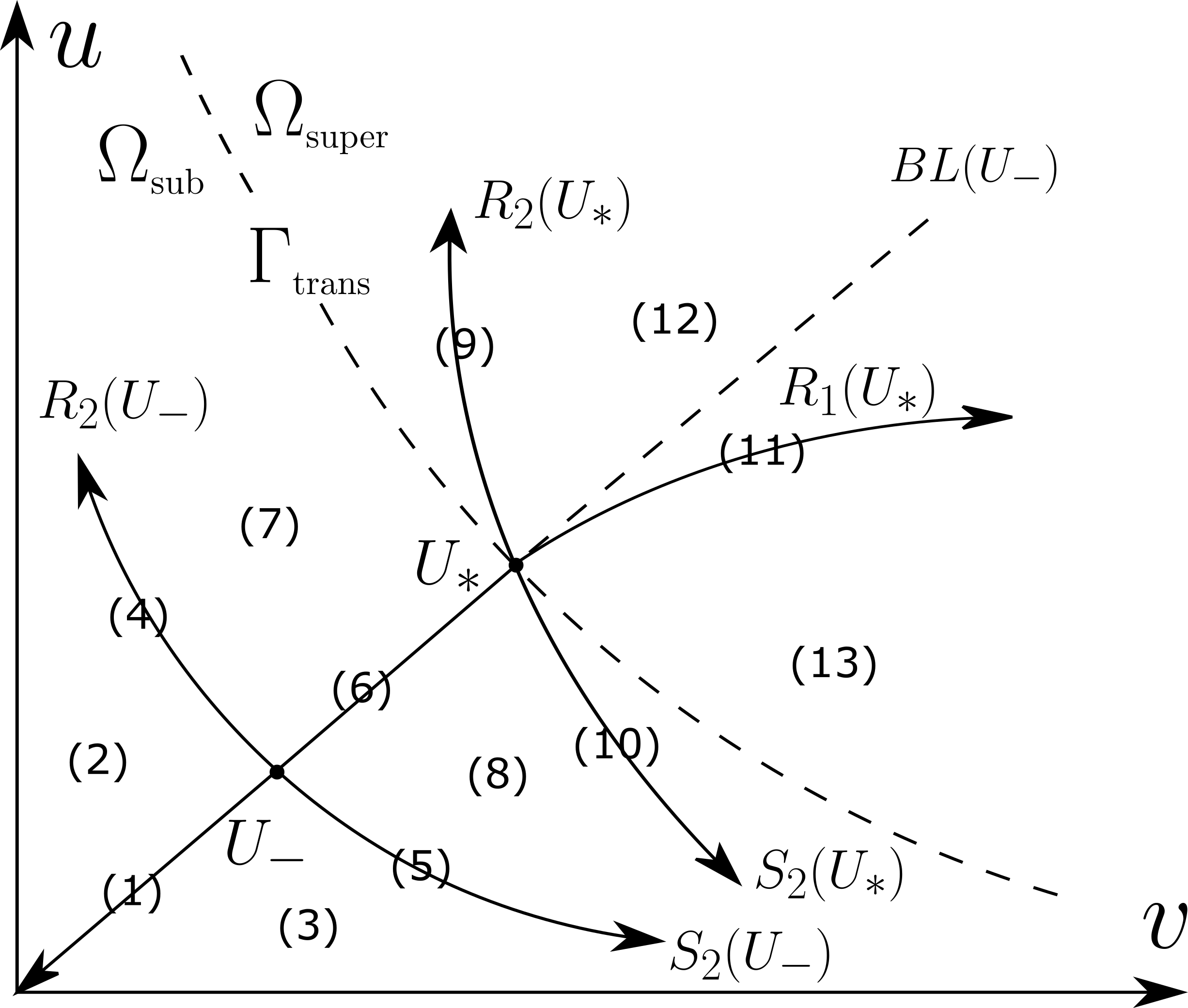}\quad 
		\includegraphics[width=0.45\textwidth]{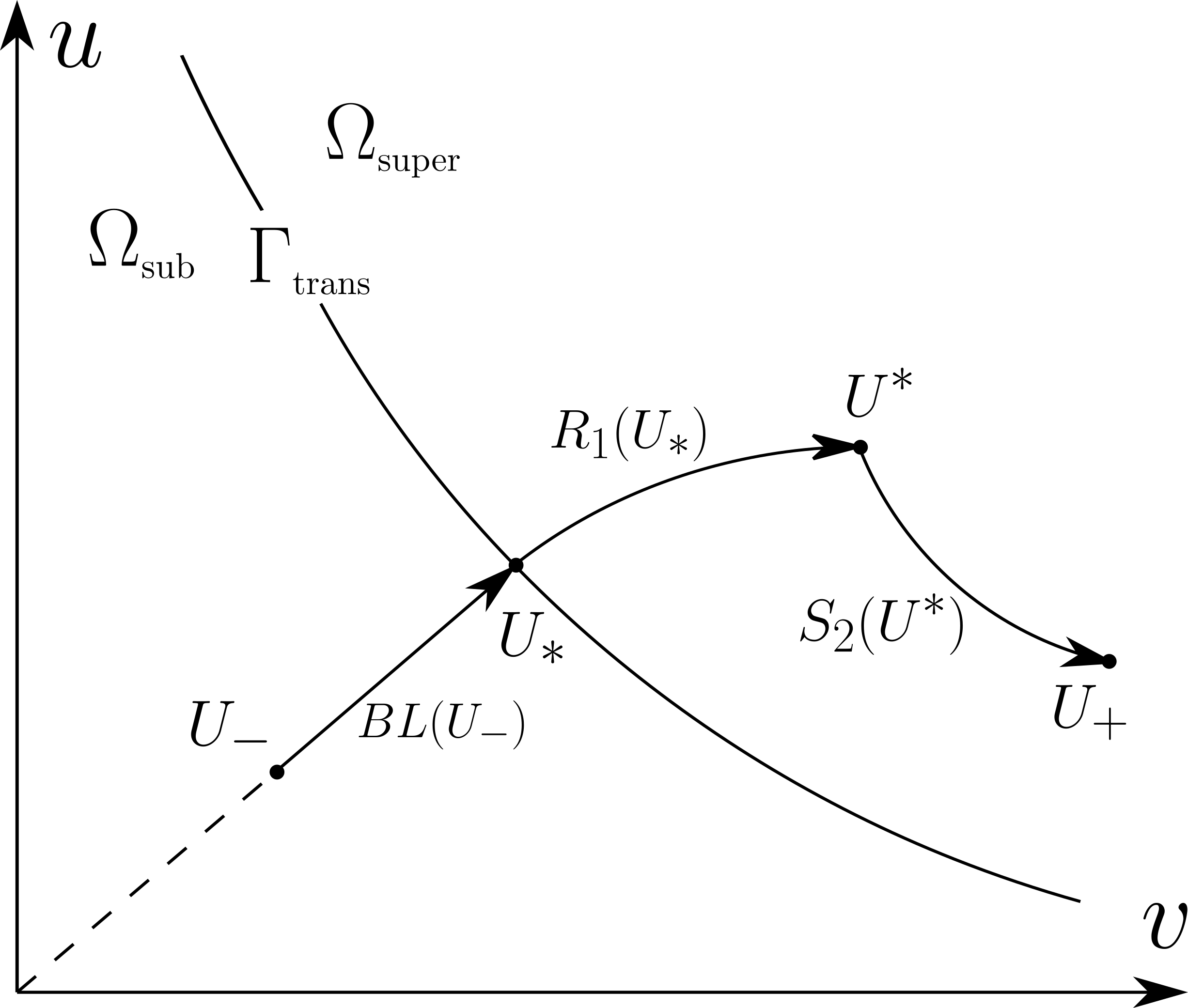}
		\caption{(Left) A complete classification of the asymptotic patterns for the inflow problem with $U_-\in \Omega_{\textup{sub}}$. The numbers of classification are the same as those in \cite{M01}. (Right) The case (13), where the three elementary waves connect $U_-$ and $U_+$. The left-end state $U_-\in \Omega_{\textup{sub}}$ is connected to $U_*\in BL(U_-)\cap\Gamma_{\textup{trans}}$ via the BL solution, $U_*$ is connected to $U^*\in R_1(U_*)$ via the 1-rarefaction wave, and $U^*$ is connected to $U_+\in S_2(U^*)$ via the viscous 2-shock wave.}
		\label{fig:classification}
	\end{figure}
	
	\subsection{Main result}
	We now present the main result of the paper. As we explained, we are considering the case when $U_-$ and $U_+$ are connected via the superposition of the BL solution, 1-rarefaction wave, and viscous 2-shock. The goal of this paper is to prove the asymptotic stability of the superposition of these elementary waves. The precise statement of the main theorem is as follows.
	
	\begin{theorem}\label{thm:main}
	For any $U_-\in \Omega_{\textup{sub}}$, let $U_*$ be the (unique) constant state which satisfies $U_*\in BL(U_-)\cap \Gamma_{\textup{trans}}$. Then, there exist positive constants $\delta_0$, $\e_0$ small enough such that 
		for any two constant states $U^*, U_+$ with $U^*\in R_1(U_*)$ and $U_+\in S_2(U^*)$ satisfying
	\[|U_*-U^*|+|U^*-U_+|<\delta_0,\] 
	 there exists $\beta>0$ large enough (depending only on the shock strength $|U^*-U_+|$) such that the following holds.
	If any initial data $U_0=(v_0,u_0)$ satisfy
 \begin{equation} \label{initial condition}
 \begin{aligned}
     \|U_0 - U^* - ( U^{BL}  - U_*)\|_{L^2(0,\beta)} + \|U_0-U_+ - (U^{BL} - U_*)\|_{L^2(\beta,\infty)} +\|\pa_x (U_0-U^{BL})\|_{L^2(\R_+)} <\e_0,
    \end{aligned}
    \end{equation}
	then the inflow problem \eqref{eq:NS_inflow_x}-\eqref{eq:fftx} subject to the initial data $U_0$ admits a unique global-in-time solution $U(t,x)=(v,u)(t,x)$. Moreover, there exists a Lipschitz continuous shift function $t\mapsto X(t)$ such that 
	\begin{align*}
	&v(t,x)-\left(v^{BL}(x-\sigma_-t)+\tilde{v}^R(t,x)+v^S(x-\sigma t-X(t)-\beta)-v_*-v^*\right)\in C([0,\infty);H^1(\sigma_-t, \infty)),\\
	&u(t,x)-\left(u^{BL}(x-\sigma_-t)+\tilde{u}^R(t,x)+u^S(x - \sigma t-X(t)-\beta)-u_*-u^*\right)\in C([0,\infty);H^1(\sigma_- t,\infty)),\\
	&u_{x x}(t,x)-\left(u^{BL}(x-\sigma_-t)+\tilde{u}^R(t,x)+u^S(x-\sigma t-X(t)-\beta)\right)_{x x}
    \in L^2(0,\infty;L^2(\sigma_- t, \infty)),
	\end{align*}
    where $(\tilde{v}^R,\tilde{u}^R)(t,x)$ denotes the smooth approximation of the 1-rarefaction wave defined in \eqref{eq:tvrtur}.
    
	Finally, the solution asymptotically converges to the superposition of waves:
	\begin{align*}
	&\sup_{x \in (\sigma_- t, \infty)}\bigg|(v,u)(t,x)-\bigg(v^{BL}(x-\sigma_-t)+v^r\left(\frac{x}{t}\right)+v^S(x - \sigma t-X(t)-\beta)-v_*-v^*,\\
	&\hspace{4.5cm}u^{BL}(x-\sigma_-t)+u^r\left(\frac{x}{t}\right)+u^S(x - \sigma t-X(t)-\beta)-u_*-u^*\bigg)\bigg|\to 0
	\end{align*}
    as $t \to \infty$, and the speed of the shift tends to zero
	\[\lim_{t\to\infty}|\dot{X}(t)|=0.\]
	\end{theorem}

\begin{remark}
\begin{enumerate}
    \item Theorem \ref{thm:main} requires the strength of the rarefaction wave and the viscous shock to be small and bounded by the constant $\delta_0$, but we do not impose any smallness condition for the boundary layer solution. 
    The small constant $\delta_0$ does not depend on  $\eps_0$ the smallness of the initial perturbation.
    \item 
   As in \cite{HMS03,HKKL_pre2}, 
  we consider shock waves $(v^S,u^S)(x-\sigma t-\beta)$ sufficiently far from the boundary by the constant $\beta$, where $\beta$ depends only on the shock strength (but independent of $\eps_0$).  
This choice plays a crucial role in minimizing the discrepancy between the boundary datum and the shock wave value at the origin.      
    \item For the initial condition \eqref{initial condition}, since we are considering the weak shock far from the origin that connects $U^*$ to $U_+$, we measure the $L^2$-norm of the initial perturbation relative to $\beta$. In addition, since we are considering  the small initial perturbation but the large boundary layer approaching $U_*$ at the far-field, the boundary layer should be added to all norms in \eqref{initial condition}.
   \item  Theorem \ref{thm:main} implies the long-time behavior towards the superposition of boundary layer, rarefaction and viscous shock, where the viscous shock is shifted by a dynamical shift $X(t)$.
	 Indeed, the decay estimate $\displaystyle{\lim_{t\to\infty}|\dot{X}(t)|=0}$ implies
		$$
		\lim_{t\rightarrow+\infty}\frac{X(t)}{t}=0,
		$$
		which means that the shift function $X(t)$ grows at most sub-linearly as $t\to\infty$. Thus, the dynamically shifted wave $(v^S,u^S)(x-\sigma t-X(t)-\beta)$ time-asymptotically converges to  the family of shock waves, $(v^S,u^S)(x-\sigma t-c)$ for $c\in\bbr$, although we do not know the final location of the asymptotic profile.
\end{enumerate}
\end{remark}

\subsection{Related results}
	We review several previous results related to the asymptotic stability of the initial boundary value problems of the NS system. The study on the nonlinear stability for the boundary value problems of the one-dimensional barotropic compressible NS equations was initiated by Matsumura and his collaborators. In \cite{MM99} and \cite{MN99}, the impermeable wall problems (that is, $u_-=0$) were treated, where they proved the stability of the outgoing viscous shock and rarefaction waves, respectively. For the case of the impermeable wall problem, the asymptotic behaviors of solutions are completely classified into these two cases, and therefore, the stabilities for all the cases are proved. A recent paper \cite{HKKL_pre2} revisited the stability of viscous shock wave for the impermeable wall problem by using the method of $a$-contraction with shifts, eliminating the initial condition for an anti-derivative variable. 
    
    On the other hand, the long-time behaviors of the inflow and outflow problems are much more complicated. In \cite{M01}, Matsumura completely classified all the asymptotic patterns of the solutions to the inflow and outflow problems according to the left-end state (boundary value) and the right-end state (far-field state). As the inflow problem is our main concern, we only review the results on it. The studies for the asymptotic behavior of the inflow problem can be further split into two cases: (i) when $U_-\in \Omega_{\textup{sub}}\cup\Gamma_{\textup{trans}}$ or (ii) when $U_-\in \Omega_{\textup{super}}$. In the former case, there are thirteen possible asymptotic patterns, which are combinations of the BL solution, rarefaction wave, and the viscous shock waves (See the left of Figure \ref{fig:classification}). 
	
	In \cite{MN01}, the stability of the BL solution (case (1),(6)), the superposition of the BL solution and the 2-rarefaction wave (case (2),(4),(7)), and the superposition of the BL solution with two rarefaction waves (case (9),(11),(12)) are proved. These results were improved in \cite{FLWZ14}, where they showed the stability of BL solution with large initial perturbations. On the other hand, the cases involving viscous 2-shocks (case (3),(5),(8)) were treated in \cite{HMS03}. Therefore, as far as we know, the following two cases remain open:\\
	
	\begin{enumerate}[(I)]
		\item The superposition of degenerate BL solution and the viscous 2-shock (case (10)).\\
		
		\item The superposition of degenerate BL solution, 1-rarefaction, and viscous 2-shock (case (13)).\\
	\end{enumerate}

	Observe that (I) is a special case of (II). Therefore, the essentially remaining case is the case (13), the superposition of all the possible elementary waves (see the right of Figure \ref{fig:classification}), which is the main object of the present paper. On the other hand, the cases of $U_-\in \Omega_{\textup{super}}$ remain mostly open. The only related work is \cite{S03} in which the stability of the superposition of 1- and 2-rarefaction was proved, and then it was improved by \cite{FLWZ14}, allowing the large perturbation.\\
 	

	\subsection{Ideas of proof}
	The main strategy for the proof is based on the method of $a$-contraction with shifts (in short, the $a$-contraction method), which provides a key energy estimate for the stability of perturbations around a shock measured by the relative entropy (see for example \cite{Kang19,Kang-V-1,KV16,KV21,KV-Inven}). Since the standard $L^2$-energy method was used for the stability of boundary layers and rarefaction waves as in \cite{MN86,MN01}, the $a$-contraction method can serve as a unified approach for analyzing the stability of superpositions consisting of a boundary layer, a rarefaction and a shock at the same framework generated by the relative entropy.
	Indeed, in \cite{KVW23}, they use the $a$-contraction method to prove the time-asymptotic stability of the superposition of the rarefaction and shock for the NS system  (see also \cite{KVW-NSF} for the NSF system). 
	 In the proof, the $a$-contraction method will be used to obtain the energy estimates for the superposition of the boundary layer, the viscous shock wave, and the approximate rarefaction, where the approximate rarefaction is a smooth approximation of the (inviscid) rarefaction.
	Then, the major difficulties we will face in the energy estimates are applying the Poincar\'e-type inequality and controlling the magnitudes of the boundary values and of the interaction terms between the three different waves (or layers). Notice that the wave interaction is due to the nonlinearity of the NS system.
	
	First, the leading order terms in \eqref{leading_order} are localized by a derivative of weight and viscous shock. Those terms are controlled in Lemma \ref{lem:leading} by using the Poincar\'e-type inequality of Lemma \ref{lem:Poincare} with the change of variable $x\mapsto y$, where the $y$ variable is defined by the viscous shock in \eqref{ydef}. However, the image $[y_0,1]$ of the $y$ variable depends on time, contrary to the whole space $\bbr$ as in \cite{KV21,KVW23}, since  the viscous shock connects the left state $U^*$ at $x=-\infty$ to the right state $U_+$ at $x=+\infty$ whereas the solution stays on the half-space $\bbr_+$. 
	To overcome it, we use the smallness of $|X(t)|$ to get $0<y_0\le  Ce^{-C\delta_S\beta}$ where $\delta_S$ denotes the shock strength, by which $y_0 \approx 0$ for $\beta$ large enough.  
	
	For the boundary terms arising from the integration-by-parts, since the approximate rarefaction is defined such that it achieves the exact constant state $U_*$ at $\xi=0$, the boundary terms are substantially related to the discrepancy between the left-end state $U^*$ and the value at $\xi=0$ of viscous shock. In addition, since the speed of dynamical shift is small enough as the $L^\infty$-norm of perturbation, we make the boundary terms small enough by choosing $\beta$ sufficiently large.

	The wave interaction terms are the last term of $J^{\textup{bad}}$ in Lemma \ref{lem:rel} for the zeroth-order estimate, and the last term $\mathcal{I}_6$ of \eqref{est_rel_ent_H1} for the first-order estimate.
	To control the interaction term at the zeroth-order estimate, we decompose it into the three combinations $R_1, R_2, R_3$ as in \eqref{SI1 decomposition}. To control them, we first obtain the  time-decay estimates for wave interactions as in Lemma \ref{lem:wave-interaction}. Note that, in Lemma \ref{lem:wave-interaction}, the first two upper bounds for the interactions with the boundary layer are not time-integrable. However, since the terms $R_1, R_2$ are the integrals of the wave interaction multiplied by the perturbation, the terms $R_1, R_2$ can be controlled by the diffusion $D_{u_1}$ and the time-integrable function as in \eqref{r1}-\eqref{r2}.
	For the first-order estimate, it is easier to use the effective velocity $h$ to control the interaction estimates. As in the estimate \eqref{i62}, we handle the interaction term $\mathcal{I}_{62}$. In order to show that the wave interaction $(R_1+R_2+R_3)^2$ is integrable in space and time, and its norm is small and controlled by the wave strength, we need the delicate estimate as in Lemma \ref{L2 wave interaction}.  \\

	The rest of the paper is organized as follows. In Section \ref{sec:prelim}, we gather several important properties of the elementary waves, as well as some technical inequalities on the relative quantities, which will be used throughout the paper. Section \ref{sec:3} presents an a priori estimate on the $H^1$-perturbation, from which we deduce the proof of Theorem \ref{thm:main}. In Section \ref{sec:4} and Section \ref{sec:high-order}, we present the detailed proof of a priori estimates. Finally, Appendix \ref{app:A} and Appendix \ref{app:B} contain proofs of the global existence of the solution and time-asymptotic behavior, respectively.
	
	\section{Preliminaries}\label{sec:prelim}
	\setcounter{equation}{0}
	In this section, we provide several useful properties of the elementary waves, such as the decay estimates and the lower and upper bounds on the relative quantities.
	
	\subsection{Boundary layer solution}
	For a given $U_-=(v_-,u_-)\in\Omega_\textup{sub}$ and $U_*=(v_*,u_*)\in BL(U_-)$, the existence of a boundary layer solution $U^{BL}(\xi)=(v^{BL}(\xi),u^{BL}(\xi))$, that is, the solution to \eqref{eq:BL}, was shown in \cite[Lemma 1.1]{MN01}. In particular, the following lemma presents the decay estimate of the degenerate BL solution.

	\begin{lemma}\cite{MN01}\label{lem:boundary_layer}
		Suppose $U_*\in BL(U_-)\cap \Gamma_{\textup{trans}}$ and let $\delta_{BL}:=|(v_--v_*,u_--u_*)|$. Then there exists a boundary layer solution $(v^{BL}(\xi),u^{BL}(\xi))$ to \eqref{eq:BL} that satisfies
		\begin{align}
		\begin{aligned}\label{BL_properties}
		&v^{BL}_\xi,\,u^{BL}_\xi>0,\\
		&|(v^{BL}(\xi),u^{BL}(\xi))-(v_*,u_*)|\le \frac{C\delta_{BL}}{1+\delta_{BL}\xi},\\
		&|(v^{BL}_{\xi},u^{BL}_{\xi})|\le \frac{C\delta^2_{BL}}{(1+\delta_{BL}\xi)^2},\\
		&|(v^{BL}_{\xi\xi},u^{BL}_{\xi\xi})|\le \frac{C\delta_{BL}}{1+\delta_{BL}\xi}|(v^{BL}_\xi,u^{BL}_\xi)|,\quad |(v^{BL}_{\xi\xi},u^{BL}_{\xi\xi})|\le \frac{C\delta^3_{BL}}{(1+\delta_{BL}\xi)^3}.
		\end{aligned}
		\end{align}
	\end{lemma}
	\begin{proof}
	Since the first and second estimates are already proved in \cite{MN01}, we focus on proving the third and fourth estimates. By integrating \eqref{eq:BL} over $[\xi,\infty)$, one can observe that $(v^{BL},u^{BL})$ satisfies
	\begin{align*}
	&\sigma_-(v^{BL}-v_*)+(u^{BL}-u_*)=0,\\
	&\frac{u_{\xi}^{BL}}{v^{BL}}=-\sigma_-(u^{BL}-u_*)+p(v^{BL})-p(v_*),
	\end{align*}
	with 
	\[\sigma_- = -\frac{u^{BL}(\xi)}{v^{BL}(\xi)} = -\frac{u_-}{v_-} = -\frac{u_*}{v_*}.\]
	Since $(v_*,u_*)\in \Gamma_{\textup{trans}}$, we have
	\[u_* = c(v_*)=v_*\sqrt{-p'(v_*)}\quad\mbox{therefore}\quad \sigma_-^2 = -p'(v_*).\]
	This gives
	\begin{align*}
	|p(v^{BL})-p(v_*)-\sigma_-(u^{BL}-u_*)| &= |p(v^{BL})-p(v_*)+\sigma_-^2(v^{BL}-v_*)|\\
	&=|p(v^{BL})-p(v_*)-p'(v_*)(v^{BL}-v_*)|\le C|v^{BL}-v_*|^2.
	\end{align*}
	Therefore, using the second estimate, we conclude that
	\[|(v_{\xi}^{BL},u_{\xi}^{BL})|\le \frac{C\delta_{BL}^2}{(1+\delta_{BL}\xi)^2}.\]
	Similarly, we use
	\begin{align*}
	u^{BL}_{\xi\xi} &= \left(v^{BL}\left(p(v^{BL})-p(v_*)-p'(v_*)(v^{BL}-v_*)\right)\right)_\xi\\
	&=v^{BL}_\xi\left(p(v^{BL})-p(v_*)-p'(v_*)(v^{BL}-v_*)\right)+v^{BL}\left((p'(v^{BL})-p'(v_*))v^{BL}_\xi\right)
	\end{align*}
	to observe that
	\[|(v^{BL}_{\xi\xi},u^{BL}_{\xi\xi})|\le \frac{C\delta_{BL}}{1+\delta_{BL}\xi}|(v^{BL}_\xi,u^{BL}_\xi)|, \quad |(v^{BL}_{\xi\xi},u^{BL}_{\xi\xi})|\le \frac{C\delta^3_{BL}}{(1+\delta_{BL}\xi)^3}.\]
	\end{proof}
	\subsection{Rarefaction wave}
	Next, we consider a smooth approximation of the 1-rarefaction wave connecting $U_*=(v_*,u_*)\in\Gamma_{\textup{trans}}$ and $U^*=(v^*,u^*)\in R_1(U_*)$, as in \cite{MN01}. To this end, let $w(t,x)=w(t,x;w_-,w_+)$ be a unique global smooth solution to the Burgers' equation
	\[w_t + ww_x=0,\quad t>0,\quad x\in \R,\]
	subject to the smooth initial data
	\[w(0,x):=\frac{w_++w_-}{2}+\frac{w_+-w_-}{2}\tanh x.\]
	Then, we define a smooth approximation of the rarefaction wave $(\tv^R,\tu^R)$ on the whole space $\R$ as
	\begin{align}
	\begin{aligned}\label{eq:tvrtur}
		&\tv^R(t,x) = \lambda_1^{-1}(w(t,x;2\lambda_1(v_*)-\lambda_1(v^*),\lambda_1(v^*))),\\
		&\tu^R(t,x) = u_*-\int_{v_*}^{\tv^R(t,x)}\lambda_1(s)\,ds,
	\end{aligned}
	\end{align}	
	where $\lambda_1(v)=-\sqrt{-p'(v)}$ is the first characteristic speed. Then, we define a smooth approximation of the rarefaction wave on the half-space $\R_+$ as a restriction of $(\tv^R,\tu^R)$:
	\[
    U^R(t,\xi)=(v^R,u^R)(t,\xi) := (\tv^R,\tu^R)(t,\xi+\sigma_-t).
    \]
	Then, $U^R$ satisfies the following properties.
	\begin{lemma}\cite{MN01}\label{lem:rarefaction}
		Let $\delta_R:=|(v_*-v^*,u_*-u^*)|$. Then, the smooth approximation $(v^R,u^R)$ of the rarefaction wave satisfies
		\begin{align}
			\begin{aligned}\label{eq:rarefaction}
				&v^R_t -\sigma_-v^R_\xi-u^R_\xi=0,\quad \xi>0,\quad t>0,\\
				&u^R_t - \sigma_-u^R_\xi+p(v^R)_\xi=0
			\end{aligned}
		\end{align}
		with
		\[(v^R,u^R)(t,0)=(v_*,u_*),\quad (v^R,u^R)(t,\infty)=(v^*,u^*).\]
		Moreover, it satisfies the following properties: for any $p \in [1,\infty]$,
		\begin{align}
		\begin{aligned}\label{rarefaction_properties}
		&v^R_\xi, u^R_\xi>0,\quad v^R_\xi\le Cu^R_\xi\le C\delta_R,\\
		&\|(v^R_\xi,u^R_\xi)\|_{L^p(\R_+)}\le \frac{C\delta_R^{1/p}}{(1+t)^{1-1/p}},\quad \|(v^R_{\xi\xi},u^R_{\xi\xi})\|_{L^p(\R_+)}\le C\min\left\{\delta_R,\frac{1}{1+t}\right\},\\
		&|(v^R-v^*,u^R-u^*)(t,\xi)|+|(v^R_\xi,u^R_\xi)(t,\xi)|\le C\delta_R\exp[-C(|\xi+\sigma_-t|+t)],\quad \xi\ge -\sigma_-t.
		\end{aligned}
		\end{align}
		Finally, we have
		\[\lim_{t\to\infty}\sup_{\xi\in\R_+}\left|(v^R,u^R)(t,\xi)-(v^r,u^r)\left(\frac{\xi+\sigma_-t}{t}\right)\right|=0\]
		where $(v^r,u^r)$ is the rarefaction wave generated by $U_*$ and $U^*$ on the whole space given in \eqref{eq:rarefaction_whole}.
	\end{lemma}
	
	\begin{remark}
		At first glance, the choice of $(\tv^R,\tu^R)$ in \eqref{eq:tvrtur} is not natural since the usual choice of approximation of rarefaction wave $v^r$ is
		\[\lambda^{-1}(w(t,x;\lambda_1(v_*),\lambda_1(v^*))).\]
		However, as pointed out in \cite{MN01}, the choice of $\tv^R$ in \eqref{eq:tvrtur} gives us the correct boundary condition $(v^R,u^R)(t,0)=(v_*,u_*)$.
	\end{remark}

	\subsection{Viscous shock wave}
	Finally, we consider the far-field state $U_+=(v_+,u_+)\in S_2(U^*)$ satisfying the following Rankine-Hugoniot conditions
	\begin{align*}
		&-\sigma (v_+-v^*)-(u_+-u^*)=0,\\
		&-\sigma (u_+-u^*)+(p(v_+)-p(v^*))=0,
	\end{align*}
	and the entropy conditions $v^*<v_+$ and $u^*>u_+$. Then, it is well-known that there exists a viscous shock profile $U^S=(v^S,u^S)(\zeta)$ with $\zeta = x-\sigma t=\xi -(\sigma-\sigma_-)t$ satisfying \eqref{eq:shock}.
	Note that we consider the viscous 2-shock wave defined on the whole space $\R$.
	\begin{lemma}\cite{MN85}\label{lem:viscous_shock}
		Let $\delta_S:=|(v^*-v_+,u^*-u_+)|$. Then, there exists a viscous shock profile $(v^S,u^S)$ satisfying the following estimates:
	\begin{align}
	\begin{aligned}\label{shock-properties}
		& v^S_\zeta>0,\quad u^S_\zeta<0,\\
		& |v^S(\zeta)-v^*|,|u^S(\zeta)-u^*|\le C\delta_S e^{-C\delta_S|\zeta|},\quad \zeta<0,\\
		& |v^S(\zeta)-v_+|,|u^S(\zeta)-u_+|\le C\delta_S e^{-C\delta_S|\zeta|},\quad \zeta>0,\\
		&|(v^S_\zeta,u^S_\zeta)|\le C\delta_S^2e^{-C\delta_S|\zeta|},\quad |(v^S_{\zeta\zeta},u^S_{\zeta\zeta})|\le C\delta_S|(v^S_\zeta,u^S_\zeta)|,\quad \zeta \in \R.
	\end{aligned}
	\end{align}
	\end{lemma}

	\subsection{Estimates on the relative quantities}
	We now present several estimates on the relative quantities. For a differentiable function $F:\R\to \R$, we define its relative quantity as
	\[F(v|w):=F(v)-F(w)-F'(w)(v-w).\]
	In particular, when $F$ is convex, then the relative quantity is always positive. In the following lemma, we gather some important estimates on the relative quantities for the pressure $p(v)=v^{-\gamma}$ and the internal energy $Q(v)=v^{1-\gamma}/(\gamma-1)$.
	\begin{lemma}\cite{KV21}\label{lem:relative_quantity}
		Let $\gamma>1$ and $v^*$ be given constants. Then, there exist constants $C$ and $\delta_*>0$ such that the following estimates hold.
		\begin{enumerate}
			\item For any $v,w$ such that $0<w<2v^*$ and $0<v<3v^*$,
			\[|v-w|^2\le CQ(v|w),\quad |v-w|^2\le Cp(v|w).\]
			\item For any $v,w>v^*/2$, $|p(v)-p(w)|\le C|v-w|$.
			\item For any $0<\delta<\delta_*$ and $v,w>0$ satisfying $|v-w|<\delta$, 
			\begin{align}
			\begin{aligned}\label{relative-est}
			&p(v|w)\le\left(\frac{\gamma+1}{2\gamma}\frac{1}{p(w)}+C|p(v)-p(w)|\right)|p(v)-p(w)|^2,\\
			&Q(v|w)\le \left(\frac{p(w)^{-\frac{1}{\gamma}-1}}{2\gamma}+C|p(v)-p(w)|\right)|p(v)-p(w)|^2,\\
			&Q(v|w)\ge\frac{p(w)^{-\frac{1}{\gamma}-1}}{2\gamma}|p(v)-p(w)|^2-\frac{1+\gamma}{3\gamma^2}p(w)^{-\frac{1}{\gamma}-2}|p(v)-p(w)|^3.
			\end{aligned}
			\end{align}
		\end{enumerate}
	\end{lemma}
	
	\subsection{Poincar\'e type inequality} One of the main ingredients for attaining the zeroth-order estimate in Lemma \ref{lem:main} is the Poincar\'e type inequality. Here, the optimal constant $\frac{1}{2}$ is independent of the size of the domain.
	\begin{lemma}\label{lem:Poincare}
		\cite{HKKL_pre2} For any $c<d$ and function $f:[c,d] \to \mathbb{R}$ satisfying $\int_c^d (y-c)(d-y)|f'(y)|^2\,dy < \infty$,
		\[
		\int_c^d \left|f(y) - \frac{1}{d-c}\int_c^d f(y)\,dy \right|^2\,dy \leq \frac{1}{2}\int_c^d (y-c)(d-y)|f'(y)|^2\,dy.
		\]
	\end{lemma}

	\section{A priori estimate and proof of Theorem \ref{thm:main}}\label{sec:3}
	\setcounter{equation}{0} 
    In this section, we state a priori estimates for the $H^1$-perturbation between the solution and the superposition wave, which constitute the key step in proving Theorem  \ref{thm:main}. The superposition wave $\overline{U}=(\overline{v},\overline{u})$ is constructed as a superposition of the boundary layer solution, the approximate rarefaction wave, and the viscous shock wave shifted by dynamical shift function $X(t)$, which will be defined in \eqref{ODE_X}, and constant shift $\beta$:
    \begin{equation} \label{eq: composite wave}
    \overline{U}(t,\xi)=\begin{pmatrix}
	\overline{v}(t,\xi)\\
	\overline{u}(t,\xi)
	\end{pmatrix}:=\begin{pmatrix}
		v^{BL}(\xi)+v^R(t,\xi)+v^S(\xi-(\sigma-\sigma_-)t-X(t)-\beta)-v_*-v^*\\
	 	u^{BL}(\xi)+u^R(t,\xi)+u^S(\xi-(\sigma-\sigma_-)t-X(t)-\beta)-u_*-u^*
	 \end{pmatrix}.
     \end{equation}
	 For simplicity, we omit the arguments if there is no confusion:
	 \[
	 v^{BL} = v^{BL}(\xi), \quad v^R = v^R(t,\xi), \quad v^S = v^S(\xi - (\sigma - \sigma_-)t - X(t) - \beta).
	 \]
	\subsection{Local existence} We begin by stating the local existence of solutions. 
	
	\begin{proposition}\label{prop:local}
		For any constant $\beta>0$, let $(\hat{v},\hat{u})$ be smooth monotone functions such that
		\[(\hat{v}(x),\hat{u}(x) ) = (v_{+},u_{+}) ,\quad \mbox{for}\quad  x\ge \beta, \quad \hat{v}(0)>0.\]
		Then, for any constants $M_0$, $M_1$, $\underline{\kappa}_0$, $\overline{\kappa}_0$, $\underline{\kappa}_{1}$, and $\overline{\kappa}_1$ with $0<M_0<M_1$ and $0<\underline{\kappa}_1<\underline{\kappa}_0<\overline{\kappa}_0<\overline{\kappa}_1$, 
		there exists a finite time $T_0>0$ such that if the initial data ($v_0,u_0)$ satisfy
        \[\|(v_0-\underline{v},u_0-\underline{u})\|_{H^1(\R_+)} \le M_0 \quad\mbox{and}\quad 0<\underline{\kappa}_0\le v_0(x)\le\overline{\kappa}_0,\quad\forall x\in\R_+,\]
	where $(\underline{v},\underline{u})$ are smooth functions defined as
    \begin{equation} \label{eq: underline v, u}
        (\underline{v},\underline{u})=(v^{BL},u^{BL})+(\hat{v},\hat{u})-(v_*,u_*).
    \end{equation}
	Then the inflow problem \eqref{eq:NS_inflow}-\eqref{boundary_inflow} has a unique solution $(v,u)$ on $[0,T_0]$ satisfying
    \begin{align*}
		&v- \underline{v} \in C([0,T_0];H^1(\mathbb{R_+})),\\
		& u- \underline{u} \in C([0,T_0];H^1(\mathbb{R_+})) \cap L^2([0,T_0];H^2(\R_+)),
		\end{align*}
	    and
        \[\|(v,u)-(\underline{v},\underline{u})\|_{L^\infty([0,T_0];H^1(\R_+))} \le M_1.\]
		Moreover, it satisfies
		\[\underline{\kappa}_1\le v(t,x)\le \overline{\kappa}_1, \quad (v,u)(t,0)=(v_-,u_-), \quad \forall(t,x)\in [0,T_0]\times \R_+.\]
	\end{proposition}

	\begin{proof} Since the proof follows from the standard iteration argument, we omit it.
	\end{proof}
    
The superposition $(\overline{v}, \overline{u})$ can be represented as the following decomposition:  
\[
(\overline{v}, \overline{u}) = (v^{BL}, u^{BL}) + (v^R + v^S - v^*, u^R + u^S - u^*) - (v_*, u_*).
\]  
Later, we will choose $(\hat{v}(x), \hat{u}(x))$ which approximates the superposition $(v^R + v^S - v^*, u^R + u^S - u^*)$ at $t = 0$. Consequently, the smooth functions $(\underline{v},\underline{u})$, defined in \eqref{eq: underline v, u}, are used to approximate the superposition $(\overline{v},\overline{u})$. We refer to \eqref{ubarvbar}--\eqref{est-init} of Appendix \ref{app:A} for the details.

    \subsection{Construction of shift} As previously mentioned, the viscous shock must be shifted to obtain stability. Here, we explicitly introduce the shift function $X:\bbr_+\to\bbr$ as a solution to the following ODE:
	\begin{align}
	\begin{aligned}\label{ODE_X}
	\dot{X}(t)&=-\frac{M}{\delta_S}\Bigg(\int_{\R_+} a u^S_\xi(\xi-(\sigma-\sigma_-)t-X(t)-\beta) (u-\overline{u}) \,d \xi\\
	&\hspace{1.5cm}+\frac{1}{\sigma}\int_{\R_+} a p'(v^S) v^S_\xi(\xi-(\sigma-\sigma_-)t-X(t)-\beta)  (u-\overline{u}) \,d \xi\Bigg),\quad X(0)=0,
	\end{aligned}
	\end{align}
	where $M$ is a constant which will be determined in the proof of Lemma \ref{lem:leading}. Here, $a$ is the weight function explicitly defined as
	\begin{equation}\label{weight}
	a(t,\xi) = 1+\frac{u^*-u^S(\xi-(\sigma-\sigma_-)t-X(t)-\beta)}{\sqrt{\delta_S}}.
	\end{equation}
	From this definition, the following properties can be easily obtained:
	\begin{equation}\label{est:a}
	1\le a\le 1+\sqrt{\delta_S},\quad a_\xi = -\frac{u^S_\xi}{\sqrt{\delta_S}}>0.
	\end{equation}
    The existence of the shift $X$ as a Lipschitz continuous solution to \eqref{ODE_X} is guaranteed by the standard Cauchy-Lipschitz theory, see also \cite[Lemma 3.3]{KVW23}.
	\begin{proposition}
		For any $c_1,c_2,c_3>0$, there exists a constant $C>0$ such that the following is true. For any $T>0$, and any functions $v,u\in L^\infty((0,T)\times\R)$ with
		\[c_1\le v(t,x)\le c_2,\quad |u(t,x)|\le c_3,\quad(t,x)\in[0,T]\times\R,\]
		the ODE \eqref{ODE_X} has a unique Lipschitz continuous solution $X$ on $[0,T]$. Moreover, we have
		\[|X(t)|\le Ct,\quad t\in[0,T].\]
	\end{proposition}

   The reason why we choose the shift $X(t)$ as in \eqref{ODE_X} becomes clear in Section \ref{sec:4}, where we will exploit the definition of the shift function to estimate the weighted relative entropy.

	\subsection{A priori estimate}
	We now state a priori estimate, which is the key estimate for obtaining the time-asymptotic behavior. 
	
	\begin{proposition}\label{apriori-estimate}
		For any $U_-\in \Omega_{\textup{sub}}$, let $U_*$ be the constant state which satisfies $U_*\in BL(U_-)\cap \Gamma_{\textup{trans}}$. Then, there exist positive constants $C_0,\delta_0$, and $\e_1$ such that for any two constant states $U^*$ and $U_+$ with $U^*\in R_1(U_*)$, and $U_+\in S_2(U^*)$ satisfying $\delta_R,\delta_S < \delta_0$, there exists $\beta>0$ large enough (depending only on the shock strength $\delta_S$) such that the following holds. Suppose that $(v,u)$ is the solution to \eqref{eq:NS_inflow} on $[0,T]$ for some $T>0$, and $(\overline{v},\overline{u})$ is the superposition defined in \eqref{eq: composite wave}.  Also, let $X(t)$ be the shift given in \eqref{ODE_X} with the weight function $a$ defined in \eqref{weight}.  Assume that 
		\begin{align*}
		v-\overline{v} \in C([0,T];H^1(\bbr_+)),\quad  u-\overline{u} \in C([0,T];H^1(\bbr_+))\cap L^2(0,T;H^2(\bbr_+)) ,
		\end{align*}
		and
		\begin{equation}\label{smallness}
		\|(v-\overline{v},u-\overline{u})\|_{L^\infty(0,T;H^1(\bbr_+))} \le \e_1.
		\end{equation}
		Then, for all $0\le t\le T$,
		\begin{equation*}
		\begin{aligned}
		 &\sup_{0\le t\le T}\norm{(v-\overline{v},u-\overline{u})(t,\cdot)}_{H^1(\R_+)}^2  + \delta_S \int_0^t |\dot{X}|^2\,ds + \int_0^t (\mathcal{G}_1 + \mathcal{G}_2 + \mathcal{D}_v + \mathcal{D}_{u_1} + \mathcal{D}_{u_2})\,ds\\
    &\quad \leq C_0 \big( \norm{(v-\overline{v})(0,\cdot)}_{H^1(\R_+)}^2 +  \norm{(u-\overline{u})(0,\cdot)}_{H^1(\R_+)}^2 \big) + C_0\delta_0^{1/6} + C_0  e^{-C_0\delta_S \beta},
		\end{aligned}
		\end{equation*}
		where $C_0$ is a positive constant independent of $T$ and
		\begin{equation*}
		\begin{aligned}
		&\mathcal{G}_1:=\int_{\R_+}|u^S_\xi||u-\overline{u}|^2\,d\xi, \quad \mathcal{G}_2 =\int_{\R_+}|\overline{u}_\xi||v-\overline{v}|^2\,d\xi, \\
        &\mathcal{D}_v:= \int_{\R_+} |(p(v) - p(\overline{v}))_\xi|^2\,d\xi, \quad \mathcal{D}_{u_1}:=\int_{\R_+}|(u-\overline{u})_\xi|^2\,d\xi,   
        \quad \mathcal{D}_{u_2}:= \int_{\R_+} |(u - \overline{u})_{\xi \xi}|^2\,d\xi.
		\end{aligned}
		\end{equation*}
	\end{proposition}

	The proof of Proposition \ref{apriori-estimate} will be postponed to Section \ref{sec:4} and Section \ref{sec:high-order}, where we use the method of $a$-contraction with shift and the high-order estimates, respectively.

    \subsection{Proof of Theorem \ref{thm:main}} 
    Now, we are in the position to present the proof of Theorem \ref{thm:main}. Since the proof itself is very lengthy and technical, we put the detailed proofs of the global existence and the time-asymptotic behavior in Appendix \ref{app:A} and Appendix \ref{app:B} separately. Here, we only present the main idea of the proof.\\
    
    \noindent $\bullet$ (Global existence) First of all, from Proposition $\ref{prop:local}$, we deduce that the inflow problem \eqref{eq:NS_inflow}-\eqref{boundary_inflow} admits a local-in-time solution $(v,u)$ which satisfies
    \[(v-\underline{v},u-\underline{u}) \in C([0,T];H^1(\mathbb{R}_+) )\quad  \text{and} \quad \|(v-\underline{v},u-\underline{u})\|_{H^1(\mathbb{R}_+)} <\frac{\e_1}{2}\]
    for some $T>0$, where $(\underline{v},\underline{u})$ is a pair of functions that is defined in \eqref{eq: underline v, u}. Moreover, one can prove that the $H^1$-perturbation between $(\underline{v},\underline{u})$ and the superposition $(\overline{v},\overline{u})$ is less than $\e_1/2$, that is,
    \[\|(\underline{v}-\overline{v},\underline{u}-\overline{u})\|_{L^\infty(0,T;H^1(\mathbb{R}_+))} <\frac{\e_1}{2}.\]
    Then, we employ the standard continuation argument based on Proposition \ref{apriori-estimate} to obtain the global existence of the inflow problem. \\
    
    \noindent $\bullet$ (Asymptotic behavior) To show the time-asymptotic behavior of the solution, we will follow a similar strategy used in \cite{HKKL_pre2, KVW23}. Precisely, we will show that the function
    \[g(t):=\|(v(t,\cdot)-\overline{v}(t,\cdot))_\xi\|_{L^2(\R_+)}^2+\|(u(t,\cdot)-\overline{u}(t,\cdot))_\xi\|_{L^2(\R_+)}^2\] 
    satisfies $g\in W^{1,1}(\R_+)$ by using Proposition \ref{apriori-estimate}. Then, the Gagliardo-Nirenberg-Sobolev inequality implies the desired convergence results, and completes the proof of Theorem \ref{thm:main}.
    
   \subsection{Notations} In the following, we focus on verifying Proposition \ref{apriori-estimate}. For simplicity, we will use the concise notation $L^p:=L^p(\R_+)$ and $H^1:=H^1(\R_+)$.
 
	\section{$L^2$-perturbation estimate}\label{sec:4}
	\setcounter{equation}{0}
	
	In this section, we first derive the $L^2$-estimate of the perturbation as a first step to prove Proposition \ref{apriori-estimate}. Precisely, we prove the following estimate on the $L^2$-perturbation between the solution $U=(v,u)$ to \eqref{eq:NS_inflow} and the superposition $\overline{U}=(\overline{v},\overline{u})$ defined in  \eqref{eq: composite wave}.
	\begin{lemma}\label{lem:main}
		Under the assumption of Proposition \ref{apriori-estimate}, we have
		\begin{align}
		\begin{aligned}\label{est_vu}
		&\sup_{t\in[0,T]}\|U(t,\cdot)-\overline{U}(t,\cdot)\|_{L^2}^2+\int_0^t (\delta_S|\dot{X}|^2+G_S+G_1+G_{BL}+G_R+D_{u_1})\,ds\\
		&\le C\|U_0-\overline{U}(0,\cdot)\|_{L^2}^2 +C\delta_0^{1/6} +Ce^{-C\delta_S\beta}+C\e_1^2 \left(\int_0^t \|(p(v)-p(\bar{v}))_{\xi}\|_{L^2}^2 ds +\int_0^t \|(u-\overline{u})_{\xi\xi}\|_{L^2}^2 ds\right),
		\end{aligned}
		\end{align}
		 for any $t\in[0,T]$, where
		\begin{align*}
		&G_S:=\int_{\R_+}|u^S_\xi||u-\overline{u}|^2\,d\xi,\; G_1 = C_*\int_{\R_+}a_\xi\left|p-\overline{p}-\frac{u-\overline{u}}{2C_*}\right|^2\,d\xi,\\
		&G_{BL}:=\int_{\R_+}a u^{BL}_\xi p(v|\overline{v})\,d\xi,\; G_R:=\int_{\R_+}au^R_\xi p(v|\overline{v})\,d\xi, \; D_{u_1}:=\int_{\R_+}\frac{a}{v}|(u-\overline{u})_\xi|^2\,d\xi,
		\end{align*}
    and 
    \begin{equation*} 
        C_* := \frac{1}{2\sqrt{-p'(v^*)}} - \sqrt{\delta_S}{\frac{\gamma + 1}{2 \gamma p(v^*)}} > 0.
    \end{equation*}
 
	\end{lemma}
	Here, we set $\overline{p}:=p(\bar{v})$.
	\subsection{Wave interaction estimates}
	We first provide several estimates on the interaction between different elementary waves. Recall that under a priori assumption \eqref{smallness}, the shift $X(t)$ satisfies
	\begin{align*}
	|\dot{X}(t)|\le C\e_1,\quad \mbox{and therefore},\quad |X(t)|\le C\e_1t.
	\end{align*}
	\begin{lemma}\label{lem:wave-interaction}
		Under the assumption of Proposition \ref{apriori-estimate}, there exist positive constants $c$ and $C$ such that the following estimates hold: 
		\begin{align*}
		(1)~&\int_{\R_+}\left(v^{BL}_\xi(v^R-v_*)+v^R_\xi(v_*-v^{BL})\right)d\xi\le C\left(\frac{\delta^{1/8}_R\log(1+\delta_{BL}t)}{(1+t)^{7/8}}+\frac{\delta_{BL}\delta_R}{1+\delta_{BL}t}\right),\\
		(2)~&\int_{\R_+}\left(v^{BL}_\xi(v^S-v^*)+v^S_\xi(v_*-v^{BL})\right)d\xi\le C\left(\delta_{BL}\delta_S e^{-c\delta_St}+\frac{\delta_{BL}\delta_S}{1+\delta_{BL}t}\right),\\
		(3)~&\int_{\R_+}\left(v^{R}_\xi(v^S-v^*)+v^S_\xi(v^*-v^{R})\right)d\xi\le C \left(\delta_R\delta_Se^{-c\delta_St}+\delta_R\delta_Se^{-ct}\right).
		\end{align*}
	\end{lemma}

	\begin{proof}
				
		\noindent (1) 
		The proof for (1) is motivated by \cite{MN01}. First, it follows from \eqref{BL_properties} and \eqref{rarefaction_properties} that
		\[v^{BL}_{\xi},v^R_\xi>0,\quad v^R-v_*>0,\quad v_*-v^{BL}>0.\]
		 Therefore, the integrand is always positive. On the other hand, using $v^{R}(t,0)=v_*=v^{BL}(\infty)$, we have
		\begin{align}
		\begin{aligned}\label{wave_interaction_BL_R}
		\int_{\R_+}&\left(v^{BL}_\xi(v^R-v_*)+v^R_\xi(v_*-v^{BL})\right)d\xi=\left(\int_0^{{t}}+\int_{{t}}^\infty\right)\left(v^{BL}_\xi(v^R-v_*)+v^R_\xi(v_*-v^{BL})\right)d\xi\\
		&= \Bigg(\left[(v^{BL}-v_*)(v^R-v_*)\right]^{\xi=t}_{\xi=0}-2\int_0^{{t}}v^R_{\xi}(v^{BL}-v_*)\,d\xi\\
		&\hspace{1cm}+\left[(v^R-v_*)(v_*-v^{BL})\right]_{\xi=t}^{\xi=\infty}+2\int_{{t}}^\infty v^{BL}_{\xi}(v^R-v_*)\,d\xi\Bigg)\\
		&\le 2\int_0^{{t}}v^R_\xi(v_*-v^{BL})\,d\xi +2\int_{{t}}^\infty v^{BL}_\xi (v^R-v_*)\,d\xi,
		\end{aligned}
		\end{align}
		where we used $(v^{BL}(\xi)-v_*)(v^R(t,\xi)-v_*)<0$ for all $t,\xi>0$. 
        
        We again use \eqref{BL_properties} and \eqref{rarefaction_properties} to observe that for all $\xi\in \R_+$,
		\begin{equation}\label{est:rarefaction}
		\begin{aligned}
		&|v^R-v_*|\le C\delta_R,\quad |v^R_\xi|\le \|v^R_\xi\|_{L^\infty}^{1/8}\|v^R_\xi\|^{7/8}_{L^\infty}\le \frac{C\delta_R^{1/8}}{(1+t)^{7/8}},\\
		&|v^{BL}-v_*|\le \frac{C\delta_{BL}}{1+\delta_{BL}\xi},\quad |v^{BL}_\xi|\le \frac{C\delta^2_{BL}}{(1+\delta_{BL}\xi)^2},
		\end{aligned} 
		\end{equation}
		which imply
		\begin{align*}	
		\int_0^{{t}}&v^R_\xi(v_*-v^{BL})\,d\xi + \int_{{t}}^\infty v^{BL}_\xi(v^R-v_*)\,d\xi\\
		&\le \frac{C\delta_R^{1/8}\delta_{BL}}{(1+t)^{7/8}}\int_0^{{t}}\frac{1}{1+\delta_{BL}\xi}\,d\xi+C\delta^2_{BL}\delta_R\int_{{t}}^\infty\frac{1}{(1+\delta_{BL}\xi)^2}\,d\xi\\
		&= \frac{C\delta_R^{1/8}}{(1+t)^{7/8}}\log(1+\delta_{BL}t)+\frac{C\delta_{BL}\delta_R}{1+\delta_{BL}t}.
		\end{align*}
		This together with \eqref{wave_interaction_BL_R} yields the desired estimate.\\
		
		\noindent (2) For notational convenience, we set $b = (\sigma - \sigma_-)/2$. From \eqref{BL_properties} and \eqref{shock-properties}, we have
		\[v^{BL}_\xi, v^S_\xi>0,\quad v^S-v^*>0,\quad v_*-v^{BL}>0,\]
		and therefore, the integrand is again positive. We use $v^{BL}(\infty)=v_*$ and $v^{BL}(0)=v_-$ to obtain
		\begin{align}
		\begin{aligned}\label{wave_interaction_BL_S}
		&\int_{\R_+}\left(v^{BL}_\xi(v^S-v^*)+v^S_\xi(v_*-v^{BL})\right)d\xi=\left(\int_0^{bt}+\int_{bt}^\infty\right)\left(v^{BL}_\xi(v^S-v^*)+v^S_\xi(v_*-v^{BL})\right)d\xi\\
		&= \Bigg(\left[(v^{BL}-v_*)(v^S-v^*)\right]^{\xi=bt}_{\xi=0}-2\int_0^{bt}v^S_{\xi}(v^{BL}-v_*)\,d\xi\\
		&\hspace{1cm}+\left[(v^S-v^*)(v_*-v^{BL})\right]_{\xi=bt}^{\xi=\infty}+2\int_{bt}^\infty v^{BL}_{\xi}(v^S-v^*)\,d\xi\Bigg)\\
		&\le\delta_{BL}|v^S(-(\sigma-\sigma_-)t-X(t)-\beta)-v^*| +2\int_0^{bt}v^S_\xi(v_*-v^{BL})\,d\xi +2\int_{bt}^\infty v^{BL}_\xi (v^S-v^*)\,d\xi,
		\end{aligned}
		\end{align}
		where we used $(v^{BL}-v_*)(v^S-v^*)<0$. However, since $X(t)\le C\e_1t$ with sufficiently small $\e_1$, we have
		\[-(\sigma-\sigma_-)t-X(t)-\beta\le -\frac{\sigma-\sigma_-}{2}t-\beta=-bt-\beta \leq -bt <0,\]
		and therefore we use \eqref{shock-properties} to obtain
		\begin{equation}\label{est:shock_at_0}
		|v^S(-(\sigma-\sigma_-)t-X(t)-\beta)-v^*|
		 \le C\delta_S e^{-C\delta_S t}.
		\end{equation}
		On the other hand, we again use \eqref{BL_properties} and \eqref{shock-properties} to obtain
		\begin{equation*}
			\begin{aligned}
				&|v^S-v^*|\le \delta_S,\quad |v^S_\zeta(\zeta)|\le C\delta^2_Se^{-C\delta_S|\zeta|},\quad  \zeta\in\R,\\
		&|v^{BL}-v_*|\le \frac{C\delta_{BL}}{1+\delta_{BL}\xi}\le C\delta_{BL},\quad |v^{BL}_\xi|\le \frac{C\delta^2_{BL}}{(1+\delta_{BL}\xi)^2},\quad \xi\in\R_+.
			\end{aligned}
		\end{equation*}
		These estimates imply
		\begin{align}
		\begin{aligned}\label{est:vSvBL}
		\int_0^{bt}&v^S_\xi(v_*-v^{BL})\,d\xi+\int_{bt}^\infty v^{BL}_\xi(v^S-v^*)\,d\xi\\
		&\le C\delta_{BL}\delta_S^2\int_0^{bt}e^{-C\delta_S|\xi-(\sigma-\sigma_-)t-X(t)-\beta|}\,d\xi+C\delta^2_{BL}\delta_S\int_{bt}^{\infty}\frac{1}{(1+\delta_{BL}\xi)^2}\,d\xi\\
		&\le C\delta_{BL}\delta_S e^{-C\delta_St}+C\frac{\delta_{BL}\delta_S}{1+\delta_{BL}t}.
		\end{aligned}
		\end{align}
		
		Here, we use $|X(t)|\le C\e_1t$ so that for $\xi\in[0,bt]=[0,\frac{\sigma-\sigma_-}{2}t]$,
		\[|\xi-(\sigma-\sigma_-)t-X(t)-\beta|\ge \left|\xi-\frac{3(\sigma-\sigma_-)t}{4}-\beta\right|=-\xi+\frac{3bt}{2}+\beta \ge -\xi+\frac{3bt}{2},\]
		which yields
		\[\int_0^{bt}e^{-C\delta_S|\xi-(\sigma-\sigma_-)t-X(t)-\beta|}\,d\xi\le e^{-\frac{3Cb}{2}\delta_St}\int_0^{bt}e^{C\delta_S\xi}\,d\xi
		\le \frac{1}{C\delta_S}e^{-C\delta_St}.\]
		Combining \eqref{est:shock_at_0} and \eqref{est:vSvBL} with \eqref{wave_interaction_BL_S}, we obtain the desired bound.\\
		
		\noindent (3) Finally, we use $v^R(t,0)=v_*$ and $v^R(t,\infty)=v^*$ and $(v^R-v^*)(v^S-v^*)<0$ to obtain
		\begin{align}
		\begin{aligned}\label{wave_interaction_R_S}
		&\int_{\R_+}\left(v^{R}_\xi(v^S-v^*)+v^S_\xi(v^*-v^{R})\right)d\xi=\left(\int_0^{-\sigma_-t}+\int_{-\sigma_-t}^\infty\right)\left(v^{R}_\xi(v^S-v^*)+v^S_\xi(v^*-v^{R})\right)d\xi\\
		&= \Bigg(\left[(v^{R}-v^*)(v^S-v^*)\right]^{\xi=-\sigma_-t}_{\xi=0}-2\int_0^{-\sigma_-t}v^S_{\xi}(v^{R}-v^*)\,d\xi\\
		&\hspace{1cm}+\left[(v^S-v^*)(v^*-v^{R})\right]_{\xi=-\sigma_-t}^{\xi=\infty}+2\int_{-\sigma_- t}^\infty v^{R}_{\xi}(v^S-v^*)\,d\xi\Bigg)\\
		&\le \delta_{R}|v^S(-(\sigma-\sigma_-)t-X(t)-\beta)-v^*|+2\int_0^{-\sigma_-t}v^S_\xi(v^*-v^R)\,d\xi +2\int_{-\sigma_-t}^\infty v^{R}_\xi (v^S-v^*)\,d\xi.
		\end{aligned}
		\end{align}
		We bound the first term of the right-hand side of \eqref{wave_interaction_R_S} by using \eqref{est:shock_at_0} as
		\begin{equation}\label{est:deltaRv^S}
		\delta_{R}|v^S(-(\sigma-\sigma_-)t-X(t)-\beta)-v^*|
		\le C\delta_R\delta_Se^{-C\delta_S t}.
		\end{equation}
		
		On the other hand, we use \eqref{rarefaction_properties} and the same argument as before to obtain
		\begin{align}
		\begin{aligned}\label{est:vSvR}
		&\int_0^{-\sigma_-t}v^S_\xi(v^*-v^R)\,d\xi+\int_{-\sigma_-t}^\infty v^R_\xi(v^S-v^*)\,d\xi\\
		&\le  C\left(\delta_R\delta_S^2\int_0^{-\sigma_-t}e^{-C\delta_S|\xi-(\sigma-\sigma_-)t-X(t)-\beta|}\,d\xi+\delta_{R}\delta_S\int_{-\sigma_-t}^{\infty}e^{-C(\xi+\sigma_-t+t)}\,d\xi\right)\\
		&\le C \left(\delta_R\delta_Se^{-C\delta_St}+\delta_R\delta_Se^{-Ct}\right).
		\end{aligned}
		\end{align}
		
		Finally, combining \eqref{est:deltaRv^S} and \eqref{est:vSvR} with \eqref{wave_interaction_R_S}, we prove the desired inequality.
	\end{proof}

	\subsection{Relative entropy estimate}
	
	 To prove Lemma \ref{lem:main}, we exploit the relative entropy method with the weight. To this end, we define the flux, diffusion matrix, and the entropy as
	 \[A(U):=\begin{pmatrix}
	 	-\sigma_-v-u\\ -\sigma_-u+p(v)
	 \end{pmatrix},\quad M(U):=\begin{pmatrix}
		0 & 0 \\ 0 & v^{-1}
	 \end{pmatrix},\quad \eta(U):=Q(v)+\frac{u^2}{2}\]
	 to rewrite the NS system \eqref{eq:NS_inflow} in the following general form
	 \begin{equation}\label{eq:general_form}
	 	U_t +A(U)_\xi =(M(U)\pa_\xi\nabla \eta(U))_\xi.
	 \end{equation}
	 On the other hand, recall the definition \eqref{eq: composite wave} of the superposition:
	 \[\overline{U}(t,\xi):=U^{BL}(\xi)+U^R (t,\xi)+U^S(\xi-(\sigma-\sigma_-)t-X(t)-\beta)-U_*-U^*.\]
     From \eqref{eq:BL}, \eqref{eq:shock}, and  \eqref{eq:rarefaction}, the superposition satisfies 
	 \begin{align}
	 	\begin{aligned}\label{eq:composite_wave}
	 		&\overline{v}_t -\sigma_-\overline{v}_\xi - \overline{u}_\xi = -\dot{X}v^S_\xi,\\
	 		&\overline{u}_t -\sigma_-\overline{u}_\xi +p(\overline{v})_\xi=-\dot{X}u^S_\xi+ \left(\frac{\overline{u}_\xi}{\overline{v}}\right)_\xi + S.
	 	\end{aligned}
	 \end{align}
	 Here, $S = S_{I1}+S_{I2} + S_R$ with
	 \begin{align}
	 	\begin{aligned}\label{SI1SI2}
	 		&S_{I1} := p(\overline{v})_\xi - p(v^{BL})_\xi -p(v^R)_{\xi}-p(v^S)_{\xi},\\
	 		&S_{I2} := -\left(\left(\frac{\overline{u}_\xi}{\overline{v}}\right)_\xi-\left( \frac{u^{BL}_\xi}{v^{BL}}\right)_\xi-\left( \frac{u^{R}_\xi}{v^{R}}\right)_\xi-\left(\frac{u^S_\xi}{v^S}\right)_\xi\right),\quad S_R:=-\left(\frac{u^R_\xi}{v^R}\right)_\xi.
	 	\end{aligned}
	 \end{align}
  Therefore, the equation \eqref{eq:composite_wave} for the superposition can also be expressed as  
	\begin{equation}\label{eq:general_form_2}
		\overline{U}_t +A(\overline{U})_\xi =(M(\overline{U})\pa_\xi \nabla \eta(\overline{U}))_\xi-\dot{X}\overline{U}_\xi+\begin{pmatrix}
		0 \\ S
	\end{pmatrix}.
	\end{equation}
    
	In order to measure the difference between the solution $U$ and the superposition $\overline{U}$, we define the relative entropy $\eta(U|\overline{U})$ and relative flux $A(U|\overline{U})$:
	\begin{align*}
	&\eta(U|\overline{U}):=\eta(U)-\eta(\overline{U})-\nabla\eta(\overline{U})(U-\overline{U}),\\
	&A(U|\overline{U}):=A(U)-A(\overline{U})-\nabla A(\overline{U})(U-\overline{U}),
	\end{align*}	
	as well as the relative entropy flux $G(U;\overline{U})$:
	\[G(U;\overline{U}):=G(U)-G(\overline{U})-\nabla \eta(\overline{U})(A(U)-A(\overline{U})).\]
	For the case of the NS system \eqref{eq:general_form}, the relative quantities are explicitly computed as
	\begin{align*}
	&\eta(U|\overline{U})=Q(v|\overline{v})+\frac{(u-\overline{u})^2}{2},\quad A(U|\overline{U})=\begin{pmatrix}
	0\\
	p(v|\overline{v})
	\end{pmatrix},\\
	&G(U;\overline{U})=(p(v)-p(\overline{v}))(u-\overline{u})-\sigma_-\eta(U|\overline{U}).
	\end{align*}
	Then, thanks to Lemma \ref{lem:relative_quantity}, the relative entropy $\eta(U|\overline{U})$ is equivalent to the $L^2$-perturbation $\|U-\overline{U}\|_{L^2}^2$, as long as the specific volume is bounded below and above. Therefore, we estimate the following weighted relative entropy to prove Lemma \ref{lem:main}
	\[\int_{\R_+}a(t,\xi)\eta(U(t,\xi)|\overline{U}(t,\xi))\,d\xi.\]
	\begin{lemma}\label{lem:rel}
		Let $a$ be the weight function defined in \eqref{weight}. Suppose $U$ be the solution to \eqref{eq:general_form} and $\overline{U}$ be the superposition satisfying \eqref{eq:general_form_2}. Then, 
	\begin{equation}\label{eq:rel_ent_1}
		\frac{d}{dt}\int_{\R_+} a(t,\xi)\eta(U(t,\xi)|\overline{U}(t,\xi))\,d\xi=\dot{X}Y+J^{\textup{bad}}-J^{\textup{good}}+J^{\textup{bd}},
	\end{equation}
	where
	\begin{align*}
		Y&:=-\int_{\R_+}a_\xi\eta(U|\overline{U})\,d\xi+\int_{\R_+}a\nabla^2\eta(\overline{U})U^S_\xi(U-\overline{U})\,d\xi,\\
		J^{\textup{bad}}&:=\int_{\R_+}a_\xi (p(v)-p(\overline{v}))(u-\overline{u})\,d\xi-\int_{\R_+}a\overline{u}_\xi p(v|\overline{v})\,d\xi\\
		&\quad -\int_{\R_+}a_\xi \frac{u-\overline{u}}{v}(u-\overline{u})_\xi\,d\xi+\int_{\R_+}a_\xi(u-\overline{u})(v-\overline{v})\frac{\overline{u}_\xi}{v\overline{v}}\,d\xi\\
		&\quad +\int_{\R_+}a(u-\overline{u})_\xi\frac{v-\overline{v}}{v\overline{v}}\overline{u}_\xi\,d\xi-\int_{\R_+}a(u-\overline{u})S\,d\xi,\\
		J^{\textup{good}}&:=\frac{\sigma}{2}\int_{\R_+}a_\xi|u-\overline{u}|^2\,d\xi+\sigma\int_{\R_+}a_\xi Q(v|\overline{v})\,d\xi+\int_{\R_+}\frac{a}{v}|(u-\overline{u})_\xi|^2\,d\xi,\\
		J^{\textup{bd}}&:=\Bigg[a(u-\overline{u})(p(v)-p(\overline{v}))-\frac{\sigma_-a}{2}(u-\overline{u})^2-\sigma_-aQ(v|\overline{v})\\
		&\hspace{2cm} -\frac{a}{v}(u-\overline{u})(u-\overline{u})_\xi+\frac{a}{v\overline{v}}(u-\overline{u})(v-\overline{v})\overline{u}_\xi\Bigg]_{\xi=0}.
	\end{align*}
	\end{lemma}
	\begin{proof}
		Since the proof is almost the same as \cite[Lemma 6.2]{HKKL_pre2} or \cite[Lemma 4.2]{HKKL_pre2}, we refer to these literature for the detailed proof. Compared to the previous literature, the additional term $\int_{\R_+} a(u-\overline{u})S\,d\xi$ in $J^{\textup{bad}}$ is due to the interaction between the elementary waves. Moreover, since $a_\xi>0$, the terms in $J^{\textup{good}}$ have positive sign.
	\end{proof}
	We first estimate the boundary terms $J^{\textup{bd}}$.
	\begin{lemma} \label{lem: boundary est-1}
		There exists a positive constant $C$ such that
		\[\left|\int_0^t J^{\textup{bd}}\,ds \right|\le C(\e_1+\delta_S^{1/3})e^{-C\delta_S\beta}+C\e_1^2\int_0^t \|(u-\overline{u})_{\xi\xi}\|_{L^2}^2\,ds.\]
	\end{lemma}
	\begin{proof}
	We split the terms in $J^{\textup{bd}}$ into the following five terms:
	\begin{align*}
	&J^{\textup{bd}}_1 := \left[a(u-\overline{u})(p(v)-p(\overline{v}))\right]_{\xi=0},\quad J^{\textup{bd}}_2 := -\left[\frac{\sigma_- a}{2}(u-\overline{u})^2\right]_{\xi=0},\\
	&J^{\textup{bd}}_3 := -[\sigma_- aQ(v|\overline{v})]_{\xi=0},\quad J^{\textup{bd}}_4 := -\left[\frac{a}{v}(u-\overline{u})(u-\overline{u})_{\xi}\right]_{\xi=0},\quad J^{\textup{bd}}_5 := \left[\frac{a}{v\overline{v}}(u-\overline{u})(v-\overline{v})\overline{u}_{\xi}\right]_{\xi=0}.
	\end{align*}
	We first estimate $J^{\textup{bd}}_1$ by using Lemma \ref{lem:relative_quantity} and a priori assumption as
	\begin{align*}
	|J^{\textup{bd}}_1|\le\left|\left[a(u-\overline{u})(p(v)-p(\overline{v}))\right]_{\xi=0}\right|\le C|u_--\overline{u}(t,0)|\|p(v)-p(\overline{v})\|_{L^\infty}\le C\e_1|u_--\overline{u}(t,0)|.
	\end{align*}
	However, since $U^{BL}(0)=U_-$ and $U^{R}(t,0)=U_*$, we have
	\[\overline{U}(t,0) =U_-+ U^S(-(\sigma-\sigma_-)t-X(t)-\beta)-U^*,\]
	and therefore \eqref{shock-properties} implies
	\begin{equation}\label{est_boundary}
	|u_--\overline{u}(t,0)|\le |u^S(-(\sigma-\sigma_-)t-X(t)-\beta)-u^*|\le C\delta_Se^{-C\delta_S t}e^{-C\delta_S\beta},
	\end{equation}
	where we used $X(t)\le C\e_1 t$ to obtain
	\[-(\sigma-\sigma_-)t-X(t)-\beta \le -\frac{(\sigma-\sigma_-)t}{2}-\beta<0.\]
	Therefore, we have
	\[\left|\int_0^t J^{\textup{bd}}_1\,ds\right|\le C\e_1\delta_S e^{-C\delta_S\beta}\int_0^t e^{-C\delta_Ss}\,ds\le C\e_1 e^{-C\delta_S\beta}.\]
	Similarly, we use the estimate
	\[|v_--\overline{v}(t,0)|=|v^S(-(\sigma-\sigma_-)t-X(t)-\beta)-v^*|\le C\delta_Se^{-C\delta_S t}e^{-C\delta_S\beta}\]
	and the same argument as above with the estimate \eqref{relative-est} to derive
	\[\left|\int_0^t J^{\textup{bd}}_2\,ds\right|,\,\left|\int_0^t J^{\textup{bd}}_3\,ds\right|\le C\e_1e^{-C\delta_S\beta}.\]
	For $J^{\textup{bd}}_4$, we use Sobolev interpolation inequality, \eqref{est_boundary}, and a priori assumption to obtain
	\begin{align*}
	\left|\int_0^t J^{\textup{bd}}_4\,ds \right|&=\left|\int_0^t \left[\frac{a}{v}(u-\overline{u})(u-\overline{u})_\xi\right]_{\xi=0}\,ds\right|\le C\int_0^t |u_--\overline{u}(s,0)|\|(u-\overline{u})_\xi\|_{L^\infty}\,ds\\
	&\le C\int_0^t |u_--\overline{u}(s,0)|^{4/3}\,ds+C\int_0^t \|(u-\overline{u})_{\xi}\|^2_{L^2}\|(u-\overline{u})_{\xi\xi}\|_{L^2}^2\,ds\\
	&\le C\delta_S^{1/3}e^{-C\delta_S\beta}+C\e_1^2\int_0^t \|(u-\overline{u})_{\xi\xi}\|_{L^2}^2\,ds.
	\end{align*}
	Finally, to estimate $J^{\textup{bd}}_5$, we observe that
	\[\overline{u}_{\xi}(t,\xi) = u^{BL}_{\xi}(\xi)+u^R_{\xi}(t,\xi) +u^S_{\xi}(\xi-(\sigma-\sigma_-)t-X(t)-\beta).\]
	Therefore, using \eqref{BL_properties}, \eqref{rarefaction_properties}, and \eqref{shock-properties}, we have
	\[|\overline{u}_{\xi}(t,0)|\le |u^{BL}_{\xi}(0)+u^R_{\xi}(t,0)+u^S_{\xi}(-(\sigma-\sigma_-)t-X(t)-\beta)|\le \delta_{BL}^2+\delta_R+\delta_S^2.\]
	Thus, after the same argument, we have
	\[\left|\int_0^t J^{\textup{bd}}_5\,ds\right|\le C\delta_S^2(\delta_{BL}^2+\delta_R+\delta_S^2)\int_0^te^{-C\delta_S s}e^{-C\delta_S\beta}\,ds\le C\delta_S(\delta_{BL}^2+\delta_R+\delta_S^2)e^{-C\delta_S\beta}.\]
	Combining all the estimates on the $J^{\textup{bd}}_i$ for $i=1,2,\ldots,5$, and using the smallness of $\delta_R$, and $\delta_S$, we obtain the desired estimate.
	\end{proof}
	
	\subsection{Decomposition of the good and bad terms}
	Now, we manipulate the terms $J^{\textup{bad}}$ and $J^{\textup{good}}$ in \eqref{eq:rel_ent_1} and decompose them into main bad terms and remaining terms. Notice that one of the most problematic terms in $J^{\textup{bad}}$ is
	\begin{equation}\label{bad_term-1}
	\int_{\R_+}a_\xi(p(v)-p(\overline{v}))(u-\overline{u})\,d\xi,
	\end{equation}
	as it is a cross-term between $v-\overline{v}$ and $u-\overline{u}$ with a leading order. To control it, we derive an appropriate quadratic term $|p(v)-p(\overline{v})|^2$ from the terms in $J^{\textup{good}}$ and control \eqref{bad_term-1} via maximization process. The following lemma shows how we extract the quadratic term for $v$-perturbation.
	
	\begin{lemma}\label{lem:quadratic}
		There exists a positive constant $C_*$ such that
		\begin{align}
		\begin{aligned}\label{est:quadratic}
			-&\int_{\R_+}a\overline{u}_{\xi}p(v|\overline{v})\,d\xi -\sigma\int_{\R_+}a_\xi Q(v|\overline{v})\,d\xi\\
			&\le -C_*\int_{\R_+}a_\xi |p(v)-p(\overline{v})|^2\,d\xi -G_{BL}-G_{R} +C(\delta_0+\e_1)\int_{\R_+}a|u^S_{\xi}||p(v)-p(\overline{v})|^2\,d\xi\\
			&\quad +C\delta_S\int_{\R_+}a_\xi |p(v)-p(\overline{v})|^2\,d\xi +C\int_{\R_+}a_\xi |p(v)-p(\overline{v})|^3\,d\xi\\
			&\quad + C \int_{\R_+}a_{\xi}(|v^{BL}-v_*|+|v^R-v^*|)|p(v)-p(\overline{v})|^2\,d\xi,
		\end{aligned}
		\end{align}
		where
		\[G_{BL}=\int_{\R_+}au^{BL}_\xi p(v|\overline{v})\,d\xi,\quad G_R=\int_{\R_+}au^R_\xi p(v|\overline{v})\,d\xi.\]
		Note that, since $u^{BL}_\xi>0$ and $u^R_\xi>0$, we have $G_{BL}>0$ and $G_{R}>0$.
	\end{lemma}

	\begin{proof}
		Since $\overline{u}_\xi = u^{BL}_\xi+u^R_\xi+u^S_\xi$, we have
		\begin{align*}
			-\int_{\R_+}a\overline{u}_{\xi}p(v|\overline{v})\,d\xi
			=-G_{BL}-G_R-\int_{\R_+}au^S_\xi p(v|\overline{v})\,d\xi.
		\end{align*}
		We now use \eqref{relative-est} to obtain
		\begin{align*}
		&p(v|\overline{v})\le \frac{\gamma+1}{2\gamma p(\overline{v})}|p(v)-p(\overline{v})|^2+C\e_1|p(v)-p(\overline{v})|^2,\\
		&|p(\overline{v})-p^*|\le C(|v^{BL}-v_*|+|v^R-v^*|+|v^S-v^*|)\le C|v^{BL} - v_*|+C\delta_0,
		\end{align*}
		where $p^*:=p(v^*)$. Together with \eqref{est:a}, these estimates yield
		\begin{align*}
		\begin{aligned}
			&-\int_{\R_+}au^S_\xi p(v|\overline{v})\,d\xi \\
			&\quad \le \frac{\gamma+1}{2\gamma p^*}\int_{\R_+}a|u^S_\xi||p(v)-p(\overline{v})|^2\,d\xi+C(\delta_0+\e_1)\int_{\R_+}a|u^S_{\xi}||p(v)-p(\overline{v})|^2\,d\xi\\
			&\qquad + \int_{\R_+} a |u_\xi^S||v^{BL}-v_*||p(v) - p(\bar{v})|^2\,d\xi\\
			&\quad \le \frac{\gamma+1}{2\gamma p^*}\int_{\R_+}(1+\sqrt{\delta_S})\sqrt{\delta_S}a_\xi|p(v)-p(\overline{v})|^2\,d\xi +C(\delta_0+\e_1)\int_{\R_+}a|u^S_{\xi}||p(v)-p(\overline{v})|^2\,d\xi\\
			&\qquad + \int_{\R_+} a |u_\xi^S||v^{BL}-v_*||p(v) - p(\bar{v})|^2\,d\xi.
		\end{aligned}
		\end{align*}
		Therefore, we have
		\begin{equation}\label{upper_p}
		\begin{aligned}
		-\int_{\R_+}a\overline{u}_{\xi}p(v|\overline{v})\,d\xi&\le -G_{BL}-G_R+\frac{\gamma+1}{2\gamma p^*}\int_{\R_+}(1+\sqrt{\delta_S})\sqrt{\delta_S}a_\xi|p(v)-p(\overline{v})|^2\,d\xi\\
		&\quad +  C(\delta_0+\e_1)\int_{\R_+}a|u^S_\xi||p(v)-p(\overline{v})|^2 \, d \xi+\int_{\R_+} a |u_\xi^S||v^{BL}-v_*||p(v) - p(\bar{v})|^2\,d\xi\\
		&\le -G_{BL}-G_R+\frac{\gamma+1}{2\gamma p^*}\int_{\R_+}(1+\sqrt{\delta_S})\sqrt{\delta_S}a_\xi|p(v)-p(\overline{v})|^2\,d\xi\\
		&\quad +C(\delta_0+\e_1)\int_{\R_+}a|u^S_\xi||p(v)-p(\overline{v})|^2 \, d \xi+C\int_{\R_+} a_\xi|v^{BL}-v_*||p(v) - p(\bar{v})|^2\,d\xi.\\
		\end{aligned}
		\end{equation}
		On the other hand, since
		\begin{equation}\label{est:s}
			|\sigma-\sigma^*|\le C\delta_S,\quad \mbox{where}\quad \sigma^*:=\sqrt{-p'(v^*)},
		\end{equation}
		we use \eqref{relative-est} to obtain
		\begin{align}
		\begin{aligned}\label{lower_Q}
			&\sigma\int_{\R_+}a_{\xi} Q(v|\overline{v})\,d\xi \ge \sigma \int_{\R_+}a_\xi\left(\frac{p(\overline{v})^{-\frac{1}{\gamma}-1}}{2\gamma}-\frac{1+\gamma}{3\gamma^2}p(\overline{v})^{-\frac{1}{\gamma}-2}|p(v)-p(\overline{v})|\right)|p(v)-p(\overline{v})|^2\,d\xi\\
			&\quad \ge\frac{\sigma^*(p^*)^{-\frac{1}{\gamma}-1}}{2\gamma} \int_{\R_+}a_\xi|p(v)-p(\overline{v})|^2\,d\xi -C\int_{\R_+}a_\xi|\overline{v}-v^*||p(v)-p(\overline{v})|^2\,d\xi\\
			&\qquad-C\int_{\R_+}a_\xi|p(v)-p(\overline{v})|^3\,d\xi\\
			&\quad \ge\frac{\sigma^*(p^*)^{-\frac{1}{\gamma}-1}}{2\gamma} \int_{\R_+}a_\xi|p(v)-p(\overline{v})|^2\,d\xi -C\delta_S\int_{\R_+}a_\xi|p(v)-p(\overline{v})|^2\,d\xi\\
			&\qquad- C\int_{\R_+} a_\xi (|v^{BL}-v_*|+|v^{R}-v^*|)|p(v) - p(\bar{v})|^2\,d\xi -C\int_{\R_+}a_\xi|p(v)-p(\overline{v})|^3\,d\xi.
		\end{aligned}
		\end{align}
		Here, we used
		\begin{align*}
			\left|\sigma\frac{p(\overline{v})^{-\frac{1}{\gamma}-1}}{2\gamma}-\sigma^*\frac{(p^*)^{-\frac{1}{\gamma}-1}}{2\gamma} \right| \le C\delta_S +C|v^{BL}-v_*|+C|v^R-v^*|.
		\end{align*}
		Now, as $\frac{\sigma^*(p^*)^{-\frac{1}{\gamma}-1}}{2\gamma}=\frac{1}{2\sigma^*}$ by definition, we introduce a constant
		\[C_*:= \frac{1}{2\sigma^*}-\sqrt{\delta_S}\frac{\gamma+1}{2\gamma p^*}>0\]
		and combine \eqref{upper_p} and \eqref{lower_Q} to derive \eqref{est:quadratic}.
	\end{proof}

	Therefore, we use Lemma \ref{lem:quadratic} to control the first two terms in $J^\textup{bad}$ and the second term in $J^\textup{good}$ as
	\begin{align*}
		&\int_{\R_+}a_\xi (p(v)-p(\overline{v}))(u-\overline{u})\,d\xi -\int_{\R_+}a \overline{u}_\xi p(v|\overline{v})\,d\xi -\sigma\int_{\R_+}a_\xi Q(v|\overline{v})\,d\xi\\
		&\le \int_{\R_+}a_\xi\left[(p(v)-p(\overline{v}))(u-\overline{u})-C_*|p(v)-p(\overline{v})|^2\right]\,d\xi -G_{BL}-G_R \\
		&\quad +C(\delta_0+\e_1)\int_{\R_+}a|u^S_{\xi}||p(v)-p(\overline{v})|^2d\xi +C \delta_S\int_{\R_+}a_\xi |p(v)-p(\overline{v})|^2d\xi +C\int_{\R_+}a_\xi |p(v)-p(\overline{v})|^3d\xi\\
		&\quad + C\int_{\R_+} a_\xi (|v^{BL}-v_*|+|v^{R}-v^*|)|p(v) - p(\bar{v})|^2\,d\xi.
	\end{align*}
	Since
	\[(p-\overline{p})(u-\overline{u})-C_*|p-\overline{p}|^2=-C_*\left|(p-\overline{p})-\frac{u-\overline{u}}{2C_*}\right|^2+\frac{1}{4C_*}|u-\overline{u}|^2,\]
	we finally derive
	\begin{align}
	\begin{aligned}\label{est:two_bad}
		&\int_{\R_+}a_\xi (p(v)-p(\overline{v}))(u-\overline{u})\,d\xi -\int_{\R_+}a \overline{u}_\xi p(v|\overline{v})\,d\xi -\sigma\int_{\R_+}a_\xi Q(v|\overline{v})\,d\xi\\
		&\le -C_*\int_{\R_+}a_\xi \left|(p-\overline{p})-\frac{u-\overline{u}}{2C_*}\right|^2\,d\xi + \frac{1}{4C_*}\int_{\R_+}a_\xi |u-\overline{u}|^2\,d\xi -G_{BL}-G_R\\
		&\quad +C(\delta_0+\e_1)\int_{\R_+}a|u^S_{\xi}||p(v)-p(\overline{v})|^2d\xi +C \delta_S\int_{\R_+}a_\xi |p(v)-p(\overline{v})|^2d\xi\\
		&\quad+C\int_{\R_+}a_\xi |p(v)-p(\overline{v})|^3d\xi +C\int_{\R_+} a_\xi (|v^{BL}-v_*|+|v^{R}-v^*|)|p(v) - p(\bar{v})|^2\,d\xi.
	\end{aligned}
	\end{align}
	Therefore, we use \eqref{est:two_bad} to further estimate the right-hand side of \eqref{eq:rel_ent_1} as
	\begin{align*}
	\frac{d}{dt}\int_{\R_+}a\eta(U|\overline{U})\,d\xi \le \dot{X}Y+B-G+J^\textup{bd},
	\end{align*}
	where
	\begin{align*}
	 	B&:=\frac{1}{4C_*}\int_{\R_+}a_\xi |u-\overline{u}|^2\,d\xi  +C(\delta_0+\e_1)\int_{\R_+}a|u^S_{\xi}||p(v)-p(\overline{v})|^2\,d\xi \\
	 	&\quad +C \delta_S\int_{\R_+}a_\xi |p(v)-p(\overline{v})|^2\,d\xi +C\int_{\R_+}a_\xi |p(v)-p(\overline{v})|^3\,d\xi-\int_{\R_+}a_\xi \frac{u-\overline{u}}{v}(u-\overline{u})_\xi\,d\xi\\
	 	&\quad+\int_{\R_+}a_\xi(u-\overline{u})(v-\overline{v})\frac{\overline{u}_\xi}{v\overline{v}}\,d\xi +\int_{\R_+}a(u-\overline{u})_\xi\frac{v-\overline{v}}{v\overline{v}}\overline{u}_\xi\,d\xi\\
		&\quad + C \int_{\R_+}a_{\xi}(|v^{BL} - v_*|+|v^R-v^*|)|p(v)-p(\overline{v})|^2\,d\xi-\int_{\R_+}a(u-\overline{u})S\,d\xi,
	 \end{align*}
 	and
 	\begin{align*}
 		G&:=C_*\int_{\R_+}a_\xi \left|(p-\overline{p})-\frac{u-\overline{u}}{2C_*}\right|^2\,d\xi +\frac{\sigma}{2}\int_{\R_+}a_\xi|u-\overline{u}|^2\,d\xi +G_{BL}+G_R+\int_{\R_+}\frac{a}{v}|(u-\overline{u})_\xi|^2\,d\xi.
 	\end{align*}
	We now decompose the terms in $B$ and $G$ as
	\begin{align*}
		B = \sum_{i=1}^9 B_{i},\quad G = G_1+G_2+G_{BL}+G_R+D_{u_1}.
	\end{align*}
	Precisely, we define
	\begin{align*}
		\begin{split}
		&\begin{aligned}
			B_1 &= \frac{1}{4C_*}\int_{\R_+}a_\xi |u-\overline{u}|^2\,d\xi,
			&&B_2 =C(\delta_0+\e_1)\int_{\R_+}a|u^S_{\xi}||p(v)-p(\overline{v})|^2\,d\xi,\\
			B_3&=C \delta_S\int_{\R_+}a_\xi |p(v)-p(\overline{v})|^2\,d\xi,
			&&B_4= C\int_{\R_+}a_\xi |p(v)-p(\overline{v})|^3\,d\xi,\\
			B_5&= -\int_{\R_+}a_\xi \frac{u-\overline{u}}{v}(u-\overline{u})_\xi\,d\xi,
			&& B_6=\int_{\R_+}a_\xi(u-\overline{u})(v-\overline{v})\frac{\overline{u}_\xi}{v\overline{v}}\,d\xi,\\
			B_7&=\int_{\R_+}a(u-\overline{u})_\xi\frac{v-\overline{v}}{v\overline{v}}\overline{u}_\xi\,d\xi, && B_8=  C \int_{\R_+}a_{\xi}(|v^{BL} - v_*|+|v^R-v^*|)|p(v)-p(\overline{v})|^2\,d\xi,\\
			B_9&=-\int_{\R_+}a(u-\overline{u})S\,d\xi,
		\end{aligned}
		\end{split}
	\end{align*}
	and
	\begin{align*}
		\begin{split}
			&\begin{aligned}
				&G_1 =C_*\int_{\R_+}a_\xi \left|(p-\overline{p})-\frac{u-\overline{u}}{2C_*}\right|^2\,d\xi,
				\quad G_2 =\frac{\sigma}{2}\int_{\R_+}a_\xi|u-\overline{u}|^2\,d\xi,\\
				&G_{BL}=\int_{\R_+}au^{BL}_\xi p(v|\overline{v})\,d\xi,
				\quad G_{R}=\int_{\R_+}a u^{R}_{\xi}p(v|\overline{v})\,d\xi,\quad D_{u_1} =\int_{\R_+}\frac{a}{v}|(u-\overline{u})_\xi|^2\,d\xi.
			\end{aligned}
		\end{split}
	\end{align*}
	Finally, we split $Y$, which is,
	\begin{align*}
		Y=-\int_{\R_+}a_\xi \left(\frac{|u-\overline{u}|^2}{2}+Q(v|\overline{v})\right)\,d\xi+\int_{\R_+}au^S_\xi(u-\overline{u})\,d\xi-\int_{\R_+}ap'(\overline{v})v^S_\xi(v-\overline{v})\,d\xi,
	\end{align*}
	as
	\[Y=\sum_{i=1}^6Y_i,\]
	where
	\begin{align*}
		&Y_1:=\int_{\R_+}au^S_\xi(u-\overline{u})\,d\xi,\quad Y_2:=\frac{1}{\sigma}\int_{\R_+}ap'(v^S)v^S_\xi(u-\overline{u})\,d\xi,\\
		&Y_3:=\frac{1}{\sigma}\int_{\R_+}a\big(p'(\overline{v}) - p'(v^S)\big)v^S_\xi(u-\overline{u})d\xi, \quad Y_4:=-\int_{\R_+}ap'(\overline{v})v^S_\xi\left(v-\overline{v}+\frac{2C_*}{\sigma}(p(v)-p(\overline{v}))\right)d\xi,\\
        &Y_5:=\frac{2C_*}{\sigma}\int_{\R_+}ap'(\overline{v})v^S_\xi\left(p(v)-p(\overline{v})-\frac{u-\overline{u}}{2C_*}\right)\,d\xi, \quad Y_6:= -\int_{\R_+}a_\xi \left(\frac{|u-\overline{u}|^2}{2}+Q(v|\overline{v})\right)\,d\xi.\\
	\end{align*}
	In particular, the equation \eqref{ODE_X} for the shift $X$ can be written as $\dot{X}:=-\frac{M}{\delta_S}(Y_1+Y_2)$, which yields
	\[\dot{X}Y =-\frac{\delta_S}{M}|\dot{X}|^2+\dot{X}\sum_{i=3}^6Y_i.\]
	In summary, we obtain the following estimate on the weighted relative entropy:
	\begin{align}
	\begin{aligned}\label{eq:rel_ent_2}
		\frac{d}{dt}\int_{\R_+}a\eta(U|\overline{U})\,d\xi&\le-\frac{\delta_S}{2M}|\dot{X}|^2+B_1-G_2-\frac{5}{8}D_{u_1}\\
		&\quad -\frac{\delta_S}{2M}|\dot{X}|^2+\dot{X}\sum_{i=3}^6 Y_i+\sum_{i=2}^9 B_i-G_1-G_{BL}-G_R-\frac{3}{8}D_{u_1}+J^\textup{bd},
	\end{aligned}
	\end{align}
where the terms in the first line of the right-hand side are of leading order, and therefore, we need to estimate it carefully as in the following subsection. The terms in the second line are of a high order, and therefore, they can be treated in a more rough way. However, the $B_7$ term requires a more detailed analysis because the strength of the boundary layer solution  is not weak.

	\subsection{Estimate on the leading order terms}
	We first estimate the leading order terms in \eqref{eq:rel_ent_2}:
	\begin{equation}\label{leading_order}
	-\frac{\delta_S}{2M}|\dot{X}|^2+B_1-G_2-\frac{5}{8}D_{u_1}.\end{equation}
	
	\begin{lemma}\label{lem:leading}
		For sufficiently large $\beta>0$, there exists a constant $C_1>0$ independent of $\delta_S$ such that
		\begin{equation}\label{est:leading}
		-\frac{\delta_S}{2M}|\dot{X}|^2+B_1-G_2-\frac{5}{8}D_{u_1}\le -C_1G_S,
		\end{equation}
		where
		\[G_S =\int_{\R_+}|u^S_\xi||u-\overline{u}|^2\,d\xi.\]
	\end{lemma}
	\begin{proof}
	The desired estimate uses the localization by derivatives of weight and shocks.  For that,  we define the change of variable $y=y(\xi)$ for each fixed time $t$ as
	\beq\label{ydef}
	y:=\frac{u^*-u^S(\xi-(\sigma-\sigma_-)t-X(t)-\beta)}{\delta_S},
	\eeq
	which satisfies
	\[\frac{dy}{d\xi} = -\frac{u^S_\xi(\xi-(\sigma-\sigma_-)t-X(t)-\beta)}{\delta_S}>0,\]
	and
	\[y(0)=\frac{u^*-u^S(-(\sigma-\sigma_-)t-X(t)-\beta)}{\delta_S}=:y_0>0,\quad \lim_{\xi\to\infty}y(\xi)=1.\]
	We note that the estimate \eqref{est_boundary} implies that for any $t>0$,
	\[y_0\le Ce^{-C\delta_S\beta},\]
	which can be arbitrarily close to 0 by choosing large enough $\beta$. Moreover, we introduce 
	\[f:=(u(t,\cdot)-\overline{u}(t,\cdot))\circ y^{-1},\]
	which is a perturbation $u-\overline{u}$ in terms of the $y$ variable. We estimate each term in \eqref{leading_order} separately.\\

	\noindent $\bullet$ (Estimate of $-\frac{\delta_S}{2M}|\dot{X}|^2$): First, observe that the weight function can be written as $a=1+\sqrt{\delta_S}y$. Then, using the auxiliary variable $y$ and $f$, we represent the terms $Y_1$ and $Y_2$ as
	\begin{align*}
		&Y_1 = \int_{\R_+}a u^S_{\xi}(u-\overline{u})\,d\xi=-\delta_S\int_{y_0}^1(1+\sqrt{\delta_S}y)f\,dy,\\
		&Y_2 = -\frac{1}{\sigma^2}\int_{\R_+}ap'(v^S)u^S_\xi(u-\overline{u})\,d\xi= \frac{\delta_S}{\sigma^2}\int_{y_0}^1(1+\sqrt{\delta_S}y)p'(v^S)f\,dy.
	\end{align*}
	Then, using $|\sigma^2+p'(v^S)|\le C\delta_S$, one can obtain
	\begin{align*}
	\left|\dot{X}-2M\int_{y_0}^1 f\,dy\right|\le \frac{M}{\delta_S}\sum_{i=1}^2\left|Y_i+\delta_S\int_{y_0}^1f\,dy\right|\le C\sqrt{\delta_S}\int_{y_0}^1 |f|\,dy.
	\end{align*}
	This and the algebraic inequality $\frac{a^2}{2} - b^2 \leq (b-a)^2$ for all $a, b \geq 0$ imply
	\[-\frac{\delta_S}{2M}|\dot{X}|^2\le -M\delta_S \left(\int_{y_0}^1 f\,dy\right)^2+C\delta_S^2\int_{y_0}^1|f|^2\,dy.\]
	
	\noindent $\bullet$ (Estimate of $B_1-G_2$): Recalling the definition of $B_1$ and $G_2$, we have
	\[B_1-G_2 = \left(\frac{1}{4C_*}-\frac{\sigma}{2}\right)\int_{\R_+}a_\xi|u-\overline{u}|^2\,d\xi,\]
	where
	\[
		C_*= \frac{1}{2\sigma^*}-\sqrt{\delta_S}\frac{\gamma+1}{2\gamma p^*}=\frac{1}{2\sigma^*}-\sqrt{\delta_S}\alpha^*\sigma^*,\quad \alpha^*:=\frac{\gamma+1}{2\gamma\sigma^*p^*}.
	\]
	Using \eqref{est:s} and the smallness of $\delta_S$, we obtain the following estimate
	\begin{equation}\label{est:C}
	\frac{1}{4C_*}-\frac{\sigma}{2} = \frac{\sigma^*}{2}\left(\frac{1}{1 - 2\sqrt{\delta_S} \alpha^* (\sigma^*)^2}\right) - \frac{\sigma}{2} \le (\sigma^*)^3\alpha^*\sqrt{\delta_S}+C\delta_S.
	\end{equation}
	Therefore, we have
	\[B_1-G_2=\sqrt{\delta_S}\left(\frac{1}{4C_*}-\frac{\sigma}{2}\right)\int_{y_0}^1 |f|^2\,dy\le \left((\sigma^*)^3\alpha^*\delta_S+C\delta_S^{3/2}\right)\int_{y_0}^1|f|^2\,dy.\]
	
	\noindent $\bullet$ (Estimate of $D_{u_1}$): We use $a\ge 1$ and the change of variables to derive
	\begin{align*}
		D_{u_1}&\ge\int_{\R_+}\left(\frac{1}{v^S} + \left(\frac{1}{v} - \frac{1}{\bar{v}}\right)\right)|(u-\overline{u})_\xi|^2d\xi + \int_{\R_+} \left(\frac{1}{\bar{v}} - \frac{1}{v^S}\right)|(u-\overline{u})_\xi|^2d\xi\\
		&\ge \int_{y_0}^1 |\pa_y f|^2\left(\frac{1}{v^S} + \left(\frac{1}{v} - \frac{1}{\bar{v}}\right)\right)\frac{dy}{d\xi}dy.
	\end{align*}
	Here, we used $\bar{v} - v^S = v^{BL} + v^R - v^* - v_* \leq 0$.
	Then, following the estimates in \cite{HKKL_pre2}, we have 
	\[
	\left|\frac{1}{y(1-y)}\frac{1}{v^S}\left(\frac{dy}{d\xi}\right) - \delta_S (\sigma^*)^3 \alpha^*  \right| \le C\delta_S^2.
	\]
	This implies that
	\[D_{u_1}\ge (\sigma^*)^3\alpha^* \delta_S(1-C(\delta_S+\e_1))\int_{y_0}^1 (y-y_0)(1-y)|\pa_yf|^2\,dy .\]
	Now, we use Poincar\'e type inequality in Lemma \ref{lem:Poincare} with $c=y_0, d=1$ to have
	\[
		D_{u_1} \ge  2(\sigma^*)^3\alpha^* \delta_S(1-C(\delta_S+\e_1))\left(\int_{y_0}^1 |f|^2\,dy - \frac{1}{1-y_0}\left(\int_{y_0}^1 f\,dy\right)^2 \right).
	\]
	Here, we used the identity:
	\[
	\int_{y_0}^1 \left|f - \frac{1}{1-y_0}\int_{y_0}^1 f\,dy\right|^2\,dy = \int_{y_0}^1 |f|^2\,dy - \frac{1}{1-y_0}\left(\int_{y_0}^1 f\,dy\right)^2.
	\]
	Then, gathering all the estimates and using the smallness of $\delta_S$ and $\e_1$, we obtain
	\begin{align*}
	-\frac{\delta_S}{2M}&|\dot{X}|^2+B_1-G_2-\frac{5}{8}D_{u_1}\\
	&\le -M\delta_S\left(\int_{y_0}^1f\,dy\right)^2+((\sigma^*)^3\alpha^*\delta_S+C\delta_S^{3/2}+C\delta_S^2)\int_{y_0}^1|f|^2\,dy\\
	&\quad -\frac{9}{8}(\sigma^*)^3\alpha^*\delta_S\left(\int_{y_0}^1|f|^2\,dy-\frac{1}{1-y_0}\left(\int_{y_0}^1f\,dy\right)^2\right)\\
	&\le -M\delta_S\left(\int_{y_0}^1f\,dy\right)^2-\frac{1}{16}(\sigma^*)^3\alpha^*\delta_S\int_{y_0}^1|f|^2\,dy+\frac{9(\sigma^*)^3\alpha^*\delta_S}{8(1-y_0)}\left(\int_{y_0}^1f\,dy\right)^2.
	\end{align*}
	 
	Now, by choosing $M=\frac{3}{2}(\sigma^*)^3\alpha^*$, and $\beta$ sufficiently large so that $y_0<\frac{1}{4}$ to derive
	\[-\frac{\delta_S}{2M}|\dot{X}|^2+B_1-G_2-\frac{5}{8}D_{u_1}\le -\frac{(\sigma^*)^3\alpha^*}{16}\delta_S\int_{y_0}^1|f|^2\,dy=-\frac{(\sigma^*)^3\alpha^*}{16}\int_{\R_+}|u^S_\xi||u-\overline{u}|^2\,d\xi.\]
	Therefore, \eqref{est:leading} holds for $C_1 = \frac{(\sigma^*)^3\alpha^*}{16}$.
	\end{proof}
	
	\subsection{Estimate on the remaining terms}
	We now present the estimate on the remaining terms of \eqref{eq:rel_ent_2}. First, we use Cauchy-Schwarz inequality to bound the remaining term as
	\begin{align}
	\begin{aligned}\label{remaining}
	-\frac{\delta_S}{2M}&|\dot{X}|^2+\dot{X}\sum_{i=3}^6 Y_i+\sum_{i=2}^9 B_i -G_1-G_{BL}-G_R-\frac{3}{8}D_{u_1}+J^{\textup{bd}}\\
	&\le -\frac{\delta_S}{4M}|\dot{X}|^2+\frac{C}{\delta_S}\sum_{i=3}^6|Y_i|^2+\sum_{i=2}^9 B_i-G_1-G_{BL}-G_R-\frac{3}{8}D_{u_1}+J^{\textup{bd}}.
	\end{aligned}
	\end{align}
	
In the next two lemmas, we estimate the terms on the right-hand side of \eqref{remaining}.

\begin{lemma}\label{lem:rem1}
	For sufficiently small $\delta_0$ and $\e_1$, we have
	\begin{align*}
	&\frac{C}{\delta_S}|Y_3|^2 \le C\e_1^2 D_{u_1} + C \bigg((\delta_{BL}^4 + \delta_R^4)\delta_S^2 e^{-c\delta_St} + \frac{\delta_{BL}^4\delta_S^2}{(1+\delta_{BL}t)^4} + \delta_R^4 \delta_S^2 e^{-ct}\bigg),\\
	&\frac{C}{\delta_S}|Y_4|^2 \le C\delta_S (G_1 + C_1G_S) + C \bigg((\delta_{BL}^4 + \delta_R^4)\delta_S^2 e^{-c\delta_St} + \frac{\delta_{BL}^4\delta_S^2}{(1+\delta_{BL}t)^4} + \delta_R^4 \delta_S^2 e^{-ct}\bigg)\\
	&\phantom{\frac{C}{\delta_S}|Y_4|^2 \le C\delta_S} +C\e_1^2 \int_{\R_+} |(p(v)-p(\bar{v}))_\xi|^2\,d\xi, \\
	&\frac{C}{\delta_S} \sum_{i=5}^6 |Y_i|^2 \le C(\delta_S^{1/2}+\e_1^2)(G_1 + C_1 G_S).
	\end{align*}
\end{lemma}
\begin{proof}
	\noindent $\bullet$ (Estimate of $Y_3$): Using \eqref{smallness} and Sobolev interpolation, we have
	\begin{align*}
		\frac{C}{\delta_S}|Y_3|^2 &\le \frac{C}{\delta_S}\|u - \bar{u}\|_{L^\infty}^2 \left(\int_{\R_+}|v^S_\xi| \big|p'(\overline{v}) - p'(v^S)\big| \,d\xi\right)^2 \\
		&\le \frac{C}{\delta_S}\|u - \bar{u}\|_{L^2}\|(u - \bar{u})_\xi\|_{L^2} \left(\int_{\R_+}\big| |v^S_\xi|\bar{v} - v^S \big| \,d\xi\right)^2 \\
		&\le \frac{C}{\delta_S}\left(\e_1^2 \delta_S  D_{u_1} + \frac{C}{\delta_S}\left(\int_{\R_+} |v^S_\xi|\big|\bar{v} - v^S \big| \,d\xi\right)^4 \right)\\
		&\le C\e_1^2 D_{u_1} + \frac{C}{\delta_S^2} \left(-\int_{\R_+} v^S_\xi\big(v^{BL} + v^R - v_* - v^*)\,d\xi\right)^4.
	\end{align*}
	Here, we used $v_\xi^S > 0$ and $v^{BL} + v^R - v_* - v^* \le 0 $ . Therefore, we obtain 
	\[
		\frac{C}{\delta_S}|Y_3|^2 \le  C\e_1^2 D_{u_1} + \frac{C}{\delta_S^2}\left[\left(\int_{\R_+} v_\xi^S\big( v_* -v^{BL}) \,d\xi\right)^4 + \left(\int_{\R_+} v_\xi^S\big(v^* - v^R) \,d\xi\right)^4\right].
	\]
	Then, thanks to Lemma \ref{lem:wave-interaction}, we have
	\begin{align*}
		\left(\int_{\R_+} v_\xi^S\big( v_* -v^{BL}) \,d\xi\right)^4 &\le \left(\int_{\R_+}\left(v^{BL}_\xi(v^S-v^*)+v^S_\xi(v_*-v^{BL})\right)d\xi\right)^4\\
		&\le  C\left(\delta_{BL}^4\delta_S^4 e^{-c\delta_St}+\left(\frac{\delta_{BL}\delta_S}{1+\delta_{BL}t}\right)^4 \right),
	\end{align*}
	and
	\begin{align*}
		\left(\int_{\R_+}v_\xi^S (v^* - v^R) \,d\xi\right)^4 &\le \left(\int_{\R_+}\left(v^{R}_\xi(v^S-v^*)+v^S_\xi(v^*-v^{R})\right)d\xi\right)^4\\
		&\le  C \left(\delta_R^4\delta_S^4e^{-c\delta_St}+\delta_R^4\delta_S^4e^{-ct}\right).
	\end{align*}
	Here, we used $v^{BL}_\xi(v^S-v^*) \ge 0$ and $v^{R}_\xi(v^S-v^*) \ge 0$. Thus, we obtain
	\begin{align*}
		\frac{C}{\delta_S}|Y_3|^2 \le C\e_1^2 D_{u_1} + C \bigg((\delta_{BL}^4 + \delta_R^4)\delta_S^2 e^{-c\delta_St} + \frac{\delta_{BL}^4\delta_S^2}{(1+\delta_{BL}t)^4} + \delta_R^4 \delta_S^2 e^{-ct}\bigg).
	\end{align*}
    
        \noindent $\bullet$ (Estimate of $Y_4$): From \eqref{est:s} and \eqref{est:C}, we find that
		\[
		\left|\frac{\sigma}{2C_*} + p'(v^S) \right| \le \left|\frac{\sigma}{2C_*} - \sigma^2 \right| + \left|\sigma^2 + p'(v^S) \right| \le C\sqrt{\delta_S}.
		\]
		By virtue of the above inequality, $Y_4$ can be estimated as follows:
		\begin{align*}
			Y_4&=-\frac{2C_*}{\sigma}\int_{\R_+}ap'(\overline{v})v^S_\xi\left((p(v)-p(\overline{v})) +\frac{\sigma}{2C_*}(v-\overline{v})\right)\,d\xi\\
			&\leq C\int_{\R_+} |v^S_\xi|\big|(p(v)-p(\overline{v})) -p'(\bar{v})(v-\overline{v})\big|\,d\xi + C\int_{\R_+} |v^S_\xi||p'(\bar{v}) - p'(v^S)||v-\bar{v}|\,d\xi\\
			&\quad + C\int_{\R_+}|v^S_\xi|\left|\frac{\sigma}{2C_*} + p'(v^S) \right||v-\bar{v}|\,d\xi\\
			&\leq \underbrace{C\int_{\R_+} |v^S_\xi||v-\bar{v}|^2\,d\xi + C\delta_S^{1/2}\int_{\R_+} |v^S_\xi||v-\bar{v}|\,d\xi}_{=:I_1} + \underbrace{C\int_{\R_+} |v^S_\xi||p(\bar{v}) - p(v^S)||p(v)-p(\bar{v})|\,d\xi}_{=:I_2}.
		\end{align*}
		For the first two terms on the right-hand side, we use H\"older's inequality to have
		\begin{align*}
			\frac{C}{\delta_S}|I_1|^2 &\le \frac{C}{\delta_S}\left(\int_{\R_+} |v^S_\xi||v-\bar{v}|^2\,d\xi\right)^2 + C \left(\int_{\R_+} |v^S_\xi||v-\bar{v}|\,d\xi\right)^2\\
			&\le \frac{C}{\delta_S}\|v-\bar{v}\|_{L^2}^2\|v_\xi^S\|_{L^\infty}\int_{\R_+} |v_\xi^S||v-\bar{v}|^2\,d\xi + C \int_{\R_+}|v_\xi^S|\,d\xi\int_{\R_+}|v_\xi^S||v-\bar{v}|^2\,d\xi\\
			&\le C\delta_S (G_1 + C_1 G_S).
		\end{align*}
		For $I_2$, using the similar calculations as in $Y_3$, we find that 
		\[
			\frac{C}{\delta_S}|I_2|^2 \le C\e_1^2 \int_{\R_+} |(p(v)-p(\bar{v}))_\xi|^2\,d\xi + C \bigg((\delta_{BL}^4 + \delta_R^4)\delta_S^2 e^{-c\delta_S}e^{-c\delta_St} + \frac{\delta_{BL}^4\delta_S^2}{(1+\delta_{BL}t)^4} + \delta_R^4 \delta_S^2 e^{-ct}\bigg).
		\]
		Combining the above estimates, we obtain the desired estimate for $Y_4$.

        \noindent $\bullet$ (Estimate of $Y_5$): For $Y_5$, we have
        \begin{align*}
            \frac{C}{\delta_S}|Y_5|^2 \le \frac{C}{\delta_S} \int_{\mathbb{R}_+} |v^S_\xi|\,d\xi  \int_{\mathbb{R}_+} |v^S_\xi| \left|p(v)-p(\overline{v})-\frac{u-\overline{u}}{2C_*}\right|^2\,d\xi \le C \delta_S^{1/2}G_1.
        \end{align*}

        \noindent $\bullet$ (Estimate of $Y_6$): Using \eqref{est:a}, we obtain
		\begin{align*}
			\frac{C}{\delta_S}|Y_6|^2 &\le \frac{C}{\delta_S}\left(\int_{\R_+}a_\xi \left(\frac{|u-\overline{u}|^2}{2}+Q(v|\overline{v})\right)\,d\xi\right)^2 \\
			&\le \frac{C}{\delta_S}\|a_x\|_{L^\infty}\big(\|u-\bar{u}\|_{L^2}^2+\|v-\bar{v}\|_{L^2}^2\big)  \int_{\R_+}|a_\xi| \left(\frac{|u-\overline{u}|^2}{2}+Q(v|\overline{v})\right)\,d\xi\\
			&\le C\e_1^2(G_1 + C_1 G_S).
		\end{align*}
	\end{proof}
	\begin{lemma}\label{lem:rem2}
		For sufficiently small $\delta_0$ and $\e_1$, we have
		\begin{align}
		&\sum_{i=2}^6 B_i\le \frac{2}{100} (G_1+C_1G_S + D_{u_1}+G_{BL}+G_R), \label{est:B36}\\
		&B_7 \le \frac{4}{5}G_{BL} + \left(\frac{1}{100} + \frac{5}{16}\right)D_{u_1} + C\delta_0 (G_1 + C_1 G_S + G_R + G_{BL}) \notag,
		\end{align}
		and
		\begin{align*}
			B_8 &\le C \delta_S^{4/3}\delta_{BL}^{4/3} \left(e^{-c\delta_S t} +  \frac{1}{(1+\delta_{BL}t)^{4/3}}\right)+C\delta_S^{4/3}\delta_R^{4/3}e^{-c\delta_St}+2\e_1^6\int_{\R_+}|(p(v)-p(\bar{v}))_\xi|^2\,d\xi, \\
			B_9&\le \frac{8D_{u_1}}{100}+C\e_1^{2/3}\Bigg(\delta_R^{1/6}(1+t)^{-7/6}\log(1+\delta_{BL}t)^{4/3}+\frac{\delta_{BL}^{4/3}\delta_R^{4/3}}{(1+\delta_{BL}t)^{4/3}}\\
			&\hspace{3.2cm} +(\delta_{BL}^{4/3}+\delta_R^{4/3})\delta_S^{4/3}e^{-c\delta_St}+\frac{\delta_{BL}^{4/3}\delta_S^{4/3}}{(1+\delta_{BL}t)^{4/3}}\\
			&\hspace{3.2cm} +\delta_R^{4/3}\delta_S^{4/3}e^{-ct}+\frac{\delta_R^{1/6}(\delta_{BL}^{4/3}+\delta_R^{4/3}+\delta_S^{4/3})}{(1+t)^{7/6}}+\|u^R_{\xi\xi}\|_{L^1}^{4/3}\Bigg).
		\end{align*}
	\end{lemma}
	\begin{proof}
	\noindent $\bullet$ (Estimate of $B_2$): Using \eqref{est:a} and Young's inequality, we have
	\begin{align}
	\begin{aligned}\label{est-vdiff}
	\left|\int_{\R_+}|u^S_\xi||p-\overline{p}|^2\,d\xi\right|&\le C\sqrt{\delta_S}\int_{\R_+}a_\xi|p-\overline{p}|^2\,d\xi\\
	&\le C\sqrt{\delta_S}\int_{\R_+}a_\xi\left|(p-\overline{p})-\frac{u-\overline{u}}{2C_*}\right|^2\,d\xi+ C\sqrt{\delta_S}\int_{\R_+}a_\xi |u-\overline{u}|^2\,d\xi\\
	&\le C(G_1+C_1 G_S).
	\end{aligned}
	\end{align}
	Therefore, $B_2$ is controlled as
	\[B_2 =C(\delta_0+\e_1)\int_{\R_+}|u^S_\xi||p(v)-p(\overline{v})|^2\,d\xi \le C(\delta_0+\e_1)(G_1 + C_1 G_S).\]
    
    \noindent $\bullet$ (Estimate of $B_3$):
    Next, we estimate $B_3$ as
    \begin{align*}
    |B_3|\le C\sqrt{\delta_S}\int_{\mathbb{R}_+} |u^S_\xi| \left|p(v)-p(\overline{v})\right|^2 \, d\xi \le C \sqrt{\delta_S}(G_1+C_1 G_S).
    \end{align*}
   
    \noindent $\bullet$ (Estimate of $B_4$):
    For $B_4$, we use an algebraic inequality $|p|^3\le 8(|p-q|^3+|q|^3)$ and the interpolation inequality to obtain
    \begin{align*}
    |B_4|&\le C\int_{\mathbb{R}_+} |a_\xi| \left|p(v)-p(\overline{v})-\frac{u-\overline{u}}{2C_1}\right|^3d\xi+C\int_{\mathbb{R}_+} |a_\xi| \left|u-\overline{u}\right|^3d\xi \\
    &\le C\varepsilon_1G_1+C\lVert u-\overline{u} \rVert_{L^\infty}^2 \int_{\mathbb{R}_+} |a_\xi||u-\overline{u}| \,d\xi\\
    &\le C\varepsilon_1G_1+C\lVert u-\overline{u} \rVert_{L^2}\lVert (u-\overline{u})_\xi \rVert_{L^2}\frac{1}{\sqrt{\delta_S}}\sqrt{\int_{\mathbb{R}_+} |v^S_\xi|d\xi}\sqrt{\int_{\mathbb{R}_+} |v^S_\xi| |u-\overline{u}|^2d\xi} \\
	&\le C\e_1\big(G_1 + D_{u_1} + C_1G_S).
    \end{align*}
	
    \noindent $\bullet$ (Estimate of $B_5$): Using Young's inequality, we find that
    \begin{align*}
         B_5 \le  C \int_{\mathbb{R}_+} |a_\xi|  |u-\overline{u}| |(u-\overline{u})_\xi| \, d\xi \le \frac{1}{100}D_{u_1} + C\delta_S G_S.
    \end{align*}

    \noindent $\bullet$ (Estimate of $B_6$): We use \eqref{est:a} and Young's inequality again to get
    \begin{align*}
	\begin{aligned}
	\left|\int_{\R_+} a_\xi  \overline{u}_\xi (v-\overline{v})(u-\overline{u}) \,d\xi\right| &\le \int_{\R_+} |a_\xi|^{3/2} |u-\overline{u}|^2\,d\xi + C \int_{\R_+} | a_\xi|^{1/2} |\overline{u}_\xi|^2 |v-\overline{v}|^2\,d\xi  \\
    &\le C \delta_S^{1/4} G_S + C \delta_S^{3/4} \int_{\R_+}  |\overline{u}_\xi| |v-\overline{v}|^2\,d\xi. 
	\end{aligned}
	\end{align*}
For the second term, we observe that $\overline{u}_\xi=u^{BL}_\xi + u^R_\xi +u^S_\xi$ and therefore, 
\begin{align*}
    \int_{\R_+}  |\overline{u}_\xi| |v-\overline{v}|^2\,d\xi \le C (G_{BL} + G_R) + \int_{\R_+}  |u^S_\xi| |p(v)-p(\overline{v})|^2 \,d\xi \le C (G_{BL} + G_R + G_1+C_1 G_S).
\end{align*}
Therefore, we have
\begin{align*}
    B_6 \le C \delta_S^{1/4} (G_1 + C_1 G_S + G_{BL} + G_R ).
\end{align*}
    Combining all those estimates and using the smallness of $\delta_0$ and $\e_1$, we obtain the desired estimate \eqref{est:B36}.\\

	\noindent $\bullet$ (Estimate of $B_7$): We begin by splitting $B_7$ into two terms 
	\[
	B_7 =\int_{\R_+}a(u-\overline{u})_\xi\frac{v-\overline{v}}{v\overline{v}}(u^S_\xi + u^R_\xi)\,d\xi + \int_{\R_+}a(u-\overline{u})_\xi\frac{v-\overline{v}}{v\overline{v}}u^{BL}_\xi\,d\xi =: B_{71} + B_{72}.
	\]
	For $B_{71}$, applying Young's inequality gives
	\[
	B_{71} \le \frac{1}{100}D_{u_1} + C\delta_0 (G_1 + C_1 G_S + G_R).
	\]   
    To estimate $B_{72}$, we use Young's inequality again
    \begin{equation} \label{B_72}
        B_{72} \le  \frac{5}{16} \int_{\mathbb{R}_+}\frac{a}{v}|(u-\overline{u})_{\xi}|^2\,d\xi + \frac{4}{5}\int_{\mathbb{R}_+} \frac{a}{v\bar{v}^2}  (u^{BL}_\xi)^2 |v-\overline{v}|^2 \,d\xi. 
    \end{equation}
    To handle the second term on the right-hand side of inequality \eqref{B_72}, we define a non-negative function $g(v,\overline{v})$ as
    \begin{equation} \label{eq: g(v,ov)}
        g(v,\overline{v}) p(v|\bar{v}) = |v-\overline{v}|^2.
    \end{equation}
    Substituting \eqref{eq: g(v,ov)} into \eqref{B_72}, we have
    \begin{equation} \label{B_722}
        \frac{4}{5}\int_{\mathbb{R}_+} \frac{a}{v\bar{v}^2}  (u^{BL}_\xi)^2 |v-\overline{v}|^2 \,d\xi \le \frac{4}{5}\int_{\mathbb{R}_+} a   u^{BL}_\xi p(v|\overline{v}) \underbrace{\left[\frac{u^{BL}_\xi}{\bar{v}} \frac{g(v,\overline{v})}{v\overline{v}}\right]}_{=: I_1 \times I_2} \,d\xi .
    \end{equation}
    We now estimate the two factors inside the bracket. For $I_1$, we observe that
	\begin{equation} \label{I1}
		I_1 = \frac{u^{BL}_\xi}{\bar{v}} = \left(\frac{1}{\bar{v}} - \frac{1}{v^{BL}}\right)u^{BL}_\xi + \frac{u^{BL}_\xi}{v^{BL}} \leq C\delta_0 u^{BL}_\xi + \frac{u^{BL}_\xi}{v^{BL}}.
	\end{equation} 
    After integrating $\eqref{eq:BL}_2$ over $[\xi,+\infty)$, and using $v_- < v^{BL} < v_*$, we find that 
	\begin{align*}
		\frac{u^{BL}_\xi}{v^{BL}} &= -\sigma_- ( u^{BL}-u_*) + p(v^{BL}) - p_*\\
		&= \sigma_-^2 (v^{BL} - v_*) + p(v^{BL}) - p_* \leq p(v^{BL}) = (v^{BL})^{-\gamma}.
	\end{align*}
	Combining these estimates, we obtain
	\[
	I_1 \leq C\delta_0 u^{BL}_\xi +  (v^{BL})^{-\gamma}.
	\]
   For the second term $I_2$, let $Z := \bar{v}/v$. Using the definition of $g(v,\overline{v})$ in \eqref{eq: g(v,ov)}, we have
	\begin{align*}
		I_2 =\frac{g(v,\overline{v})}{v \overline{v}}= \frac{(v-\bar{v})^2}{\bar{v}v \big(v^{-\gamma} - \bar{v}^{-\gamma} + \gamma \bar{v}^{-\gamma -1}(v-\bar{v}) \big)}
		=\bar{v}^{\gamma} \frac{(Z-1)^2}{Z^{\gamma + 1} - (\gamma +1)Z + \gamma} \le \frac{1}{\gamma} \bar{v}^{\gamma} .
	\end{align*}
    Here, we used the algebraic inequality 
    $Z^{\gamma + 1} - (\gamma + 1)Z + \gamma \geq \gamma(Z- 1)^2$ for $Z \geq 0$.
    Moreover, since $ |\bar{v} - v^{BL}| < \delta_0$, we obtain
    \begin{equation}\label{I_2}
        	I_2 \le  \frac{1}{\gamma} (v^{BL})^\gamma +C \delta_0.
    \end{equation}
   
	Substituting \eqref{I1} and \eqref{I_2} into \eqref{B_722}, we obtain
	\begin{align*}
    \frac{4}{5}\int_{\mathbb{R}_+} \frac{a}{v\bar{v}^2}  (u^{BL}_\xi)^2 |v-\overline{v}|^2 \,d\xi \le \left(\frac{4}{5} + C \delta_0\right) G_{BL}.
	\end{align*}
	Finally, combining all results, we arrive at the following estimate for $B_7$
	\[
	B_7 \leq \frac{4}{5}G_{BL} + \left(\frac{1}{100} + \frac{5}{16}\right)D_{u_1} + C\delta_0 (G_1 + G_S + G_R + G_{BL}).
	\]
	\noindent $\bullet$ (Estimate of $B_8$): We further split $B_8$ into two terms as
	\[B_8 = C\int_{\R_+}a_\xi|v^{BL}-v_*||p(v)-p(\overline{v})|^2\,d\xi+C\int_{\R_+}a_\xi|v^R-v^*||p(v)-p(\overline{v})|^2\,d\xi=:B_{81}+B_{82}.\]
	
	Using interpolation inequality and H\"older's inequality, we have
	\begin{align*}
		B_{81}&=  C \int_{\R_+}|a_{\xi}||v^{BL} - v_*||p(v)-p(\overline{v})|^2\,d\xi \le \frac{C}{\delta_S^{1/2}}\||v_\xi^S||v^{BL}-v_*|\|_{L^2}\|p(v) - p(\bar{v})\|_{L^4}^2\\
		&\le \frac{C}{\delta_S^{1/2}}\||v_\xi^S||v^{BL}-v_*|\|_{L^2}\|p(v) - p(\bar{v})\|_{L^2}^{3/2}\|(p(v) - p(\bar{v}))_\xi\|_{L^2}^{1/2}
	\end{align*}
	
	To control the interaction term $\||v_\xi^S||v^{BL}-v_*|\|_{L^2}$, we need a delicate estimate as follows: 
	 observe that
	\begin{align*}
		\int_{\mathbb{R}_+} |v_\xi^{S}|^2|v^{BL} - v_*|^2\,d\xi &\leq \left(\int_0^{bt} + \int_{bt}^\infty \right) |v_\xi^{S}|^2|v^{BL} - v_*|^2\,d\xi \\
		&\leq \int_0^{bt} \delta_S^4 \delta_{BL}^2 e^{-C\delta_S |\xi - \sigma t - X(t) - \beta|}\,d\xi + \int_{bt}^\infty |v_\xi^{S}|^2|v^{BL} - v_*|^2\,d\xi\\
		&\leq C\delta_S^3 \delta_{BL}^2 e^{-C\delta_S t} +\int_{bt}^\infty |v_\xi^{S}|^2|v^{BL} - v_*|^2\,d\xi.
	\end{align*}
	The last term of the above inequality can be estimated as follows:
	\begin{align*}
		\int_{bt}^\infty |v_\xi^{S}|^2|v^{BL} - v_*|^2\,d\xi &= -\int_{bt}^\infty |v_\xi^S|^2 \int_{\xi}^\infty \left(|v^{BL} - v_*|^2\right)_z(z)\,dz\,d\xi\\
		&= -2 \int_{bt}^\infty (v^{BL})_z(z)(v^{BL} - v_*)(z) \int_{bt}^z |v_\xi^S|^2\,d\xi\,dz\\
		&\leq 2 \left(\int_{-\infty}^{\infty} |v_\xi^S|^2\,d\xi\right) \left(\int_{bt}^\infty  |(v^{BL})_z||v^{BL} - v_*|\,dz\right)\\
		&\leq  C\delta_S^3 \int_{bt}^\infty \frac{\delta_{BL}^3}{(1+\delta_{BL}z)^3}\,dz\le \frac{C\delta_S^3 \delta_{BL}^2}{(1+\delta_{BL}t)^2},
	\end{align*}
	and so,
	\beq\label{sdb}
	\begin{aligned}
		\int_{\mathbb{R}_+} |v_\xi^{S}|^2|v^{BL} - v_*|^2\,d\xi \le \delta_S^3 \delta_{BL}^2 e^{-C\delta_S t} + \frac{C\delta_S^3 \delta_{BL}^2}{(1+\delta_{BL}t)^2}.
	\end{aligned}
	\eeq
	Thus, from \eqref{smallness}, Young's inequality and \eqref{sdb}, we obtain
	\begin{align*}
		|B_{81}| &\le C\delta_S^{-2/3}\||v_\xi^S||v^{BL}-v_*|\|_{L^2}^{4/3} + \e_1^6\int_{\R_+}|(p(v)-p(\bar{v}))_\xi|^2\,d\xi\\
		&\le C \delta_S^{-2/3} \left(\delta_S^3 \delta_{BL}^2 e^{-C\delta_S t} + \frac{\delta_S^3 \delta_{BL}^2}{(1+\delta_{BL}t)^2}\right)^{2/3}+\e_1^6\int_{\R_+}|(p(v)-p(\bar{v}))_\xi|^2\,d\xi\\
		&\le C \delta_S^{4/3}\delta_{BL}^{4/3} \left(e^{-C\delta_S t} + \frac{1}{(1+\delta_{BL}t)^{4/3}}\right)+\e_1^6\int_{\R_+}|(p(v)-p(\bar{v}))_\xi|^2\,d\xi.
	\end{align*}
	Similarly, $B_{82}$ can be estimated as
	\begin{align*}
	|B_{82}|&\le C\delta_S^{-2/3}\||v^S_\xi||v^R-v^*|\|_{L^2}^{4/3}+\e_1^6\int_{\R_+}|(p(v)-p(\overline{v}))_\xi|^2\,d\xi\\
	&\le C\delta_S^{4/3}\delta_R^{4/3}e^{-C\delta_St}+\e_1^6\int_{\R_+}|(p(v)-p(\overline{v}))_\xi|^2\,d\xi,
	\end{align*}
	where we control the interaction term as
	\[\||v^S_\xi||v^{R}-v^*|\|^2_{L^2}\le \||v^S_\xi||v^{R}-v^*|\|_{L^\infty}\||v^S_\xi||v^{R}-v^*|\|_{L^1} \le  C\delta_S^3\delta^2_Re^{-C\delta_St},\]
	by using \eqref{rarefaction_properties}, \eqref{shock-properties} and Lemma \ref{lem:wave-interaction}. This yields the aforementioned bound on $B_8$.\\
	
	\noindent $\bullet$ (Estimate of $B_9$): Again, we recall that
	\[B_9=-\int_{\R_+}a(u-\overline{u})S\,d\xi,\quad S = S_{I1}+S_{I2}+S_R,\]
	where
	\begin{align*}
		&S_{I1}=p(\overline{v})_\xi-p(v^{BL})_\xi-p(v^R)_\xi-p(v^S)_\xi,\\
		&S_{I2}=-\left(\frac{\overline{u}_{\xi}}{\overline{v}}-\frac{u^{BL}_\xi}{v^{BL}}-\frac{u^R_\xi}{v^R}-\frac{u^S_\xi}{v^S}\right)_\xi,\quad S_R=-\left(\frac{u^R_\xi}{v^R}\right)_\xi.
	\end{align*}
	Therefore, we split $B_9$ as
	\begin{align*}
	B_9&=-\int_{\R_+}a(u-\overline{u})S_{I1}\,d\xi-\int_{\R_+}a(u-\overline{u})S_{I2}\,d\xi-\int_{\R_+}a(u-\overline{u})S_{R}\,d\xi=:B_{91}+B_{92}+B_{93}.
	\end{align*}
	\noindent $\diamond$ (Estimate of $B_{91}$): Since $\overline{v}_\xi=v^{BL}_\xi+v^R_\xi+v^S_\xi$, we split the term $S_{I1}$ as
 \begin{equation}\label{SI1 decomposition}
	\begin{aligned}
	|S_{I1}|&=|p'(\overline{v})(v^{BL}_\xi+v^R_\xi+v^S_\xi)-p'(v^{BL})v^{BL}_\xi -p'(v^R)v^R_\xi -p'(v^S) v^S_\xi|\\
	&=|(p'(\overline{v})-p'(v^{BL}))v^{BL}_\xi +(p'(\overline{v})-p'(v^R))v^R_\xi +(p'(\overline{v})-p'(v^S))v^S_\xi|\\
	&\le C|v^{BL}_\xi||v^R+v^S-v_*-v^*|+C|v^R_\xi||v^{BL}+v^S-v_*-v^*|\\
	&\quad +C|v^S_\xi||v^{BL}+v^R-v_*-v^*|\\
	&\le C|v^{BL}_\xi|(|v^R-v_*|+|v^S-v^*|)+C|v^R_\xi|(|v^{BL}-v_*|+|v^S-v^*|)\\
	&\quad +C|v^S_\xi|(|v^{BL}-v_*|+|v^R-v^*|)\\
	&=C\Big[(v^{BL}_\xi(v^R-v_*)+v^R_\xi (v_*-v^{BL}))+(v^{BL}_\xi (v^S-v^*)+v^S_\xi (v_*-v^{BL}))\\
	&\phantom{=C\Big[}+(v^R_\xi (v^S-v^*)+v^S_\xi (v^*-v^R))\Big]\\
	&=:C\big(R_1+R_2+R_3\big),
	\end{aligned}
 \end{equation}
 where we use the following properties:
 \[v^{BL}_{\xi},v^R_\xi,v^S_\xi>0,\quad v^R-v_*,\,v^*-v^R,\,v_*-v^{BL},\,v^S-v^*>0.\]
 In particular, $R_1,R_2,R_3>0$. Therefore, we further decompose $B_{91}$ as
 \[|B_{91}|\le\int_{\R_+}a|u-\overline{u}||S_{I1}|\,d\xi\le \int_{\R_+}a|u-\overline{u}|(R_1+R_2+R_3)\,d\xi.\]
 Now, using Lemma \ref{lem:wave-interaction}, the interpolation inequality, Young's inequality, and the a priori assumption, we obtain
\begin{align}\label{r1}
	\begin{aligned}
	\int_{\R_+}a|u-\overline{u}|R_1\,d\xi&\le C\|u-\overline{u}\|_{L^\infty}\int_{\R_+}\left(v^{BL}_\xi(v^R-v_*)+v^R_\xi(v_*-v^{BL})\right)\,d\xi\\
	&\le C\|(u-\overline{u})_\xi \|_{L^2}^{1/2} \|u-\overline{u}\|_{L^2}^{1/2}  \left(\delta_R^{1/8}(1+t)^{-7/8}\log(1+\delta_{BL}t)+\frac{\delta_{BL}\delta_R}{1+\delta_{BL}t}\right)\\
	&\le \frac{1}{100}D_{u_1}+C\e_1^{2/3}\left(\delta_R^{1/6}(1+t)^{-7/6}\log(1+\delta_{BL}t)^{4/3}+\frac{\delta_{BL}^{4/3}\delta_R^{4/3}}{(1+\delta_{BL}t)^{4/3}}\right).
	\end{aligned}
	\end{align}

	By the same argument, we have
	\begin{align}\label{r2}
	\begin{aligned}
		\int_{\R_+}a|u-\overline{u}|R_2\,d\xi
		&\le C\|u-\overline{u}\|_{L^\infty} \left(\delta_{BL}\delta_Se^{-c\delta_St}+\frac{\delta_{BL}\delta_S}{1+\delta_{BL}t}\right)\\
		&\le \frac{1}{100}D_{u_1}+C\e_1^{2/3}\left(\delta_{BL}^{4/3}\delta_S^{4/3}e^{-c\delta_St}+\frac{\delta_{BL}^{4/3}\delta_S^{4/3}}{(1+\delta_{BL}t)^{4/3}}\right),
	\end{aligned}
	\end{align}
	
	and
	\begin{align*}
		\int_{\R_+}a|u-\overline{u}|R_3\,d\xi
		&\le C\|u-\overline{u}\|_{L^\infty} \left(\delta_{R}\delta_S e^{-c\delta_St}+\delta_R\delta_Se^{-ct}\right)\\
		&\le \frac{1}{100}D_{u_1}+C\e_1^{2/3}\left(\delta_{R}^{4/3}\delta_{S}^{4/3}e^{-c\delta_St}+\delta_R^{4/3}\delta_S^{4/3}e^{-ct}\right).
	\end{align*} 
	
	Combining the above estimates, we obtain
	\begin{align}
	\begin{aligned}\label{est:B81}
	|B_{91}|&\le \frac{3D_{u_1}}{100}+C\e_1^{2/3}\Bigg(\delta_R^{1/6}(1+t)^{-7/6}\log(1+\delta_{BL}t)^{4/3}+\frac{\delta_{BL}^{4/3}\delta_R^{4/3}}{(1+\delta_{BL}t)^{4/3}}\\
	&\hspace{3cm} +\delta_{BL}^{4/3}\delta_S^{4/3}e^{-c\delta_St}+\frac{\delta_{BL}^{4/3}\delta_S^{4/3}}{(1+\delta_{BL}t)^{4/3}}+\delta_{R}^{4/3}\delta_{S}^{4/3}e^{-c\delta_St}+\delta_R^{4/3}\delta_S^{4/3}e^{-ct}\Bigg).
	\end{aligned}
	\end{align}
	
	\noindent $\diamond$ (Estimate of $B_{92}$): We split the term $S_{I2}$ as
	\begin{align}
	\begin{aligned}\label{SI2_estimate}
		|S_{I2}|&=\left|\left(\frac{u^{BL}_\xi+u^R_\xi+u^S_\xi}{\overline{v}}-\frac{u^{BL}_\xi}{v^{BL}}-\frac{u^R_\xi}{v^R}-\frac{u^S_\xi}{v^S}\right)_\xi\right|\\
		&\le C\Bigg[\big|(u^{BL}_{\xi\xi},u^{BL}_\xi v^{BL}_\xi)\big|(v^R-v_*,v^S-v^*)+\big|(u^{R}_{\xi\xi},v^{R}_\xi u^{R}_\xi)\big|(v^{BL}-v_*,v^S-v^*)\\
		&\hspace{1cm} + \big|(u^{S}_{\xi\xi},u^{S}_\xi v^{S}_\xi)\big|(v^{BL}-v_*,v^R-v^*)\\
		&\hspace{1cm} +|u^{BL}_\xi|(|v^R_\xi|+|v^S_\xi|)+|u^{R}_\xi|(|v^{BL}_\xi|+|v^S_\xi|)+|u^{S}_\xi|(|v^R_\xi|+|v^{BL}_\xi|)\Bigg].
	\end{aligned}
	\end{align}
	Using
	\[|u^{BL}_{\xi\xi}|\le C|u^{BL}_\xi|\sim C|v^{BL}_\xi|,\quad |u^R_{\xi\xi}|\le C|u^R_{\xi}|\sim C|v^R_{\xi}|,\quad |u^S_{\xi\xi}|\le C|u^S_{\xi}|\sim C|v^S_{\xi}|,\]
 	we have
	\begin{align*}
	|B_{92}|&\le \int_{\R_+} a|u-\overline{u}||S_{I2}|\,d\xi\\
	&\le C\int_{\R_+} a|u-\overline{u}|\Big(|v^{BL}_\xi||(v^R-v_*,v^S-v^*)|+|v^R_\xi||(v^{BL}-v_*,v^S-v^*)|\\
	&\hspace{3cm}+|v^S_\xi||(v^{BL}-v_*,v^R-v^*)|\Big)\,d\xi\\
	&\quad + C\int_{\R_+}a|u-\overline{u}|\Big(|u^{BL}_\xi|(|v^R_\xi|+|v^S_\xi|)+|u^{R}_\xi|(|v^{BL}_\xi|+|v^S_\xi|)+|u^{S}_\xi|(|v^R_\xi|+|v^{BL}_\xi|)\Big)\,d\xi\\
	&=:B_{921}+B_{922}.
	\end{align*}
	Since the terms in $B_{921}$ are exactly the same term appearing in the estimate of $B_{91}$, the same estimate in \eqref{est:B81} holds for $B_{921}$. To control $B_{922}$, we note that 
	\[u^{BL}_{\xi}\sim v^{BL}_\xi,\quad u^R_{\xi}\sim v^R_\xi,\quad u^S_\xi\sim v^S_\xi.\]
	Therefore, we again use the Sobolev interpolation inequality and Young's inequality to obtain
	\begin{align*}
	B_{922}&\le C\int_{\R_+} a|u-\overline{u}|(|v^{BL}_\xi||v^R_{\xi}|+|v^R_\xi||v^S_\xi|+|v^S_\xi||v^{BL}_\xi|)\,d\xi\\
	&\le C\|u-\overline{u}\|_{L^\infty}\int_{\R_+} (|v^{BL}_\xi||v^R_{\xi}|+|v^R_\xi||v^S_\xi|+|v^S_\xi||v^{BL}_\xi|)\,d\xi\\
	&\le \frac{D_{u_1}}{100}+C\e_1^{2/3}\left(\int_{\R_+} (|v^{BL}_\xi||v^R_{\xi}|+|v^R_\xi||v^S_\xi|+|v^S_\xi||v^{BL}_\xi|)\,d\xi\right)^{4/3}.
	\end{align*}
	On the other hand, using \eqref{BL_properties}, \eqref{shock-properties}, and \eqref{est:rarefaction}, we have
	\begin{align*}
	\int_{\R_+}|v^{BL}_\xi||v^R_\xi|\,d\xi&\le \frac{C\delta^{1/8}_R\delta^2_{BL}}{(1+t)^{7/8}}\int_{\R_+}\frac{1}{(1+\delta_{BL}\xi)^2}\,d\xi\le \frac{C\delta^{1/8}_R \delta_{BL}}{(1+t)^{7/8}},\\
	\int_{\R_+}|v^{R}_\xi||v^S_\xi|\,d\xi&\le \frac{C\delta^{1/8}_R\delta^2_{S}}{(1+t)^{7/8}}\int_{\R_+}e^{-C\delta_S|\xi-(\sigma-\sigma_-)t-X(t)-\beta|}\,d\xi\le \frac{C\delta^{1/8}_R\delta_{S}}{(1+t)^{7/8}},\\
	\int_{\R_+}|v^S_\xi||v^{BL}_\xi|\,d\xi&\le C\delta^2_S\delta_{BL}^2\int_{\R_+}\frac{e^{-C\delta_S|\xi-(\sigma-\sigma_-)t-X(t)-\beta|}}{(1+\delta_{BL}\xi)^2}\,d\xi\\
	&=C\delta^2_S\delta_{BL}^2\left(\int_0^{bt}+\int_{bt}^{\infty}\right)\frac{e^{-C\delta_S|\xi-(\sigma-\sigma_-)t-X(t)-\beta|}}{(1+\delta_{BL}\xi)^2}\,d\xi \le C\delta_S\delta_{BL} e^{-c\delta_St}+ \frac{C\delta^2_S\delta_{BL}}{1+\delta_{BL}t},
	\end{align*}
	where $b = (\sigma - \sigma_-)/2$. Thus, 
	\begin{equation}\label{est:B822}
	B_{922}\le \frac{D_{u_1}}{100}+C\e_1^{2/3}\Bigg(\frac{\delta_R^{1/6} (\delta_{BL}^{4/3}+\delta_S^{4/3})}{(1+t)^{7/6}}+C\delta_S^{4/3}\delta_{BL}^{4/3}e^{-c\delta_St}+\frac{C\delta_S^{8/3}\delta_{BL}^{4/3}}{(1+\delta_{BL}t)^{4/3}}\Bigg).
	\end{equation}
	
	\noindent $\diamond$ (Estimate of $B_{93}$):
	Finally, we estimate $B_{93}$ as
	\begin{align}
	\begin{aligned}\label{SR estimate}
	    |B_{93}|&=\left|\int_{\R_+}a(u-\overline{u})\left(\frac{u^R_\xi}{v^R}\right)_\xi\,d\xi\right|\le C\int_{\R+}|u-\overline{u}|(|u^R_{\xi\xi}|+|u^R_\xi||v^R_\xi|)\,d\xi\\
	    &\le C\|u-\overline{u}\|_{L^\infty}\left(\int_{\R_+}|u^R_{\xi\xi}|+|u^R_\xi||v^R_\xi|\,d\xi\right).
	\end{aligned}
	\end{align}
	Since
	\[\int_{\R_+}|u^R_\xi||v^R_\xi|\,d\xi \le \frac{C\delta_R^{1/8}\delta_R}{(1+t)^{7/8}},\]
	we use the same argument to derive
	\begin{equation}\label{est:B83}
	|B_{93}|\le \frac{D_{u_1}}{100}+C\e_1^{2/3}\left(\frac{\delta_R^{1/6}\delta_R^{4/3}}{(1+t)^{7/6}}+\|u^R_{\xi\xi}\|_{L^1}^{4/3}\right).\end{equation}
	Combining the estimates \eqref{est:B81}, \eqref{est:B822}, and \eqref{est:B83}, we finally derive
	\begin{align*}
    \begin{aligned}
     	B_9&\le \frac{8D_{u_1}}{100}+C\e_1^{2/3}\Bigg(\delta_R^{1/6}(1+t)^{-7/6}\log(1+\delta_{BL}t)^{4/3}+\frac{\delta_{BL}^{4/3}\delta_R^{4/3}}{(1+\delta_{BL}t)^{4/3}}\\
	&\hspace{3.2cm} +(\delta_{BL}^{4/3}+\delta_R^{4/3})\delta_S^{4/3}e^{-c\delta_St}+\frac{\delta_{BL}^{4/3}\delta_S^{4/3}}{(1+\delta_{BL}t)^{4/3}}\\
	&\hspace{3.2cm} +\delta_R^{4/3}\delta_S^{4/3}e^{-ct}+\frac{\delta_R^{1/6}(\delta_{BL}^{4/3}+\delta_R^{4/3}+\delta_S^{4/3})}{(1+t)^{7/6}}+\|u^R_{\xi\xi}\|_{L^1}^{4/3}\Bigg).
    \end{aligned}
    \end{align*}

\end{proof}
 
\subsection{Proof of Lemma \ref{lem:main}}
Now, we collect all the estimates to prove Lemma \ref{lem:main}. Combining \eqref{eq:rel_ent_2}, \eqref{est:leading}, \eqref{remaining}, and then  applying Lemma \ref{lem:rem1} and Lemma \ref{lem:rem2}, we have
\begin{align*}
&\frac{d}{dt}\int_{\R_+}a\eta(U|\overline{U})\,d\xi \le -C_1G_S-\frac{\delta_S}{4M}|\dot{X}|^2+\frac{C}{\delta_S}\sum_{i=3}^6|Y_i|^2+\sum_{i=2}^9B_i-G_1-G_{BL}-G_R-\frac{3}{8}D_{u_1}+J^{\textup{bd}}\\
&\quad \le -\frac{\delta_S}{4M}|\dot{X}|^2-\frac{1}{10}\left(C_1G_S+G_1+G_{BL}+G_R\right)-\frac{D_{u_1}}{100} + C\e_1^2 \int_{\R_+} |(p(v) - p(\bar{v}))_\xi|^2\,d\xi\\
&\qquad +C \mathcal{P}(t) + C\e_1^{2/3}\mathcal{Q}(t)+J^{\textup{bd}},
\end{align*}
where
\begin{align*}
	\mathcal{P}(t) := (\delta_{BL}^4 + \delta_R^4)\delta_S^2e^{-c\delta_St} + \frac{\delta_{BL}^4\delta_S^2}{(1+\delta_{BL}t)^4} + \delta_R^{4/3} \delta_S^{4/3} e^{-c\delta_St}+ \delta_S^{4/3}\delta_{BL}^{4/3} \left(e^{-c\delta_S t} + \frac{1}{(1+\delta_{BL}t)^{4/3}}\right),
\end{align*}
and
\begin{align}
\begin{aligned}\label{Qt}
\mathcal{Q}(t)&:=\delta_R^{1/6}(1+t)^{-7/6}\log(1+\delta_{BL}t)^{4/3}+\frac{\delta_{BL}^{4/3}\delta_R^{4/3}}{(1+\delta_{BL}t)^{4/3}}\\
&\quad +(\delta_{BL}^{4/3}+\delta_R^{4/3})\delta_S^{4/3}e^{-c\delta_St}+\frac{\delta_{BL}^{4/3}\delta_S^{4/3}}{(1+\delta_{BL}t)^{4/3}}\\
&\quad +\delta_R^{4/3}\delta_S^{4/3}e^{-ct}+\frac{\delta_R^{1/6}(\delta_{BL}^{4/3}+\delta_R^{4/3}+\delta_S^{4/3})}{(1+t)^{7/6}}+\|u^R_{\xi\xi}\|_{L^1}^{4/3}.
\end{aligned}
\end{align}

After integrating with respect to time, and using the equivalence
\[\int_{\R_+}\eta(U|\overline{U})\,d\xi\sim \|U-\overline{U}\|_{L^2}^2,\]
and \eqref{est:a}, we obtain
\begin{align*}
&\|U-\overline{U}\|_{L^2}^2+\int_0^t (\delta_S|\dot{X}|^2+G_S+G_1+G_{BL}+G_R+D_{u_1})\,ds\\
&\le C\|U_0-\overline{U}(0,\cdot)\|_{L^2}^2 + C\e^2 \int_0^t \|(p(v) - p(\bar{v}))_\xi\|_{L^2}^2\,ds +C \mathcal{P}(t)+C\e_1^{2/3}\int_0^t \mathcal{Q}(s)\,ds +C\left|\int_0^t J^{\textup{bd}}\,ds\right|.
\end{align*}
Using Lemma \ref{lem: boundary est-1}, the last term is bounded as
\[\left|\int_0^t J^{\textup{bd}}\,ds \right|\le C(\e_1+\delta_S^{1/3})e^{-C\delta_S\beta}+C\e_1^2\int_0^t \|(u-\overline{u})_{\xi\xi}\|_{L^2}^2\,ds.\]
On the other hand, using the smallness of $\delta_0$, we estimate $\mathcal{P}(t)$ and $\mathcal{Q}(t)$ as
\begin{align*}
	\int_0^t \mathcal{P}(s)\,ds &\le C\Big((\delta_{BL}^4 + \delta_R^4)\delta_S + \delta_{BL}^3 \delta_S^2 + \delta_R^{4/3}\delta_S^{1/3} +\delta_S^{1/3}\delta_{BL}^{4/3} + \delta_S^{1/3}\delta_{BL}^{1/3}\Big) \le C\delta_S^{1/3},
\end{align*}
and
\begin{align*}
\int_0^t \mathcal{Q}(s)\,ds&\le C\left(\delta^{1/6}_R+\delta^{1/3}_{BL}\delta^{4/3}_R+(\delta^{4/3}_{BL}+\delta^{4/3}_R)\delta^{1/3}_S+\delta^{1/3}_{BL}\delta^{4/3}_S+\delta_R^{1/6}(\delta^{4/3}_{BL}+\delta^{4/3}_R+\delta^{4/3}_S)\right)\\
&\quad +C\int_0^t\|u^R_{\xi\xi}\|_{L^1}^{4/3}\,ds\\
&\le C(\delta^{1/6}_R+\delta^{4/3}_S)+C\int_0^t\|u^R_{\xi\xi}\|_{L^1}^{4/3}\,ds.
\end{align*}
Finally, it follows from \eqref{rarefaction_properties} that
\[\|u^R_{\xi\xi}\|_{L^1}\le \begin{cases}
 C\delta_R\quad \mbox{if}\quad 1+t\le \delta_R^{-1},\\
 \frac{C}{1+t}\quad \mbox{if}\quad 1+t\ge \delta_R^{-1},
 \end{cases}\]
 which implies
\[\int_0^t \|u^R_{\xi\xi}\|_{L^1}^{4/3}\,ds\le C\delta_R^{1/3}.\]
 Thus, gathering all the necessary estimates, we finally derive
\begin{align*}
&\sup_{t\in[0,T]} \|U-\overline{U}\|_{L^2}^2+\int_0^t (\delta_S|\dot{X}|^2+G_S+G_1+G_{BL}+G_R+D_{u_1})\,ds\\
&\le C\|U_0-\overline{U}(0,\cdot)\|_{L^2}^2 + \delta_S^{1/3}+C\e_1^{2/3}\left(\delta^{1/6}_R+\delta^{4/3}_S\right) +C(\e_1+\delta_S^{1/3})e^{-C\delta_S\beta}\\
&\qquad +C\e_1^2 \int_0^t \|(p(v) - p(\bar{v}))_\xi\|_{L^2}^2\,ds+C\e_1^2\int_0^t \|(u-\overline{u})_{\xi\xi}\|_{L^2}^2\,ds,
\end{align*}
which completes the proof of Lemma \ref{lem:main}.

\section{Estimate on the $H^1$-perturbation}\label{sec:high-order}
\setcounter{equation}{0}
In this section, we derive the estimates on the first-order derivatives of the perturbation $U-\overline{U}$, completing the estimate on the $H^1$-perturbation. 

\subsection{Wave interaction estimates}
To control the $H^1$-perturbation, we need further control on the wave interactions as in the following lemma.

\begin{lemma}\label{L2 wave interaction}
Let $\mathcal{R}:= R_1 + R_2 + R_3$, where $R_1, R_2, R_3$ are defined as in \eqref{SI1 decomposition}. Then, there exist positive constants $c$ and $C$ such that
	\begin{align*}
		\int_{\R_+} \mathcal{R}^2\,d\xi &\le C\Bigg(\frac{\delta_R^{1/4}\delta_{BL}}{(1+t)^{7/4}}+\frac{\delta^{2/3}_R\delta^{5/3}_{BL}}{(1+t)^{4/3}}+\delta_{BL}^2\delta_S^2(\delta_{BL}+\delta_S)\left(e^{-c\delta_St}+\frac{1}{(1+\delta_{BL}t)^2}\right)\\
		&\hspace{1.5cm} +\delta_R^2\delta_S^2\left(e^{-c\delta_St}+e^{-ct}\right)\Bigg).
	\end{align*}
	
	Consequently, we have
	\[
	\int_0^t \int_{\R_+} \mathcal{R}^2\,d\xi\,ds \leq C\left(\delta_R^{1/4}\delta_{BL}+\delta^{2/3}_R\delta^{5/3}_{BL}+\delta_{BL}\delta_S(\delta_{BL}+\delta_S)+\delta_R^2\delta_S\right)\le C\delta_0^{1/4},
	\]
	and
	\begin{align*}
		\int_0^t \frac{\norm{\mathcal{R}}_{L^2(\R_+)}}{\sqrt{1+s}}\,ds\le C\delta_0^{1/8}.
	\end{align*}
\end{lemma}
\begin{proof}
	Note that
	\[
	\int_{\mathbb{R}_+}\mathcal{R}^2\,d\xi \leq C\left( \int_{\mathbb{R}_+} R_1^2\,d\xi + \int_{\mathbb{R}_+} R_2^2 \,d\xi + \int_{\mathbb{R}_+} R_3^2\,d\xi\right).
	\]
	
	\noindent $\bullet$ (Estimate of $R_1$) We first decompose $R_1$ as 
	\begin{equation}\label{R1 decomposition}
		\int_{\mathbb{R}_+} R_1^2\,d\xi \leq  2 \int_{\R_+} |v^R_\xi|^2|v_*-v^{BL}|^2\,d\xi + 2 \int_{\R_+}|v^{BL}_\xi|^2|v^R-v_*|^2\,d\xi.
	\end{equation}
	Using \eqref{est:rarefaction}, the first term of the right hand side of \eqref{R1 decomposition} is bounded as
	\[
	\int_{\R_+} |v^R_\xi|^2|v_*-v^{BL}|^2\,d\xi \leq \frac{C\delta_R^{1/4}}{(1+t)^{7/4}}\int_{\R_+} \frac{\delta_{BL}^2}{(1+\delta_{BL}\xi)^2}\,d\xi \leq \frac{C\delta_R^{1/4}\delta_{BL}}{(1+t)^{7/4}}.
	\]
	To estimate the second term of \eqref{R1 decomposition}, we first note that
	\[
	v^R - v_* = \int_0^\xi v_\xi^R\,d\xi \leq \|v_\xi^R\|_{L^3(\R_+)} \xi^{2/3}.
	\]
	Thus, we obtain 
	\begin{align*}
		\int_{\R_+} |v_\xi^{BL}|^2|v^R-v_*|^2\,d\xi &\leq \|v_\xi^R\|_{L^3(\R_+)}^2 \int_{\R_+} \xi^{4/3} |v_\xi^{BL}|^2\,d\xi\\
		&\leq \frac{C\delta_R^{2/3}}{(1+t)^{4/3}} \int_{\R_+} \xi^{4/3} \frac{\delta_{BL}^4}{(1+\delta_{BL}\xi)^4}\,d\xi\le C\frac{\delta_R^{2/3}\delta_{BL}^{5/3}}{(1+t)^{4/3}}.
	\end{align*}
	Therefore, 
	\[
	\int_{\mathbb{R}_+} R_1^2\,d\xi \leq \frac{C\delta_R^{1/4}\delta_{BL}}{(1+t)^{7/4}} + \frac{C\delta_R^{2/3}\delta_{BL}^{5/3}}{(1+t)^{4/3}}.
	\]
	$\bullet$ (Estimate of $R_2$)
	Again, we can decompose $R_2$ as
	\begin{equation}\label{R2 decomposition}
		\int_{\mathbb{R}_+} R_2^2\,d\xi \leq  2\int_{\mathbb{R}_+} |v_\xi^{BL}|^2|v^S - v^*|^2\,d\xi +  2\int_{\mathbb{R}_+} |v_\xi^{S}|^2|v^{BL} - v_*|^2\,d\xi=:R_{21}+R_{22}.
	\end{equation}
	We estimate $R_{21}$ as
	\begin{align*}
		\int_{\mathbb{R}_+} |v_\xi^{BL}|^2|v^S - v^*|^2\,d\xi &= \left(\int_0^{bt} + \int_{bt}^\infty \right) |v_\xi^{BL}|^2|v^S - v^*|^2\,d\xi \\
		&\leq \int_0^{bt} |v_\xi^{BL}|^2|v^S - v^*|^2\,d\xi + 
		\int_{bt}^\infty C\delta_S^2 |v_\xi^{BL}|^2\,d\xi\\
		&\leq\int_0^{bt} |v_\xi^{BL}|^2|v^S - v^*|^2\,d\xi + C\delta_S^2 \delta_{BL}^4 \int_{bt}^\infty \frac{1}{(1+\delta_{BL}\xi)^4}\,d\xi\\
		&\leq \int_0^{bt} |v_\xi^{BL}|^2|v^S - v^*|^2\,d\xi + \frac{C \delta_S^2 \delta_{BL}^3}{(1+\delta_{BL}t)^3},
	\end{align*}
	where $b  = (\sigma - \sigma_-)/2$.
	
	On the other hand, we use the Fubini theorem to obtain
	\begin{align*}
		\int_0^{bt} |v_\xi^{BL}|^2|v^S - v^*|^2\,d\xi &= \int_0^{bt} |v_\xi^{BL}(\xi)|^2 \int_{-\infty}^{\xi}(|(v^S - v^*)(z)|^2)_z\,dz\,d\xi\\
		&\leq 2 \int_{-\infty}^{bt}\int_0^{bt}|(v^S - v^*)(z)||(v^S - v^*)_z(z)| |v_\xi^{BL}(\xi)|^2\,d\xi\,dz\\
		&\leq 2 \left(\int_0^\infty |v_\xi^{BL}(\xi)|^2\,d\xi\right)\left( \int_{-\infty}^{bt}|(v^S - v^*)(z)||(v^S - v^*)_z(z)|\,dz\right)\\
		&\leq C\delta_{BL}^3 \delta_S^3 \int_{-\infty}^{bt} e^{-C\delta_S|z - (\sigma - \sigma_-)t - X(t) - \beta|}\,dz\leq C\delta_{BL}^3\delta_S^2e^{-C\delta_S t}.
	\end{align*}
	Thus, we combine the above two estimates to derive
	\begin{equation}\label{R_21}
		R_{21} \leq  C\delta_{BL}^3\delta_S^2e^{-C\delta_S t} + \frac{C\delta_S^2 \delta_{BL}^3}{(1+\delta_{BL}t)^3}.
	\end{equation}
	The estimate for $R_{22}$ follows from \eqref{sdb}:
	\begin{equation}\label{R_22}
		R_{22}\leq C\delta_S^3 \delta_{BL}^2 e^{-C\delta_S t} + \frac{C\delta_S^3 \delta_{BL}^2}{(1+\delta_{BL}t)^2}.
	\end{equation}
	Therefore, we substitute \eqref{R_21} and \eqref{R_22} to \eqref{R2 decomposition} to get
	\begin{align*}
		\int_{\mathbb{R}_+} R_2^2\,d\xi &\leq  C(\delta^3_{BL}\delta^2_S+\delta_{BL}^2\delta_S^3)\left(e^{-C\delta_S t}+\frac{1}{(1+\delta_{BL}t)^2}\right).
	\end{align*}
	
	\noindent $\bullet$ (Estimate of $R_3$) Finally, we estimate $R_3$ as
	\begin{align*}
		\int_{\mathbb{R}_+} R_3^2\,d\xi &\leq \left\||v_\xi^R||v^S - v^*| + |v_\xi^S||v^R - v^*|\right\|_{L^\infty}\int_{\mathbb{R}_+} R_3\,d\xi\\
		&\leq C\delta_S \delta_R  \left(\delta_R\delta_S e^{-C\delta_St}+\delta_R\delta_Se^{-Ct}\right),
	\end{align*}
	where we use Lemma \ref{lem:wave-interaction} (3) in the last inequality. Therefore, combining the estimates for $R_1$, $R_2$, and $R_3$, we have
	\begin{align}
		\begin{aligned}\label{SI1_L2}
		\int_{\R_+}\mathcal{R}^2\,d\xi &\le C\left(\int_{\R_+}R_1^2\,d\xi +\int_{\R_+}R_2^2\,d\xi+\int_{\R_+}R_3^2\,d\xi\right)\\
		&\le C\Bigg(\frac{\delta_R^{1/4}\delta_{BL}}{(1+t)^{7/4}}+\frac{C\delta^{2/3}_R\delta^{5/3}_{BL}}{(1+t)^{4/3}}+\delta_{BL}^2\delta_S^2(\delta_{BL}+\delta_S)\left(e^{-C\delta_St}+\frac{1}{(1+\delta_{BL}t)^2}\right)\\
		&\hspace{1.5cm} +\delta_R^2\delta_S^2\left(e^{-C\delta_St}+e^{-Ct}\right)\Bigg).
		\end{aligned}
	\end{align}
	
	In particular, we have
	\[
	\int_0^t \int_{\R_+}\mathcal{R}^2\,d\xi\,ds \leq C\left(\delta_R^{1/4}\delta_{BL}+\delta^{2/3}_R\delta^{5/3}_{BL}+\delta_{BL}\delta_S(\delta_{BL}+\delta_S)+\delta_R^2\delta_S\right)\le C\delta_0^{1/4}.
	\]
	Moreover, it follows from \eqref{SI1_L2} that
	\[
	\|\mathcal{R}\|_{L^2(\R_+)} \leq \frac{C(\delta_R^{1/8}\delta_{BL}^{1/2} + \delta^{1/3}_R\delta^{5/6}_{BL})}{(1+t)^{2/3}}+\frac{\delta_{BL}\delta_S\sqrt{\delta_{BL}+\delta_S}}{1+\delta_{BL}t}+C\delta_S (\delta_{BL}+\delta_{R})e^{-C\delta_S t} + C\delta_S \delta_R e^{-Ct}.
	\]
	Therefore,
	\begin{align*}
		\int_0^t \frac{\norm{\mathcal{R}}_{L^2(\R_+)}}{\sqrt{1+s}}\,ds
		&\leq C \Bigg(\int_0^t \frac{\delta_R^{1/8}\delta^{1/2}_{BL}+\delta^{1/3}_R\delta^{5/6}_{BL}}{(1+s)^{7/6}}\,ds +\int_0^t\frac{\delta_{BL}\delta_S\sqrt{\delta_{BL}+\delta_S}}{(1+\delta_{BL}s)^{3/2}}\,ds\\
		&\hspace{1.5cm} +\int_0^t \delta_S(\delta_{BL}+\delta_R) e^{-C\delta_S s}\,ds +\delta_S\delta_R\int_0^t e^{-Cs}\,ds \Bigg)\\
		&\leq  C\left(\delta_R^{1/8} \delta^{1/2}_{BL}+\delta^{1/3}_R\delta^{5/6}_{BL}+\delta_S\sqrt{\delta_{BL}+\delta_S}+(\delta_{BL}+\delta_R)+\delta_S\delta_R\right)\\
		&\le C\delta_0^{1/8}.
	\end{align*}
\end{proof}

\subsection{$H^1$-estimate for $v-\overline{v}$}
We first derive the $H^1$-estimate for the $v$-perturbation. As we will see below, we need to control the term such as $\|\pa_\xi(v-\overline{v})\|_{L^2}$. To this end, we introduce an effective velocity $h:=u-(\ln v)_\xi$ and consider the NS system in terms of $(v,h)$-variables:
 \begin{equation} \label{NS-h}
     \begin{aligned}
         &v_t-\sigma_- v_\xi -h_\xi=(\ln v)_{\xi\xi},\\
         &h_t-\sigma_-h_\xi +p(v)_\xi =0,
     \end{aligned}
 \end{equation}
so that the parabolic term is now on the $v$-variable. Then, we rewrite the BL solution and viscous shock wave in terms of $v$ and $h$. Precisely, we set $h^{BL}:=u^{BL}-\ln( v^{BL} )_\xi.$ Then, it is straightforward to verify that $(v^{BL},h^{BL})$ satisfies
\begin{equation*} 
    \begin{aligned}
        &-\sigma_- v^{BL}_\xi-h^{BL}_\xi = (\ln v^{BL} )_{\xi \xi},\\
        &-\sigma_- h^{BL}_\xi+p(v^{BL})_\xi=0.
    \end{aligned}
\end{equation*}
Similarly, once we define $h^{S}:=u^{S}-\ln( v^{S} )_\xi$, $(v^{S},h^{S})$ satisfies
\begin{equation*}
    \begin{aligned}
        &-\sigma v^{S}_\xi-h^{S}_\xi = (\ln v^{S} )_{\xi \xi},\\
        &-\sigma h^{S}_\xi+p(v^{S})_\xi=0.
    \end{aligned}
\end{equation*}

Using newly defined elementary waves, we define the superposition $\overline{h}$ as
\begin{equation*}
    \overline{h}(t,\xi):=h^{BL}(\xi)+u^R(t,\xi)+h^S(\xi-(\sigma-\sigma_-)t-X(t)-\beta)-u_*-u^*.
\end{equation*}
Then, the superposition $(\overline{v},\overline{h})$ satisfies
\begin{equation}
    \begin{aligned}\label{eq:compositve_wave_effective_vel}
        &\overline{v}_t-\sigma_-\overline{v}_\xi-h_\xi=-\dot{X}v^S_\xi + (\ln \overline{v} )_{\xi\xi} + E \\ 
        &\overline{h}_t-\sigma_-\overline{h}_\xi+p(\overline{v})_\xi =-\dot{X}h^S_\xi+S_{I1},
    \end{aligned}
\end{equation}
where $S_{I1}$ is the same term as in \eqref{SI1SI2} and $E:=E_I+E_R$ is defined as
\begin{align*}
    E_I&:= -\left(\ln(\overline{v})-\ln (v^{BL})-\ln(v^R)-\ln(v^S) \right)_{\xi\xi}, \\
    E_R&:=-(\ln(v^R))_{\xi\xi}.
\end{align*}
The system \eqref{NS-h} can still be written as a general hyperbolic system of the form
\[U_t+A(U)_\xi=\partial_\xi (M(U)\partial_\xi \nabla \eta(U)),\]
where now the conserved quantity $U$, flux $A$, the diffusion matrix $M$ and the entropy $\eta$ are 
\[U:=\begin{pmatrix} v \\ h \end{pmatrix}, \quad A(U):=\begin{pmatrix} -\sigma_- v - h \\ -\sigma_- h +p(v) \end{pmatrix}, \quad 
M(U):=\begin{pmatrix} \frac{1}{\gamma p(v)} &0 \\ 0 &0 \end{pmatrix},\quad \eta(U):=\frac{|h|^2}{2}+Q(v). \]
Then, the relative entropy, relative flux, and relative entropy flux for this system are given by
\begin{align*}
    &\eta(U|\overline{U}):= \frac{|h-\overline{h}|^2}{2}+Q(v|\overline{v}),\quad A(U|\overline{U}):=\begin{pmatrix}
        0 \\p(v|\overline{v})
    \end{pmatrix},\\
	& G(U; \overline{U})=(p(v)-p(\overline{v}))(h-\overline{h})-\sigma_- \eta(U|\overline{U}).
\end{align*}
Similarly, the equation \eqref{eq:compositve_wave_effective_vel} for the superposition $\overline{U}:=(\overline{v},\overline{h})$ can be written as
\[\overline{U}_t+A(\overline{U})_\xi=\partial_\xi (M(\overline{U})\partial_\xi \nabla \eta(\overline{U}))-\dot{X}\partial_\xi U^S+\begin{pmatrix}
	E\\S_{I1}
\end{pmatrix}.\]

Before we proceed further, we verify that the good terms $G_1$, $C_1G_S$, $G_{R}$ and $G_{BL}$ yield another form of the good term, which is useful in the later analysis.

\begin{lemma}\label{lem:G_v}
	There exists a constant $C$ such that
	\[G_v:=\int_{\R_+}|\overline{u}_\xi||v-\overline{v}|^2\,d\xi\le C(G_1+C_1G_S+G_{BL}+G_R).\]
\end{lemma}
\begin{proof}
	Since $\overline{u}_{\xi}=u^{BL}_\xi+u^R_\xi+u^S_\xi$, we use the same estimate as in \eqref{est-vdiff} to obtain
	\begin{align*}
	G_v&=\int_{\R_+}|\overline{u}_\xi||v-\overline{v}|^2\,d\xi\le\int_{\R_+}(|u^{BL}_\xi|+|u^R_\xi|+|u^S_\xi|)|v-\overline{v}|^2\,d\xi\\
	&\le C\int_{\R_+}au^{BL}_\xi p(v|\overline{v})\,d\xi + C\int_{\R_+}au^{R}_{\xi}p(v|\overline{v})\,d\xi +C\int_{\R_+}|u^S_{\xi}||v-\overline{v}|^2\,d\xi\\
	&\le CG_{BL}+CG_R+C(G_1+C_1G_S).
	\end{align*}
\end{proof}

With the aid of \eqref{NS-h}, we derive the following estimate on the $H^1$-perturbation for $v$.

\begin{lemma}\label{lem:vhigh}
	Under the assumption of Proposition \ref{apriori-estimate}, we have
\begin{align}
\begin{aligned}\label{higher v}
    &\sup_{0\le t\le T}\left(\|v-\overline{v}\|_{H^1}^2+\|u-\overline{u}\|_{L^2}^2\right)+\int_0^t(\delta_S|\dot{X}|^2+G_S+G_1+G_{BL}+G_R+D_{v_1}+D_{u_1})\,ds\\
    &\, \, \le C\left(\|v(0,\cdot)-\overline{v}(0,\cdot)\|_{H^1}^2+\|u(0,\cdot)-\overline{u}(0,\cdot)\|_{L^2}^2\right)+Ce^{-C\delta_S \beta} + C\delta_0^{1/6} + \frac{1}{4C_2}\int_0^t D_{u_2}\,ds.
\end{aligned}
\end{align}
Here, $D_{v_1}$ and $D_{u_2}$ are defined as 
\[
D_{v_1}\coloneqq \int_{\R_+} \frac{1}{\gamma p(v)}|(p(v) - p(\overline{v}))_\xi|^2\,d\xi, \quad D_{u_2}\coloneqq \int_{\R_+} \frac{1}{v}|(u - \overline{u})_{\xi \xi}|^2\,d\xi,
\]
and the $C_2>0$ is the specific constant defined in \eqref{u_H1}.
\end{lemma}
\begin{proof}
	Following the same estimate for the relative entropy with the weight $a(t,x)\equiv 1$, (see for example \cite[Lemma 5.1]{HKKL_pre2}), we obtain
    \begin{equation}\label{est_rel_ent_H1}
    	\frac{d}{dt}\int_{\mathbb{R}_+} \eta(U(t,\xi)|\overline{U}(t,\xi))\, d \xi = \dot{X} \mathcal{Y}+\sum_{i=1}^6 \mathcal{I}_i,
    \end{equation}
    where
    \begin{align*}
    \mathcal{Y}(U) &:=\int_{\mathbb{R}_+} \nabla^2 \eta(\overline{U}) (U^S_\xi) (U-\overline{U})d\xi \\
    &=\int_{\mathbb{R}_+} h^S_\xi (h-\overline{h}) d \xi - \int_{\mathbb{R}_+} p'(\overline{v})v^S_\xi (v-\overline{v}) d\xi=:\mathcal{Y}_1+\mathcal{Y}_2, \\ 
        \mathcal{I}_1 &:=-\int_{\mathbb{R}_+} \partial_\xi G(U;\overline{U})d\xi=-\int_{\mathbb{R}_+} \partial_\xi \left((p(v)-p(\overline{v}))(h-\overline{h})-\sigma_- \eta(U|\overline{U}) \right) d\xi \\
        & = \left[(p(v)-p(\overline{v}))(h-\overline{h})\right]_{\xi=0} - \left[\frac{\sigma_-}{2}|h-\overline{h}|^2 \right]_{\xi=0} - \left[\sigma_-Q(v|\overline{v})  \right]_{\xi=0}\\
        & =: \mathcal{I}_{11} + \mathcal{I}_{12} + \mathcal{I}_{13}, \\
        \mathcal{I}_2 &:=-\int_{\mathbb{R}_+} \partial_\xi \nabla \eta (\overline{U})A(U|\overline{U})d\xi=-\int_{\mathbb{R}_+} \overline{h}_\xi p(v|\overline{v}) d\xi,\\
        \mathcal{I}_3 &:=\int_{\mathbb{R}_+} \left(\nabla \eta (U)-\nabla \eta (\overline{U})\right)\partial_\xi \left(M(U) \partial_\xi \left(\nabla \eta (U)-\nabla \eta(\overline{U}) \right) \right) d \xi\\
        &=\int_{\mathbb{R}_+}(p(v)-p(\overline{v}))\partial_\xi \left( \frac{1}{\gamma p(v)}\partial_\xi (p(v)-p(\overline{v})) \right) d\xi \\ 
        &=\left[-  \frac{1}{\gamma p(v)} (p(v)-p(\overline{v})) \partial_\xi (p(v)-p(\overline{v})) \right]_{\xi=0}-\int_{\mathbb{R}_+}  \frac{1}{\gamma p(v)} |(p(v)-p(\overline{v}))_\xi|^2 d\xi\\
        &=: \mathcal{I}_{31} -D_{v_1},\\
        \mathcal{I}_4 &:=\int_{\mathbb{R}_+} \left(\nabla \eta (U)-\nabla \eta (\overline{U}) \right)\partial_\xi \left(\left(M(U)-M(\overline{U})\right) \partial_\xi \nabla \eta(\overline{U})\right) d \xi \\
        &=\int_{\mathbb{R}_+} (p(v)-p(\overline{v}))\partial_\xi \left( \frac{p(\overline{v})-p(v)}{\gamma p(v)p(\overline{v})}\partial_\xi p(\overline{v}) \right) d\xi\\
        &=\left[ (p(v)-p(\overline{v}))^2  \frac{\partial_\xi p(\overline{v})}{\gamma p(v)p(\overline{v})} \right]_{\xi=0}+\int_{\mathbb{R}_+} \partial_\xi(p(v)-p(\overline{v}))  \frac{p(v)-p(\overline{v})}{\gamma p(v)p(\overline{v})}\partial_\xi p(\overline{v}) d\xi\\
        &=:\mathcal{I}_{41}+\mathcal{I}_{42},\\
        \mathcal{I}_5 &:=\int_{\mathbb{R}_+}(\nabla \eta)(U|\overline{U}) \partial_\xi (M(\overline{U})\partial_\xi \nabla \eta (\overline{U})) d\xi =-\int_{\mathbb{R}_+} p(v|\overline{v})(\ln \overline{v})_{\xi\xi} d\xi ,\\
        \mathcal{I}_6 &:=-\int_{\mathbb{R}_+} \nabla^2 \eta(\overline{U})(U-\overline{U})\begin{pmatrix}
            E \\ S_{I1}
        \end{pmatrix}d\xi =\int_{\mathbb{R}_+}p'(\overline{v})(v-\overline{v})E d\xi -\int_{\mathbb{R}_+} (h-\overline{h})S_{I1} d\xi\\
    	&=:\mathcal{I}_{61}+\mathcal{I}_{62}.
    \end{align*}
\noindent $\bullet$ (Estimates of $\dot{X}\mathcal{Y}$):
Using the similar argument in \eqref{SI1 decomposition}, we get
\begin{equation}\label{h decompose}
	\begin{aligned}
		|h-\overline{h}| &= \Big| (u-\overline{u}) + (\ln (v) -\ln (v^{BL}) - \ln (v^S) - \ln (v^{R}) )_\xi  +  \ln (v^R))_\xi\Big| \\ 
		&=   \Big|(u-\overline{u}) + (\ln (v) -\ln(\overline{v}) )_\xi + (\ln (\overline{v}) -\ln (v^{BL}) - \ln (v^S) - \ln (v^{R}) )_\xi  + \ln (v^R)_\xi\Big|\\
            &=  \bigg|(u-\overline{u}) + \frac{1}{v}(v-\overline{v})_\xi + \overline{v}_\xi \left(\frac{1}{v} - \frac{1}{\overline{v}}\right) \\
            &\qquad + (\ln (\overline{v}) -\ln (v^{BL}) - \ln (v^S) - \ln (v^{R}) )_\xi  + \ln (v^R)_\xi\bigg|\\
		&\le   |u-\overline{u}| +  C \big( |(v-\overline{v})_\xi| + |\overline{v}_\xi| |v-\overline{v}|\big) + C\mathcal{R} + C|v^R_\xi|,
	\end{aligned}
\end{equation}
where $\mathcal{R}$ is as in Lemma \ref{L2 wave interaction}.
Moreover, using $|h^S_\xi|\le |u^S_\xi|+C|v^S_\xi|+|v^S_{\xi\xi}|\le C|u^S_\xi|$, we estimate $\mathcal{Y}_1$ as
\begin{align*}
    |\mathcal{Y}_1| 
    &\le C \int_{\mathbb{R}_+} |u^S_\xi| \left( |u-\overline{u}| + |(v-\overline{v})_\xi| +|\overline{v}_\xi| |v-\overline{v}| +\mathcal{R}+|v^R_\xi| \right) d \xi \\
    &\le C \sqrt{\delta_S} \left[  \sqrt{G_{S}} +\delta_S \left(   \|(v-\overline{v})_\xi\|_{L^2(\mathbb{R}_+)}+\sqrt{G_v} + \|\mathcal{R} \|_{L^2(\mathbb{R}_+)} \right) \right] + \frac{C \delta_S}{1+t},
\end{align*}
where we use $\|v^R_\xi\|_{L^\infty}\le\frac{C}{1+t}$. Thus, we use Cauchy-Schwarz inequality to derive
\[|\dot{X}||\mathcal{Y}_1| \le  \frac{\delta_S}{4} |\dot{X}|^2 + CG_S + C\delta_S^2 \left( G_v +  D_{v_1} +\| \mathcal{R} \|^2_{L^2(\mathbb{R}_+)} \right)+\frac{C \delta_S}{(1+t)^2}. \]
Similarly, we obtain
\begin{align*}
    | \dot{X} | |\mathcal{Y}_2| \le \frac{\delta_S}{4} |\dot{X}|^2 + C G_S + C \sqrt{\delta_S} G_1.
\end{align*}

\noindent $\bullet$ (Estimates of $\mathcal{I}_1$, $\mathcal{I}_{31},$ and $\mathcal{I}_{41}$):
We first observe that
\begin{align*}
    \mathcal{I}_{11} + \mathcal{I}_{31} = \left[ (p(v)-p(\overline{v})) (u-\overline{u})\right]_{\xi=0} + \left[  (p(v)-p(\overline{v})) \left( \frac{v^{BL}_\xi}{v^{BL}}+\frac{v^S_\xi}{v^S}-\frac{\overline{v}^{-\gamma-1}}{v^{-\gamma}} \overline{v}_\xi\right)\right]_{\xi=0}.
\end{align*}
Then, we use \eqref{est_boundary} to have
\begin{align*}
\left[ (p(v)-p(\overline{v})) (u-\overline{u})\right]_{\xi=0} \le C\|p(v)-p(\overline{v})\|_{L^\infty} |u_- -\overline{u}(t,0)|  \le C \varepsilon_1 \delta_S e^{-C \delta_S t} e^{-C \delta_S \beta},
\end{align*}
and a similar estimate with $|v_--\overline{v}(t,0)|$ yields
\begin{align*}
	\left[  (p(v)-p(\overline{v})) \left( \frac{v^{BL}_\xi}{v^{BL}}+\frac{v^S_\xi}{v^S}-\frac{\overline{v}^{-\gamma-1}}{v^{-\gamma}} \overline{v}_\xi\right)   \right]_{\xi=0} &\le C |\overline{v}_\xi (t,0)| |v_- - \overline{v}(t,0)| \\ 
    &\le  C  (\delta_{BL}^2+\delta_R + \delta_S^2) \delta_S e^{-C \delta_S t} e^{-C \delta_S \beta},
\end{align*}
where we use $\overline{v}_\xi=v^{BL}_\xi+v^R_\xi+v^S_\xi$. 
Therefore,
\[|\mathcal{I}_{11}+\mathcal{I}_{31}|\le  C\delta_Se^{-C\delta_S t}e^{-C\delta_S\beta}.\]
Next, since
\begin{align*}
	&|\mathcal{I}_{13}|\le CQ(v(t,0)|\overline{v}(t,0))\le C|v_--\overline{v}(t,0)|^2,\\
	&|\mathcal{I}_{41}|\le C|\overline{v}_\xi(t,0)||p(v(t,0))-p(\overline{v}(t,0))|^2\le C|\overline{v}_\xi(t,0)||v_--\overline{v}(t,0)|^2,
\end{align*}
we use the same estimate as in \eqref{est_boundary} to derive
\begin{align*}
\mathcal{I}_{13} &\le C \delta_S^2 e^{-C \delta_S t} e^{-C \delta_S \beta},\\
\mathcal{I}_{41} &\le  C  (\delta_{BL}^2+\delta_R + \delta_S^2) \delta_S^2 e^{-C \delta_S t} e^{-C \delta_S \beta}\le C\delta_S^2e^{-C\delta_St}e^{-C\delta_S\beta}.
\end{align*}
Finally, to obtain the estimate on $\mathcal{I}_{12}$, we first note that
\begin{align*}
    \left[| \mathcal{R} |^2 \right]_{\xi=0} &\le C \Big[|v^{BL}_\xi |^2(|v^R-v_*|^2+|v^S-v^*|^2)+|v^R_\xi|^2(|v^{BL}-v_*|^2+|v^S-v^*|^2)\\
    &\qquad +|v^S_\xi|^2(|v^{BL}-v_*|^2+|v^R-v^*|^2)\Big]_{\xi=0} \\
    &\le C \left(|v^{BL}_\xi(t,0) |^2 +|v^R_\xi (t,0)|^2 \right) |v^S (t,0)-v^*|^2 + C\delta_{BL}^2 |v^R_\xi (t,0)|^2 + C(\delta_{BL}^2 + \delta_R^2 )  |v^S_\xi (t,0)|^2  \\
    &\le C \left( \delta_{BL}^4 + \delta_R^2 \right) \delta_S^2 e^{-C\delta_S t} e^{-C \delta_S \beta}  + \frac{ C \delta_R^{1/4}\delta_{BL}^2}{(1+t)^{7/4}} + C (\delta_{BL}^2 + \delta_R^2 ) \delta_S^4 e^{-C\delta_S t} e^{-C \delta_S \beta}.
\end{align*}
Thus, using \eqref{h decompose}, $\mathcal{I}_{12}$ can be bounded as
\begin{align*}
    \mathcal{I}_{12} & \le  C \left[|u-\overline{u}|^2 +|\overline{v}_\xi|^2 |v-\overline{v}|^2 \right]_{\xi=0}  +  C  \left[ |(v-\overline{v})_\xi|^2 \right]_{\xi=0}  + C\left[| \mathcal{R} |^2 \right]_{\xi=0} + C\left[ |v^R_\xi|^2 \right]_{\xi=0} \\ 
    &\le   C  \delta_S^2 e^{-C \delta_S t} e^{-C \delta_S \beta} +  C  \left[ |(v-\overline{v})_\xi|^2 \right]_{\xi=0}  + \frac{ C \delta_R^{1/4}\delta_{BL}^2}{(1+t)^{7/4}} + \frac{C \delta_R^{1/2}}{(1+t)^{3/2}}. 
\end{align*}
\noindent $\bullet$ (Estimates of $\mathcal{I}_2$, $\mathcal{I}_{42}$, $\mathcal{I}_5$, and $\mathcal{I}_6$): Since 
\begin{align*}
    |\overline{h}_\xi| \le C(|\overline{u}_\xi| + |\ln (v^{BL})_{\xi \xi}| + |\ln (v^S)_{\xi\xi}| ) \le C|\overline{u}|_\xi,
\end{align*}
 we estimate $\mathcal{I}_2$ as
 \begin{align*}
     \mathcal{I}_2 \le C \int_{\mathbb{R}_+} |\overline{u}_\xi| |v-\overline{v}|^2 \, d \xi \le C G_v.
 \end{align*}
 Similarly, we also obtain 
 \begin{align*}
     \mathcal{I}_5 \le C \int_{\mathbb{R}_+} |\overline{u}_\xi| |v-\overline{v}|^2 \, d \xi \le C G_v.
 \end{align*}
On the other hand, we estimate $\mathcal{I}_{42}$ as 
\begin{align*}
    \mathcal{I}_{42} \le \frac{D_{v_1}}{100}  + C \int_{\mathbb{R}_+} |\overline{v}_\xi| |v-\overline{v}|^2 \, d \xi \le \frac{D_{v_1}}{100}  + C G_{v}.
\end{align*}
Finally, we control the wave interaction terms $\mathcal{I}_{61}$ and $\mathcal{I}_{62}$.

To this end, observe that
\begin{align*}
	E_I=-\left(\frac{\overline{v}_\xi}{\overline{v}}-\frac{v^{BL}_\xi}{v^{BL}}-\frac{v^R_\xi}{v^R}-\frac{v^S_\xi}{v^S}\right)_\xi,\quad E_R=-\left(\frac{v^R_\xi}{v^R}\right)_\xi
\end{align*}
and
\begin{align*}
    &|u_{\xi\xi}^{BL}| \sim |v_{\xi\xi}^{BL}|, \quad |u_{\xi}^{BL}| \sim |v_{\xi}^{BL}|,\quad|u_{\xi\xi}^{S}| \sim |v_{\xi\xi}^{S}|, \quad |u_{\xi}^{S}| \sim |v_{\xi}^{S}|,\\
    &|u_{\xi\xi}^{R}| \sim |v_{\xi\xi}^{R}|, \quad |u_{\xi}^{R}| \sim |v_{\xi}^{R}|,\quad |\overline{u}_{\xi\xi}| \sim |\overline{v}_{\xi\xi}|, \quad |\overline{u}_{\xi}| \sim |\overline{v}_{\xi}|.
\end{align*}
This implies that
\[
|E| \sim |S_{I2}| + |S_R|.
\]
Therefore the estimate of $|\mathcal{I}_{61}|$ is the same as what we have done for the estimate of $B_{92}$ and $B_{93}$ in the proof of Lemma \ref{lem:rem2}. Thus, we can conclude that

\[
|\mathcal{I}_{61}| \leq \int_{\R_+} |p'(\overline{v})(v-\overline{v})E|\,d\xi \leq \frac{8}{100}D_{v_1} + C\e_1^{2/3}\mathcal{Q}(t),
\]
where $\mathcal{Q}(t)$ is the same term defined in \eqref{Qt}. To control $\mathcal{I}_{62}$, we use \eqref{h decompose} to have
\begin{align*}
    |\mathcal{I}_{62}|\leq\left|\int_{\R_+} (h-\overline{h})\mathcal{R}\,d\xi \right| &\leq C \int_{\R_+} |(u - \overline{u})||\mathcal{R}|\,d\xi + C\int_{\R_+} |\overline{v}_\xi||(v-\overline{v})||\mathcal{R}|\,d\xi\\
    &\quad + C \int_{\R_+} |(v-\overline{v})_\xi| |\mathcal{R}|\,d\xi + C\int_{\R_+} |\mathcal{R}|^2\,d\xi + C\int_{\R_+} |v_\xi^R||\mathcal{R}|\,d\xi.
\end{align*}
Note that the first term of the right-hand side is essentially the same term as $B_9$ in Lemma \ref{lem:rem2}. Therefore, we use H\"older's inequality and Young's inequality to obtain
\begin{align}\label{i62}
\begin{aligned}
     \left|\mathcal{I}_{62}\right| &\leq  \Big(C\e_1^{2/3}\mathcal{Q}(t) + \frac{8}{100}D_{u_1}\Big)+ \frac{D_{v_1}}{100}  + C\int_{\R_+} |\mathcal{R}|^2\,d\xi + C\frac{\delta_R^{1/2}}{\sqrt{1+t}}\norm{\mathcal{R}}_{L^2}+CG_v.
    \end{aligned}
\end{align}
We now combine all the estimates and substitute them into \eqref{est_rel_ent_H1} to obtain
\begin{align*}
\frac{d}{dt}&\int_{\R_+}\eta(U|\overline{U})\,d\xi+\frac{D_{v_1}}{2}\\
&\le \frac{\delta_S}{2}|\dot{X}|^2+C\sqrt{\delta_S}G_1+CG_S+CG_v + CD_{u_1}\\
&\quad +\frac{C\delta_S}{(1+t)^2}+C\delta_Se^{-C\delta_St}e^{-C\delta_S\beta}+C[|(v-\overline{v})_{\xi}|^2]_{\xi=0}+ \frac{ C \delta_R^{1/4}\delta_{BL}^2}{(1+t)^{7/4}}+\frac{C\delta_R^{1/2}}{(1+t)^{3/2}}\\
&\quad +C\e_1^{2/3}\mathcal{Q}(t)+C\int_{\R_+}|\mathcal{R}|^2\,d\xi +\frac{C\delta_R^{1/2}}{\sqrt{1+t}}\|\mathcal{R}\|_{L^2}.
\end{align*}
Then, we integrate over $[0,t]$ to obtain
\begin{align*}
	&\int_{\R_+} Q(v|\overline{v})(t,\xi)\,d\xi+\frac{1}{2}\|(h-\overline{h})(t,\cdot)\|_{L^2}^2+\int_0^t D_{v_1}\,ds\\
	&\le C\left(\|(v-\overline{v})(0,\cdot)\|_{L^2}^2+\|(h-\overline{h})(0,\cdot)\|_{L^2}^2\right)+\frac{C\delta_S}{2}\int_0^t|\dot{X}(s)|^2\,ds\\
	&\quad +C\sqrt{\delta_S}\int_0^tG_1\,ds+C\int_0^t G_S\,ds+C\int_0^t G_v\,ds + C\int_0^t D_{u_1}\,ds+C\int_0^t [|(v-\overline{v})_\xi|^2]_{\xi=0}\,ds\\
	&\quad +C\int_0^t \|\mathcal{R}\|_{L^2}^2\,ds+C\delta^{1/2}_R\int_0^t\frac{\|\mathcal{R}\|_{L^2}}{\sqrt{1+s}}\,ds+C\e_1^{2/3}(\delta^{1/6}_R+\delta^{4/3}_S)+Ce^{-C\delta_S\beta}+C\delta_0^{1/4}.
\end{align*}
On the other hand, we use the equation
\begin{equation}\label{eq:vpert}
	(v-\overline{v})_t-\sigma_-(v-\overline{v})_\xi-(u-\overline{u})_\xi = \dot{X} v^S_\xi
\end{equation}
to get
\begin{align*}
\int_0^t[|(v-\overline{v})_\xi|^2]_{\xi=0}\,ds &\le C\int_0^t\left(\left[(v-\overline{v})_t^2\right]_{\xi=0}+|\dot{X}(s)|^2[|v^S_\xi|^2]_{\xi=0}+[(u-\overline{u})_{\xi}^2]_{\xi=0}\right)\,ds\\
&\leq C\int_0^t \left[(v-\overline{v})_t^2\right]_{\xi=0}\,ds + C\delta_S^3 e^{-C\delta_S \beta}+C\int_0^t \|(u-\overline{u})_{\xi}\|_{L^\infty}^2\,ds.
\end{align*}
Since
\[
[(v^S - v^*)_t]_{\xi = 0 } \leq C|(v^S)'(-(\sigma - \sigma_-)t - X(t) - \beta)||\dot{X} + (\sigma - \sigma_-)| \leq C\delta_S^2 e^{-C\delta_S((\sigma - \sigma_-)t + \beta)},
\]
we obtain 
\begin{equation*}
	\begin{aligned}
		\int_0^t \left[(v-\overline{v})_t|^2 \right]_{\xi = 0} \,ds &= \int_0^t \left[(v^S - v^*)_t|^2 \right]_{\xi = 0}\,ds \\
    &\leq C\delta_S^4 e^{-C\delta_S \beta} \int_0^t  e^{-C\delta_S s}\,ds \leq C\delta_S^3 e^{-C \delta_S \beta}.
	\end{aligned}
\end{equation*}
Therefore, we use the Gagliardo-Nirenberg interpolation inequality to get
\begin{equation}\label{est:bdv}
\begin{aligned}
	C&\int_0^t[|(v-\overline{v})_\xi|^2]_{\xi=0}\,ds\\
&\le C\delta_S^3e^{-C\delta_S\beta}+C\int_0^t\|(u-\overline{u})_\xi\|_{L^\infty}^2\,ds\\
&\le C\delta_S^3e^{-C\delta_S\beta}+C\int_0^t \|(u-\overline{u})_{\xi}\|_{L^2}^2\,ds+\frac{1}{10\max\{36v_*^2, 1\}C_2}\int_0^t \norm{\frac{1}{\sqrt{v}}(u-\overline{u})_{\xi\xi}}^2_{L^2}\,ds\\
&\le C\delta_S^3e^{-C\delta_S\beta}+C\int_0^t D_{u_1}\,ds+\frac{1}{10\max\{36v_*^2, 1\}C_2}\int_0^t D_{u_2}\,ds.
\end{aligned}
\end{equation}
The choice of $C_2$ will become clear later. Then, thanks to the above estimate and Lemma \ref{L2 wave interaction}, we conclude that
\begin{align}
\begin{aligned}\label{est_vh}
	&\frac{1}{2}\|(h-\overline{h})(t,\cdot)\|_{L^2}^2+\int_0^t D_{v_1}\,ds\\
	&\, \, \le C\left(\|(v-\overline{v})(0,\cdot)\|_{L^2}^2+\|(h-\overline{h})(0,\cdot)\|_{L^2}^2\right)+\frac{C\delta_S}{2}\int_0^t|\dot{X}(s)|^2\,ds+C\sqrt{\delta_S}\int_0^tG_1\,ds+C\int_0^t G_S\,ds\\
	&\,\, \quad +C\int_0^t G_v\,ds+C\int_0^t D_{u_1}\,ds+\frac{1}{10\max\{36v_*^2, 1\}C_2}\int_0^t D_{u_2}\,ds +Ce^{-C\delta_S\beta}+C\delta_0^{1/6}.
\end{aligned}
\end{align}
On the other hand, to derive the estimate on $(v-\overline{v})_\xi$, we use \eqref{h decompose} to get 
\begin{align*}
    |(v-\overline{v})_\xi| \leq v |h-\overline{h}| + C\big( |u-\overline{u}|  + |\overline{v}_\xi| |v-\overline{v}| + \mathcal{R} + |v^R_\xi|\big).
\end{align*}
Since
\begin{equation*} 
\frac{v_-}{3} < v < 3v_*,
\end{equation*}
for small enough $\e_1$, we obtain
\begin{align}
\begin{aligned}\label{v_x h change 1}
    \|(v-\overline{v})_\xi(t,\cdot)\|_{L^2}^2 &\leq 2(3v_*)^2 \|(h-\overline{h})(t,\cdot)\|_{L^2}^2 + C\left(\|(u-\overline{u})(t,\cdot)\|_{L^2}^2 + \|\mathcal{R}\|_{L^2}^2 + \|v^R_\xi\|_{L^2}^2 \right)\\
    &\le 18v_*^2 \|(h-\overline{h})(t,\cdot)\|_{L^2}^2 + C\|(u-\overline{u})(t,\cdot)\|_{L^2}^2 + C\delta_0^{1/4},
\end{aligned}
\end{align}
where we use Lemma \ref{rarefaction_properties} and Lemma \ref{L2 wave interaction}. Similarly, we have
\begin{equation}\label{v_x h change 2}
     \|(h-\overline{h})(0,\cdot)\|_{L^2}^2 \leq C \|(v-\overline{v})_\xi (0,\cdot)\|_{L^2}^2 + C\|(u-\overline{u})(0,\cdot)\|_{L^2}^2 + C\delta_0^{1/4}.
\end{equation}
Therefore, from Lemma \ref{lem:main}, \eqref{est_vh}, \eqref{v_x h change 1}, and \eqref{v_x h change 2}, we get
\begin{equation}\label{v_x estimate}
\begin{aligned}
    &\|(v-\overline{v})_\xi (t,\cdot)\|_{L^2}^2 + \int_0^t D_{v_1}\,ds\\
    &\quad \leq \max\{36v_*^2, 1\} \left(\frac{1}{2}\|(h-\overline{h})(t,\cdot)\|_{L^2}^2 + \int_0^t D_{v_1}\,ds\right) + C\|(u-\overline{u})(t,\cdot)\|_{L^2}^2 + C\delta_0^{1/6}\\
    &\quad \leq C\Big( \|(v-\overline{v})_\xi (0,\cdot)\|_{H^1}^2 +  \|(u-\overline{u}) (0,\cdot)\|_{L^2}^2  \Big)+ Ce^{-C\delta_S \beta} + C\delta_0^{1/6}\\
    &\qquad + C_1\left(\int_0^t (\sqrt{\delta_S}G_1 + G_S + G_v + D_{u_1})\,ds + \delta_S \int_0^t |\dot{X}|^2\,ds \right)  + \frac{1}{10C_2}\int_0^t D_{u_2}\,ds.
\end{aligned}
\end{equation}
Now, we multiply \eqref{v_x estimate} by $\frac{1}{2\max\{C_1, 1\}}$, add the result to \eqref{est_vu} together with the smallness of $\delta_0, \e_1$, and use Lemma \ref{lem:G_v} to have
\begin{equation*}
    \begin{aligned}
        &\frac{1}{2\max\{C_1, 1\}}\|(v-\overline{v})_\xi (t,\cdot)\|_{L^2}^2 + \|U - \overline{U}\|^2_{L^2}\\
        &\qquad + \frac{1}{2}\left(\delta_S \int_0^t |\dot{X}|^2\,ds + \int_0^t (G_S + G_1 + G_{BL} + G_R + D_{v_1} + D_{u_1})\,ds \right)\\
        &\quad \leq C\Big(\|(v-\overline{v}) (0,\cdot)\|_{H^1}^2 + \|(u-\overline{u}(0,\cdot)\|_{L^2}^2 \Big)+ Ce^{-C\delta_S \beta} + C\delta_0^{1/6} + \left(\frac{1}{20C_2}+C\e_1^2 \right)\int_0^t D_{u_2}\,ds.
    \end{aligned}
\end{equation*}
Using the smallness of $\e_1$, we obtain the desired $H^1$-estimate on $v-\overline{v}$ in Lemma \ref{lem:vhigh}.
\end{proof}

\subsection{$H^1$-estimate for $u-\overline{u}$}
Finally, we present $H^1$-estimate for $u-\overline{u}$ and thereby completing the proof of Proposition \ref{apriori-estimate}.

\begin{lemma}\label{lem:uhigh}
	Under the assumption of Proposition \ref{apriori-estimate}, we have
   \begin{align*}
    &\norm{(v-\overline{v})(t,\cdot)}_{H^1}^2 +  \norm{(u-\overline{u})(t,\cdot)}_{H^1}^2 + \delta_S \int_0^t |\dot{X}|^2\,ds \\
    &\qquad + \int_0^t (G_1 + G_S + G_R + G_{BL} + D_{v_1} + D_{u_1} + D_{u_2})\,ds\\
    &\quad \leq C\big( \norm{(v-\overline{v})(0,\cdot)}_{H^1}^2 +  \norm{(u-\overline{u})(0,\cdot)}_{H^1}^2 \big) + C\delta_0^{1/6} + Ce^{-C\delta_S \beta}.
\end{align*}
\end{lemma}
\begin{proof}
For simplicity, we define $\psi \coloneqq u - \overline{u}$. Then, from \eqref{eq:NS_inflow} and \eqref{eq:composite_wave}, $\psi$ satisfies 
\[
\psi_t - \sigma_- \psi_\xi - \dot{X}u^S_\xi + (p - \overline{p})_\xi = \left(\frac{u_\xi}{v} - \frac{\overline{u}_\xi}{\overline{v}}\right)_\xi - S.
\]
Multiplying the above equation by $-\psi_{\xi \xi}$ and integrating with respect to $\xi$, we get
\begin{align}
\begin{aligned}\label{psi_est}
    &\frac{d}{dt}\int_{\R_+} \frac{|\psi_\xi|^2}{2}\,d\xi + [\psi_t \psi_\xi ]_{\xi = 0} -\sigma_-  \left[\frac{|\psi_\xi|^2}{2}\right]_{\xi = 0}\\
    &\quad = -\dot{X} \int_{\R_+} u_\xi^S \psi_{\xi \xi} \,d\xi +  \int_{\R_+} (p - \overline{p})_\xi \psi_{\xi \xi} \,d\xi - \int_{\R_+} \left(\frac{u_\xi}{v} - \frac{\overline{u}_\xi}{\overline{v}}\right)_\xi \psi_{\xi \xi}\,d\xi + \int_{\R_+} S \psi_{\xi \xi}\,d\xi\\
    &\quad =: J_1 + J_2 + J_3 + J_4. 
\end{aligned}
\end{align}

\noindent $\bullet$ (Estimate of $J_1$): Using Cauchy-Schwarz inequality and \eqref{shock-properties}, we have
\[
|J_1| \leq |\dot{X}|\int_{\R_+} |u_\xi^S||\psi_{\xi \xi}|\,d\xi \leq C\delta_S^3|\dot{X}|^2 + \frac{1}{20} D_{u_2}.
\]
\noindent $\bullet$ (Estimate of $J_2$): We use Young's inequality to have
\[
|J_2| \leq \frac{1}{20}D_{u_2} +C D_{v_1}.
\]
\noindent $\bullet$ (Estimate of $J_3)$: We estimate $J_3$ as
\begin{align*}
J_3 &= -\int_{\R_+} \frac{1}{v} |\psi_{\xi\xi}|^2\,d\xi - \int_{\R_+} \left(\frac{1}{v} \right)_\xi \psi_\xi \psi_{\xi \xi}\,d\xi - \int_{\R_+} \overline{u}_{\xi\xi} \left(\frac{1}{v} - \frac{1}{\overline{v}}\right)\psi_{\xi\xi}\,d\xi\\
&\quad - \int_{\R_+} \overline{u}_\xi \left(\frac{1}{v} - \frac{1}{\overline{v}}\right)_\xi \psi_{\xi \xi}\,d\xi\\
&= -D_{u_2} + J_{31} + J_{32}+J_{33}.
\end{align*}
Using $\left(\frac{1}{v}\right)_\xi \leq C|v_\xi| \leq C\big(|(v-\overline{v})_\xi| + |\overline{v}_\xi|\big)$, a priori bound, and the Sobolev interpolation, we get
\begin{align*}
    |J_{31}| &\leq C \int_{\R_+} |(v-\overline{v})_\xi||\psi_{\xi}||\psi_{\xi \xi}|\,d\xi + C  \int_{\R_+} |\overline{v}_\xi||\psi_{\xi}||\psi_{\xi \xi}|\,d\xi  \\
    &\leq C\norm{(v-\overline{v})_\xi}_{L^2}\norm{\psi_\xi}_{L^\infty}\norm{\psi_{\xi\xi}}_{L^2} + C\norm{\overline{v}_\xi}_{L^\infty}\norm{\psi_\xi}_{L^2}\norm{\psi_{\xi\xi}}_{L^2}\\
    &\leq C \e_1 \norm{\psi_\xi}_{L^2}^{1/2}\norm{\psi_{\xi\xi}}_{L^2}^{3/2} + C\norm{\psi_\xi}_{L^2}\norm{\psi_{\xi\xi}}_{L^2}\\
	&\leq \frac{1}{20}D_{u_2} + C D_{u_1}.
\end{align*}
Using $|\overline{u}_{\xi \xi}| \leq C|\overline{u}_{\xi}|$, we have
\[
|J_{32}| \leq C \int_{\R_+} |\overline{u}_\xi||v-\overline{v}||\psi_{\xi \xi}|\,d\xi \leq \frac{1}{20}D_{u_2} + CG_v,
\]

and
\begin{align*}
    |J_{33}| &\leq C \int_{\R_+} |\overline{u}_\xi|\big(|v-\overline{v}| 
 + |(v-\overline{v})|_\xi \big)|\psi_{\xi \xi}|\,d\xi \leq \frac{1}{20}D_{u_2} + C \left(G_v + \int_{\R_+} |(v-\overline{v})_\xi|^2\,d\xi \right).
\end{align*}
Since
\[
(p - \overline{p})_\xi = p'(v)(v-\overline{v})_\xi + \overline{v}_\xi(p'(v) - p'(\overline{v})),
\]
we find that
\begin{equation*}
	\begin{aligned}
		\int_{\R_+} |(v-\overline{v})_\xi|^2\,d\xi &\leq  C \int_{\R_+} |(p-\overline{p})_\xi|^2\,d\xi +  C\int_{\R_+} |\overline{v}_\xi|^2 |v-\overline{v}|^2\,d\xi\\
		&\leq C D_{v_1} + CG_v.
   \end{aligned}
\end{equation*}

This implies that
\[
 |J_{33}| \leq \frac{1}{20}D_{u_2} + C \left(G_v  + D_{v_1} \right).
\]
\noindent $\bullet$ (Estimate of $J_4$): We use Young's inequality to have
\[
J_4 \leq \frac{1}{20}D_{u_2} + C\int_{\R_+} S^2\,d\xi.
\]

Next, we divide $\int_0^t \int_{\R_+} S^2\,d\xi\,ds$ into three parts:
\[
\int_0^t \int_{\R_+} S^2\,d\xi\,ds \leq C\left(\int_0^t \int_{\R_+} S_{I1}^2\,d\xi\,ds + \int_0^t\int_{\R_+} S_{I2}^2\,d\xi\,ds +\int_0^t \int_{\R_+} S_{R}^2\,d\xi\,ds\right),
\]
where $S_{I1}$, $S_{I2}$, and $S_R$ are defined in \eqref{SI1SI2}. First, we use \eqref{SI1 decomposition} and Lemma \ref{L2 wave interaction} to have
\begin{align*}
\int_0^t \int_{\R_+} S_{I1}^2\,d\xi\,ds &\leq \int_0^t \int_{\R_+} \mathcal{R}^2\,d\xi\,ds \leq C\delta_0^{1/4}.
\end{align*}
To estimate the $S_{I2}$ term, observe from \eqref{SI2_estimate} that
\[
|S_{I2}|^2 \leq C\left[|v_\xi^{BL}|^2|v_\xi^{R}|^2 + |v_\xi^{R}|^2|v_\xi^{S}|^2+|v_\xi^{S}|^2|v_\xi^{BL}|^2 \right].
\]
Then, we have
\begin{align*}
    &\int_0^t \int_{\R_+} |v_\xi^{BL}|^2|v_\xi^{R}|^2\,d\xi \,ds \leq C\delta^{1/4}_R\delta_{BL}^4 \int_0^t \frac{1}{(1+s)^{7/4}} \int_{\R_+} \frac{1}{(1+\delta_{BL}\xi)^4}\,d\xi\,ds \leq C\delta^{1/4}_R\delta_{BL}^3,\\
    &\int_0^t \int_{\R_+} |v_\xi^{R}|^2|v_\xi^{S}|^2\,d\xi\,ds\le \int_0^t \frac{1}{(1+s)^2}\norm{v_\xi^S}_{L^2}^2\,ds \leq C\delta_S^3,
\end{align*}
and
\begin{align*}
    \int_0^t \int_{\R_+} |v_\xi^{S}|^2|v_\xi^{BL}|^2\,d\xi\,ds &\leq  C\delta^4_S \int_0^t \left(\int_0^{bs}+\int_{bs}^{\infty}\right)e^{-C\delta_S|\xi-(\sigma-\sigma_-)s-X(s)-\beta|}\frac{\delta_{BL}^4}{(1+\delta_{BL}\xi)^4}\,d\xi\,ds\\
    & \le C\int_0^t  \left(\delta^3_S \delta_{BL}^4 e^{-C\delta_S s}+ \frac{\delta^4_S\delta_{BL}^3}{(1+\delta_{BL}s)^3}\right)\,ds\\
    & \leq C\delta^2_S \delta_{BL}^4 +C\delta^4_S \delta_{BL}^2,
\end{align*}
where $b = (\sigma - \sigma_-)/2$.

For $S_R$, we use \eqref{SR estimate} to have
\begin{align*}
    \int_0^t \int_{\R_+} S_{R}^2\,d\xi\,ds &\leq C \int_0^t \int_{\R_+} |u_{\xi \xi}^R|^2\,d\xi\,ds + C\int_0^t \int_{\R_+} |u_{\xi}^R|^2|v_{\xi}^R|^2\,d\xi\,ds\\
    &\leq C\left(\int_{1+t \leq \delta_R^{-1}} \delta_R^2\,ds + \int_{1+t \ge  \delta_R^{-1}} \frac{1}{(1+s)^2}\,ds \right) + C\int_0^t \norm{u_{\xi}^R}_{L^4}^4\,ds\\
    &\leq C\delta_R.
\end{align*}
Thus, we have
\[
\int_0^t \int_{\R_+} S^2\,d\xi\,ds \leq C\delta_0^{1/4}.
\]
Combining all the estimates for $J_i$ with \eqref{psi_est}, and integrating it over $[0,t]$ for any $t\leq T$, we obtain
\begin{equation}\label{est:uhigh}
	\begin{aligned}
		&\int_0^t \left(\frac{d}{ds}\int_{\R_+} \frac{|\psi_\xi|^2}{2}\,d\xi \right) \,dt +\frac{1}{2}\int_0^t D_{u_2}\,ds\\
		&\quad \leq  \int_0^t \left[- \psi_t \psi_\xi + \sigma_-   \frac{|\psi_\xi|^2}{2} \right]_{\xi=0}\,ds + C\int_0^t \left(\delta_S^3|\dot{X}|^2 + D_{u_1}+  D_{v_1} + G_v \right)\,ds + C\delta_0^{1/4}.\\
	\end{aligned}
\end{equation}
\noindent $\bullet$ (Estimate on the boundary terms): Now, we estimate the boundary term. From $\sigma_- <0$, we have
\begin{equation}\label{est:bdu1}
	\begin{aligned}
		\left[- \psi_t \psi_\xi + \sigma_-   \frac{|\psi_\xi|^2}{2} \right]_{\xi = 0} &\leq  \left[C|\psi_t|^2 -\frac{\sigma_-}{4}|\psi_\xi|^2 + \frac{\sigma_-}{2}|\psi_\xi|^2 \right]_{\xi = 0}\leq C\left[|\psi_t|^2 \right]_{\xi = 0}.
	\end{aligned}
\end{equation}
Therefore, it suffices to estimate $[|\psi_t|^2 ]_{\xi =0}$.
Since
\[
[|(u^S - u^*)_t|]_{\xi = 0 } \leq C|(u^S)'(-(\sigma - \sigma_-)t - X(t) - \beta)||\dot{X} + (\sigma - \sigma_-)| \leq C\delta_S^2 e^{-C\delta_S((\sigma - \sigma_-)t + \beta)},
\]
we obtain
\begin{equation}\label{est:bdu2}
	\int_0^t [|\psi_t|^2 ]_{\xi = 0} \,ds = \int_0^t [|(u^S - u^*)_t|^2 ]_{\xi = 0}\,ds \leq C\delta_S^4 e^{-C\delta_S \beta} \int_0^t  e^{-C\delta_S s}\,ds\leq C\delta_S^3 e^{-C \delta_S \beta}.
\end{equation}
Combining all these estimates and using Lemma \ref{lem:G_v}, and for sufficiently small $\delta_0 >0$, we obtain
\begin{align}
    \begin{aligned}\label{u_H1}
        	&\int_{\R_+} |(u-\overline{u})_\xi|^2(t,\xi)\,d\xi + \int_0^t D_{u_2}\,ds\\
    &\quad \leq \int_{\R_+} |(u-\overline{u})_\xi|^2(0,\xi)\,d\xi + C_2\int_0^t D_{v_1}\,ds\\
    &\qquad + C_3 \int_0^t \left(C\delta_S^3 |\dot{X}|^2 +(G_1 + G_S + G_R + G_{BL}+D_{u_1})\right)\,ds + C\delta_S^3e^{-C\delta_S\beta} + C\delta_0^{1/4},
    \end{aligned}
\end{align}
for some positive constants $C_2$ and $C_3$.
Multiplying \eqref{u_H1} by $\frac{1}{2\max(C_2, C_3, 1)}$, and adding the result to \eqref{higher v}, we obtain the desired estimate in Lemma \ref{lem:uhigh}.

\end{proof}

\begin{appendix}
    
 \section{Global existence of perturbed solution}\label{app:A}
\setcounter{equation}{0}

In this appendix, we present the proof of global existence of a solution to \eqref{eq:NS_inflow}. Assume that the positive constants $C_0$, $\delta_0$, $\e_1$, and $\beta$ from Proposition \ref{apriori-estimate} are given, where $\delta_0$ and $\e_1$ are chosen sufficiently small and $\beta$ is chosen sufficiently large, if necessary. Let $(\hat{v},\hat{u})(x)$ denote the smooth functions given in Proposition \ref{prop:local}, satisfying the following estimate:
		\begin{equation}\label{ubarvbar}
			\begin{aligned}
				\|(\hat{v}-v^*,\hat{u}-u^*)\|_{L^2(0,\beta)} +\|(\hat{v}-v_+,\hat{u}
				 -u_+)\|_{L^2(\beta,\infty)}+\|\pa_\xi (\hat{v},\hat{u})\|_{L^2(\R_+)} \le C_* \delta_0,
			\end{aligned}
				\end{equation}
        for some constant $C_*>0$. Define the smooth functions $(\underline{v}, \underline{u})$ by
		\[
			\underline{v} := v^{BL} + \hat{v} - v_*, \quad \underline{u} := u^{BL} + \hat{u} - u_*.
		\]
        
		Next, we estimate the perturbation between the smooth functions $(\underline{v},\underline{u})$ and the initial superposition $(\overline{v}(0,\cdot),\overline{u}(0,\cdot))$. By using Lemma \ref{lem:boundary_layer}, Lemma \ref{lem:rarefaction}, and Lemma \ref{lem:viscous_shock}, we have
	\begin{equation}
		\begin{aligned}\label{est-init}
			&\norm{\underline{v}(\cdot)-\overline{v}(0,\cdot)}_{H^1(\mathbb{R}_+)}+\norm{\underline{u}(\cdot)-\overline{u}(0,\cdot)}_{H^1(\mathbb{R}_+)}\\
			&\quad= \norm{v^R + v^S - \hat{v} - v^*}_{H^1(\mathbb{R}_+)}+ \norm{u^R + u^S - \hat{u} - u^*}_{H^1(\mathbb{R}_+)}\\
			&\quad \le \|(\hat{v}-v^*,\hat{u}-u^*)\|_{L^2(0,\beta)} +\|(\hat{v}-v_+,\hat{u}
			-u_+)\|_{L^2(\beta,\infty)}\\
		 &\qquad +\|(v^R-v^*,u^R-u^*)(0,\cdot)\|_{L^2(0,\beta)} + \|(v^S-v^*,u^S-u^*)(\cdot-\beta)\|_{L^2(0,\beta)} \\ 
		 &\qquad +\|(v^R-v^*,u^R-u^*)(0,\cdot)\|_{L^2(\beta,\infty)} + \|(v^S-v_+,u^S-u_+)(\cdot-\beta)\|_{L^2(\beta,\infty)} \\  
		 &\qquad +\|\pa_\xi (\hat{v},\hat{u})\|_{L^2(\mathbb{R}_+)} + \|\partial_\xi (v^R,u^R)\|_{L^2(\mathbb{R}_+)} + \|\partial_\xi (v^S,u^S)\|_{L^2(\mathbb{R}_+)} \\
		 &\quad \le C_{**} \sqrt{\delta_0} ,
			\end{aligned}
			\end{equation}
    for some constant $C_{**}>0$. 
    
    For a given $\e_1$, we choose $\delta_0$ sufficiently small and $\beta$ sufficiently large so that for any $\delta_{BL}, \delta_R, \delta_S < \delta_0$, the condition $\delta_S \beta \gg 1$ holds, and the inequality
 \[ \frac{\e_1}{4(C_0+1)}-C_{**} \sqrt{\delta_0}- C_* 
 \delta_0-\delta_0^{1/6}-e^{-C_0\delta_S \beta}>0\]
 is satisfied. Using this inequality, we define two positive constants $\varepsilon_*$ and $\varepsilon_0$ as
 \[\varepsilon_*:=\frac{\e_1}{2(C_0+1)}- C_{**} \sqrt{\delta_0} -\delta_0^{1/6}-e^{-C_0 \delta_S \beta} \quad \text{and} \quad \varepsilon_0:=\frac{\e_1}{4(C_0+1)}.\]

Now, consider initial data $(v_0,u_0)$ satisfying the following smallness assumption in Theorem~\ref{thm:main} 
\begin{equation} \label{v0u0}
\begin{aligned}
   &\|\big(v_0 - v^* - (v^{BL} - v_*), u_0 - u^*- (u^{BL} - u_*)\big)\|_{L^2(0,\beta)}\\ 
   &\quad + \|\big(v_0- v_+  - (v^{BL} - v_*),u_0 -u_+  - (u^{BL} - u_*)\big)\|_{L^2(\beta,\infty)} +\|\pa_\xi (v_0-v^{BL},u_0-u^{BL})\|_{L^2(\R_+)}\\
   & <\varepsilon_0.
\end{aligned}
\end{equation}
	Using the estimates \eqref{ubarvbar} and \eqref{v0u0}, the perturbation between the initial data $(v_0,u_0)$ and the functions $(\underline{v},\underline{u})$ is also small
	\begin{equation} \label{v0v_}
    \begin{aligned}
	\|(v_0-\underline{v},u_0-\underline{u})\|_{H^1(\mathbb{R}_+)}  \le \varepsilon_*,
 \end{aligned}
	\end{equation}
	which yields, by the Sobolev embedding and the smallness of $\varepsilon_*$,
	\[\frac{v_-}{2}\le v_0(x)\le 2v_+,\quad x\in\mathbb{R}_+.\]
	Then, we apply Proposition \ref{prop:local} to obtain the local existence of a unique solution $(v,u)$ on the time interval $[0,T_0]$ for some $T_0 > 0$, satisfying
	\begin{equation}\label{est-local-1}
	\|(v-\underline{v},u-\underline{u})\|_{L^\infty(0,T_0;H^1(\mathbb{R}_+))}
    \le \frac{\e_1}{2}
	\end{equation}
	and
	\[\frac{v_-}{3}\le v(t,x)\le 3v_+,\quad (t,x)\in[0,T_0]\times\mathbb{R}_+.\]

	On the other hand, we estimate the perturbation between $(\underline{v},\underline{u})$ and $(\overline{v},\overline{u})$ as in \eqref{est-init}:
	\begin{align*}
	&\|(\underline{v}-\overline{v},\underline{u}-\overline{u})(t,\cdot)\|_{H^1(\mathbb{R}_+)}\\
    &\quad \le  \|(\hat{v}-v^*,\hat{u}-u^*)\|_{L^2(0,\beta)} +\|(\hat{v}-v_+,\hat{u}
	-u_+)\|_{L^2(\beta,\infty)}\\ 
 &\qquad +\|(v^R-v^*,u^R-u^*)(t,\cdot )\|_{L^2(0,\beta)} + \|(v^S-v^*,u^S-u^*)(\cdot-(\sigma-\sigma_-)t-X(t)-\beta)\|_{L^2(0,\beta)} \\ 
 &\qquad +\|(v^R-v^*,u^R-u^*)(t,\cdot )\|_{L^2(\beta,\infty)} + \|(v^S-v_+,u^S-u_+)(\cdot-(\sigma-\sigma_-)t-X(t)-\beta)\|_{L^2(\beta,\infty)} \\   
 &\qquad +\|\pa_\xi (\hat{v},\hat{u})\|_{L^2(\mathbb{R}_+)} + \|\partial_\xi (v^R,u^R)\|_{L^2(\mathbb{R}_+)} + \|\partial_\xi (v^S,u^S)\|_{L^2(\mathbb{R}_+)} \\ 
&\quad \le C\delta_R \sqrt{1+ |1+\sigma_-t| } + C\sqrt{\delta_S} \sqrt{(1+|(\sigma-\sigma_-)|t+|X(t)|)} \le C\sqrt{\delta_0}(1+\sqrt{t}).
	\end{align*}
We now take $T_1\in(0,T_0)$ small enough so that $C\sqrt{\delta_0}(1+\sqrt{T_1})\le \frac{\e_1}{2}$, and therefore
	\begin{equation}\label{est-local-2}
	\|(\underline{v}-\overline{v},\underline{u}-\overline{u})\|_{L^\infty(0,T_1;H^1(\mathbb{R}_+))} \le\frac{\e_1}{2}.
	\end{equation}

	Combining \eqref{est-local-1} and \eqref{est-local-2}, we observe that the smallness assumption in Proposition \ref{apriori-estimate} is satisfied on $[0,T_1]$:
	\[\|(v-\overline{v},u-\overline{u})\|_{L^\infty(0,T_1;H^1(\mathbb{R}_+))} \le\e_1.\]
 In particular, since $X(t)$ is absolutely continuous and $(v-\underline{v},u-\underline{u}) \in C([0,T_1];H^1(\mathbb{R}_+))$, it follows that
 \[(v-\overline{v},u-\overline{u}) \in C([0,T_1];H^1(\mathbb{R}_+)).\]

To attain global existence, we use a standard continuation argument. Suppose the maximal time
 \[T_M:=\sup \left\{ t>0 : \sup_{s \in [0,t]}  \| (v-\overline{v},u-\overline{u})(s,\cdot) \|_{H^1(\mathbb{R}_+)}  \le \e_1 \right\}\]
 is finite. Then, the continuity argument  implies that
 \[\sup_{s \in [0,T_M]}  \| (v-\overline{v},u-\overline{u})(s,\cdot) \|_{H^1(\mathbb{R}_+)}  =\e_1.\]
 However, using \eqref{est-init} and \eqref{v0v_}, we can estimate the initial perturbation as 
 \begin{align*}
     &\| (v-\overline{v},u-\overline{u})(0,\cdot) \|_{H^1(\mathbb{R}_+)} 
     \le  \frac{\e_1}{2(C_0+1)} - \delta_0^{1/6} -  e^{-C_0\delta_S \beta}.
 \end{align*}
Using this estimate and applying Proposition \ref{apriori-estimate}, we obtain
 \begin{align*}
     \sup_{s\in[0,T_M]}\norm{(v-\overline{v},u-\overline{u})(t,\cdot)}_{H^1 (\mathbb{R}_+)}  &\le  \frac{\e_1}{2},
 \end{align*}
	which is a contradiction. Hence, $T_M=+\infty$ and the solution exists globally. 
    
    Consequently, we obtain the following global estimate
	\begin{equation}
	\begin{aligned}\label{est-infinite}
	& \sup_{t>0}\norm{(v-\overline{v},u-\overline{u})(t,\cdot)}_{H^1 (\mathbb{R}_+)}^2 +\delta_S \int_0^\infty | \dot{X}(t)|^2 \, d t  \\ 
    &\quad +\int_0^\infty \left( \mathcal{G}_1+\mathcal{G}_2 + \mathcal{D}_{v} + \mathcal{D}_{u_1} +\mathcal{D}_{u_2} \right)  \, dt  \\ 
	& \le C_0 \norm{(v-\overline{v},u-\overline{u})(0,\cdot)}_{H^1 (\mathbb{R}_+)}^2  + C_0 \delta_0^{1/6}+C_0e^{-C_0\delta_S \beta}<\infty
	\end{aligned}
	\end{equation}
	and, for $t>0$,
	\begin{equation} \label{eq: X bound}
	|\dot{X}(t)|\le C_0 \left(\|(v-\overline{v})(t,\cdot)\|_{L^\infty(\mathbb{R}_+)}+\|(u-\overline{u})(t,\cdot)\|_{L^\infty(\mathbb{R}_+)}\right).
	\end{equation}

	\section{Time-asymptotic behavior}\label{app:B}
	\setcounter{equation}{0}
	
	In this appendix, we present the asymptotic convergence towards the superposition. Define
	\[g(t):=\|(v-\overline{v})_\xi\|_{L^2 (\mathbb{R}_+)}^2+\|(u-\overline{u})_\xi\|_{L^2(\mathbb{R}_+)}^2.\]
	We will show that $g\in W^{1,1}(\R_+)$, which implies $\lim_{t\to\infty}g(t)=0$. Once we show it, we apply  Gagliardo-Nirenberg interpolation inequality alongside the uniform bound in \eqref{est-infinite} to get
	\begin{equation} \label{eq: v,u limit}
    \begin{aligned}
        \|(v-\overline{v},u-\overline{u})\|_{L^\infty(\mathbb{R}_+)} \le \|(v-\overline{v},u-\overline{u})\|_{L^2(\mathbb{R}_+)}^{1/2} \|(v-\overline{v},u-\overline{u})_\xi\|_{L^2(\mathbb{R}_+)}^{1/2} \to 0, \quad \text{as} \quad t \to \infty.
    \end{aligned}
	\end{equation}
	Furthermore, \eqref{eq: X bound} and \eqref{eq: v,u limit} imply
	\[ \lim_{t \to \infty} |\dot{X}(t)| \le C_0 \lim_{t \to \infty}  \norm{(v-\overline{v},u-\overline{u})(t,\cdot)}_{L^\infty(\mathbb{R}_+)} =0.\]
	Therefore, it remains to show that $g\in W^{1,1}(\R_+)$.\\

	\noindent (1) $g\in L^1(\bbr_+)$: From \eqref{est-infinite}, we have 

	\begin{align*}
	\int_0^\infty |g(t)|\,dt
	&\le C\int_0^\infty (\mathcal{D}_v+\mathcal{D}_{u_1}+\mathcal{G}_2)\,d t<\infty,
	\end{align*}
	
	\noindent (2) $g'\in L^1(\R_+)$: 
	First, note that 
	\[
		\int_0^\infty |g'(t)|dt = \int_0^\infty \left|2\int_{\R_+} (v-\bar{v})_{\xi t}(v-\bar{v})_{\xi} + \frac{d}{dt}\int_{\R_+} |(u-\bar{u})_\xi|^2d\xi \right|dt.
	\]

	Using \eqref{eq:vpert}, integration by parts, Young's inequality and \eqref{est:bdv}, we have
	\begin{align*}
		\int_0^\infty \left|\int_{\R_+} (v-\bar{v})_{\xi t}(v-\bar{v})_{\xi} d\xi\right|\,dt &= \int_0^\infty \left|\int_{\mathbb{R}_+} \left(  \sigma_- (v-\overline{v})_{\xi\xi } +
		(u-\overline{u})_{\xi\xi}+\dot{X}(t)v^S_{\xi\xi}\right)(v-\overline{v})_\xi\,d \xi\right|\,dt\\
		&\le C\int_0^\infty \Big[|(v-\bar{v})_{\xi}|^2\Big]_{\xi=0} + C\int_0^\infty \big(\mathcal{D}_{v} + \mathcal{D}_{u_2} + \delta_S |\dot{X}|^2\big)\,dt\\
		&\le C+ C\int_0^\infty \big(\mathcal{D}_{v} + \mathcal{D}_{u_1} + \mathcal{D}_{u_2} + \delta_S |\dot{X}|^2\big)\,dt < +\infty.
	\end{align*}
	Finally, following the calculations in deriving \eqref{est:uhigh}, and using \eqref{est:bdu1} and \eqref{est:bdu2}, we obtain 
	\begin{align*}
		\int_0^\infty \left|\frac{d}{dt}\int_{\R_+} |(u-\bar{u})_\xi|^2d\xi \right|dt  &\le  C\delta_0^{1/4}+ \int_0^\infty \left| \left[- (u-\bar{u})_t (u-\bar{u})_\xi + \sigma_-   \frac{|(u-\bar{u})_\xi|^2}{2} \right]_{\xi=0} \right|\,dt\\
		&\quad + C\int_0^\infty \left(\mathcal{G}_2 + \mathcal{D}_{v} + \mathcal{D}_{u_1} + \mathcal{D}_{u_2} + \delta_S|\dot{X}|^2 \right)\, dt < +\infty.
	\end{align*}  
    Combining all these estimates, we conclude that  $g\in W^{1,1}(\R_+)$ and complete the proof of the asymptotic behavior in Theorem \ref{thm:main}.
\end{appendix}

\noindent \textbf{Statements and Declarations}\\

\noindent \textbf{Data availability:} No datasets were generated or analyzed during the current study. \vspace{0.3cm}

\noindent \textbf{Conflict of interest:} The authors declare that they have no conflicts of interest with this work.

\bibliographystyle{amsplain}
\bibliography{reference} 

\end{document}